\documentclass[19pt]{aims}
\usepackage{amsmath, amssymb, mathrsfs}
  \usepackage{paralist}
  \usepackage{graphics} 
  \usepackage{epsfig} 
\usepackage{graphicx} 
 \usepackage{epstopdf}
 \usepackage[colorlinks=true]{hyperref}
 \hypersetup{urlcolor=blue, citecolor=blue}
 \usepackage[toc,page]{appendix}
\usepackage{chngcntr}
\usepackage{hyperref}

\usepackage{titlesec}
\titleformat{\section}{\Large\bfseries}{\thesection.}{4pt}{}
\titleformat{\subsection}{\large\bfseries}{\thesection.\arabic{subsection}.}{4pt}{}
\titleformat{\subsubsection}{\bfseries}{\thesection.\arabic{subsection}.\arabic{subsubsection}.}{4pt}{}
\titleformat*{\paragraph}{\bfseries}
\titleformat*{\subparagraph}{\bfseries}
\usepackage[margin=1in]{geometry}

  \def\currentyear{}
   \def\currentmonth{}
    \def\ppages{}
     \def\DOI{}

\newtheorem{theorem}{Theorem}[section]
\newtheorem{corollary}[theorem]{Corollary}

\newtheorem{lemma}[theorem]{Lemma}
\newtheorem{proposition}[theorem]{Proposition}
\theoremstyle{definition}
\newtheorem{definition}[theorem]{Definition}
\newtheorem{remark}[theorem]{Remark}

\numberwithin{equation}{section}

\title{ A blowup solution  of  a complex semi-linear  Heat
 equation  with an irrational power}

\author[G. K. Duong]{}

\thanks{\today}
\begin{document}
\maketitle

\centerline{Giao Ky Duong \footnote{ G. K. Duong  is supported   by the project INSPIRE. This project has received funding from the European Union’s Horizon 2020 research and innovation programme under the Marie Sk\l odowska-Curie grant agreement No 665850.}  }
\medskip
{\footnotesize
  \centerline{ Universit\'e Paris 13, Sorbonne Paris Cit\'e, LAGA,  CNRS (UMR 7539), F-93430, Villetaneuse, France.} 
}

\bigskip
\begin{center}\thanks{\today}\end{center}

\begin{abstract} 
 In  this  paper, we   consider  the following semi-linear complex  heat equation
\begin{eqnarray*}
\partial_t u =  \Delta u + u^p, u \in \mathbb{C}
\end{eqnarray*}
in $\mathbb{R}^n,$ with an arbitrary   power $p,$  $ p  > 1$. In particular, $p$  can be non integer and  even irrational, unlike  our previous work \cite{DU2017},   dedicated to the integer case. 
We  construct for this equation a complex solution $u = u_1 + i u_2$, which blows up in finite 
  time $T$ and only at one blowup point $a.$  Moreover,  we    also  describe the asymptotics  of the solution    by   the following  final profiles: 
\begin{eqnarray*}
u(x,T)  &\sim &  \left[ \frac{(p-1)^2 |x-a|^2}{ 8 p |\ln|x-a||}\right]^{-\frac{1}{p-1}},\\
u_2(x,T)  &\sim &   \frac{2 p}{(p-1)^2}  \left[ \frac{ (p-1)^2|x-a|^2}{ 8p |\ln|x-a||}\right]^{-\frac{1}{p-1}}\frac{1}{ |\ln|x-a||}  > 0 , \text{ as } x \to a.
\end{eqnarray*}
 In addition to that,  since  we also   have    $u_1 (0,t)  \to  + \infty$ and  $u_2(0,t) \to - \infty$ as  $t \to T,$ the blowup  in the  imaginary part shows a new phenomenon unkown for the  standard heat equation in the real case: a non constant sign near the  singularity, with the existence  of  a vanishing  surface  for  the    imaginary part, shrinking  to the origin.  In our  work,    we    have succeeded to extend   for   any   power $p$ where the non linear term $u^p$ is not continuous  if   $  p$  is not integer. In particular, the   solution which  we have      constructed  has  a  positive real part.
We  study   our equation    as   a  system  of the real part and the imaginary part  $u_1$ and  $u_2$.   Our  work relies on  two main    arguments:  the reduction of  the  problem  to a  finite  dimensional one  and a topological argument based
   on the index theory to get the conclusion.  
 \end{abstract}

\maketitle
\section{Introduction}
\subsection{Ealier work}

In this work,   we are interested in the following complex-valued semilinear  heat equation
\begin{equation}\label{equ:problem}
\left\{
\begin{array}{rcl}
\partial_t u &=& \Delta u +  F(u), t \in [0,T), \\[0.2cm]
u(0) &=& u_0 \in L^\infty \text{ with Re}(u_0) \geq \lambda > 0, 
\end{array}
\right.
\end{equation}
where  $F(u) = u^p$ and $u(t): \mathbb{R}^n \to \mathbb{C}$,  $L^\infty := L^\infty(\mathbb{R}^n, \mathbb{C})$,  $p   > 1$.  

\noindent
Typically,   when  $p =2$,    model \eqref{equ:problem}    becomes the following
\begin{equation}\label{equation-problem-p=2}
\left\{
\begin{array}{rcl}
\partial_t u &=& \Delta u +  u^2, t \in [0,T), \\[0.2cm]
u(0) &=& u_0 \in L^\infty.
\end{array}
\right.
\end{equation}
This  model    is connected   to the viscous Constantin-Lax-Majda equation with a viscosity term,
 which is a one dimensional model for the  vorticity  equation in fluids.  For more details, 
 the readers  are addressed to  the following  works:  Constantin, Lax, Majda \cite{CLMCPAM1985}, Guo, Ninomiya and  Yanagida 
 in \cite{GNYTAMS2013},   Okamoto, Sakajo and Wunsch  \cite{OSWnon2008}, Sakajo in  \cite{SMAthScitokyo2003} and \cite{Snon2003}, Schochet \cite{SCPAM1986}.  In \cite{DU2017}, we treated   the case   $ p \in \mathbb{N}$. Indeed,  handling   the nonlinear term  in this case is much  easier.  In the present paper,  we do better, and give a proof  which   holds  also   in the case  $p\notin  \mathbb{N}$.  The  local  Cauchy problem   for model \eqref{equ:problem}  can be solved   in $L^{\infty}(\mathbb{R}^n,\mathbb{C})$  when     $p$ is integer,   thanks to a  fixed-point argument. However,   if   $p $ is not an  integer number,   then, the  local  Cauchy  problem has not   been  solved yet, up to our knowledge. In my point of view,  this  probably  comes from   the discontinuity of    $F(u) $  on $\{u   \in \mathbb{R}^*_-\}$ and this challenge is also one of the  main difficulties of  the paper. As  a matter  of fact,  we  solve the Cauchy  problem  in Appendix \ref{appendix-Cauchy-problem}  for data  $u_0 \in L^\infty,$ with $\text{Re}(u_0) \geq  \lambda,$ for some   $\lambda > 0$.   Accordingly, a maximal  solution may be global  in time  or may  exist only  for $t \in [0, T),$  for some  $T > 0.$  In that case, we   have to options: 
 
 \begin{itemize}
\item[$(i)$]   Either  $\| u(t)\|_{L^\infty(\mathbb{R}^n)}   \to + \infty   $  as  $t \to  T$.  
\item[$(ii)$]  Or  $\min_{x \in \mathbb{R}^n} \text{Re} (u(x,t))    \to  0$ as  $t \to T$. 
 \end{itemize}
In this paper, we are    interested   in the  case  $(i),$ which   is  referred  to as   \textit{finite-time blow-up} in the  sequel.   
  
  \noindent
  A blowup  solution $u$    is called   \textit{Type I}  if 
$$ \limsup_{t \to T}  (T-t)^{\frac{1}{p-1}} \| u(.,t)\|_{L^\infty}     < + \infty. $$
Otherwise,  the solution $ u$  is called  \textit{Type II}.

In addtion to that,   $T$  is called the bolwup time of $u$  and   a point $a \in \mathbb{R}^n$ is called a  blowup point if and only if 
there exists a sequence $\{(a_j, t_j)\} \to (a,T)$ as $j \to +\infty$ such that 
$$ |u_1(a_j,t_j)| + |u_2(a_j,t_j)| \to +\infty \text{ as } j \to +\infty.$$
 In our work,  we are interested    in      constructing   a blowup solution  of \eqref{equ:problem} which  is   \textit{Type I}.  Let us quickly     mention   some  typical works  for this situation  (for more   details,   see  the introduction  of \cite{DU2017}, treated   the integer case).
 
\bigskip
\textbf{ $(i)$ For the   real case:}     Bricmont and
 Kupiainen \cite{BKnon94}  constructed a real positive solution  to the following equation
 \begin{equation}\label{equa-semi-|u|-p-1-u}
 \partial_t u =  \Delta u +  |u|^{p-1 } u, p > 1,
 \end{equation}
  which blows up in finite time $T$, only at the origin and they  have  derived    the  profile of  the solution such that
 $$  \left\|  (T-t)^{\frac{1}{p-1}}  u (.,t)  - f_0 \left( \frac{.}{\sqrt{(T-t)|\ln (T - t)|} } \right) \right\|_{L^{\infty}(\mathbb{R}^n)} \leq \frac{C}{ 1 + \sqrt{|\ln (T-t)|}},$$
 where the profile $f_0 $ is defined  as follows
 \begin{equation}\label{defini-f-0}
 f_0 (x) = \left( p-1 + \frac{(p-1)^2 |x|^2}{4p}\right)^{-\frac{1}{p-1}}.
 \end{equation}
 
 \bigskip
 \noindent In addition to  that,  in \cite{HVcpde92},    Herrero and  Vel\'azquez    derived    the   same  result  with a  different method.  Particularly,   in \cite{MZdm97},   Merle and Zaag gave   a   proof  which is   simpler  than the one in  \cite{BKnon94} and    proposed the following two-step method (see also the note   \cite{MZasp96}):
    \begin{itemize}
    \item[-] Reduction of the  infinite dimensional problem to a finite dimensional one.
    \item[-]   Solution of the finite   dimensional problem thanks to a topological  argument based  on  Index theory.
    \end{itemize}
Moreover,  they  also  proved  the stability of the blowup profile  for   \eqref{equa-semi-|u|-p-1-u}. In addition to  that,   we would like to mention that this   method has  been   successful  in  various  situations  such as the work of Ebde and Zaag  \cite{EZsema11},  Tayachi and Zaag \cite{TZpre15}, 	and also  the works of Ghoul, Nguyen and Zaag in \cite{GNZsysparabolic2016},     \cite{GNZpre16a} (with a gradient term) and  \cite{GNZpre16b}.  We would like to  mention  also  the work of     Nguyen and Zaag in \cite{NZens16}, who considered   the following quasi-critical double  source   equation  
$$ \partial_t u  = \Delta u  + | u|^{p-1} u + \frac{\mu  |u|^{p - 1} u }{ \ln ^a (2 + u^2)} ,$$  
and also  the work  of  Duong, Nguyen and Zaag  in \cite{DNZtunisian-2017}, who considered  the following   non scale  invariant  equation 
$$ \partial_t u =  \Delta u  + |u|^{p - 1} u \ln ^\alpha ( 2 + u^2).$$


\medskip
\textbf{    $(ii)$  For the complex case:}     	The blowup    problem for the complex-valued parabolic equations has been  studied  intensively by many authors, in particular  for  the Complex    Ginzburg Landau (CGL) equation 
\begin{equation}\label{ginzburg-landau}
 \partial_t u =   (1 + i \beta ) \Delta u   + (1  + i \delta) | u|^{p-1} u  + \gamma u. 
\end{equation}
 This is the case of an   ealier work of  Zaag  in \cite{ZAAihn98}    for  equation \eqref{ginzburg-landau} when     $\beta = 0$ and $ \delta$ small enough. Later,  Masmoudi and Zaag  in  \cite{MZjfa08}  generalized the result of  \cite{ZAAihn98} and  constructed  a blowup solution  for \eqref{ginzburg-landau}  with a  super critical condition   $ p - \delta^2 -  \beta \delta   -  \beta  \delta p   >0.$
Recently,   Nouaili and Zaag  in \cite{NZ2017}  has constructed  a blowup solution   for \eqref{ginzburg-landau},       in   the critical   case  where  $\beta = 0$ and $ p = \delta^2 $).

\bigskip
\noindent
 In addtiion to that,  there are many works  for  equation \eqref{equ:problem}   or  \eqref{equation-problem-p=2},  such as the work  of Nouaili and Zaag  in \cite{NZCPDE2015} for equation  \eqref{equation-problem-p=2},  who    constructed       a complex solution $ u = u_1+ i u_2 ,$  which blows up in finite time $T$  only at the  origin.  Note that the real and the imaginary parts  blow up simultaneously.  Note also that \cite{NZCPDE2015} leaves unanswered the question of the derivation of  the profile of the imaginary  part, and this is precisely  our aim in this paper, not only for equation \eqref{equation-problem-p=2}, but also for equation  \eqref{equ:problem} with  $ p  > 1$. We  would   like to mention also some  classification results, proven  by Harada in \cite{HaradaJFA2016},  for blowup solutions of \eqref{equation-problem-p=2}  which satisfy    some reasonable assumptions. In   particular,   in that works,   we are able to derive  a sharp blowup profile for  the imaginary part of the solution. In 2018,  in \cite{DU2017},    we   handled equation \eqref{equ:problem} when  $p$ is  an integer.
 
 \subsection{Statement of  the  result  }
 
  Although, in \cite{DU2017},  we believe   we  made  an important   achievement, we acknowledge that  we   left  unanswered  the case where  $p >1$ and  $p \notin \mathbb{N}$.   From the  limitation   of the above works, it   motivates us  to  study model \eqref{equ:problem} in general even for   irrational  $p$.  The following theorem is considered as  a  generalization of \cite{DU2017}  for all  $p > 1$. 

\begin{theorem}[Existence of  a  blowup solution for \eqref{equ:problem} and a sharp discription of its profile]\label{Theorem-profile-complex}    For each  $p  > 1$ and $  p_1 \in  \left(  0, \min\left( \frac{p-1}{4}, \frac{1}{2}\right)\right) $,  there exists $T_1 (p,p_1)> 0 $ such that for all $T \leq T_1,$ there exist   initial data  $u(0)  = u_1(0)  + i u_2(0),$   such that  equation \eqref{equ:problem} has  a  unique solution $u $ on  $\mathbb{R}^n \times [0,T)$  satisfying  the following: 
\begin{itemize}
\item[$i)$] The solution $u$ blows up in finite time $T$   only at the  origin and  $\text{Re} (u) > 0$ on  $\mathbb{R}^n \times [0,T)$.  Moreover, it  satisfies  the following 
\begin{equation}\label{esttima-theorem-profile-complex}
\left\| (T- t)^{\frac{1}{p-1}} u (.,t) - f_0 \left( \frac{.}{\sqrt{(T -t) |\ln(T -t)|}}  \right)  \right\|_{L^{\infty}(\mathbb{R}^n)}  \leq \frac{C}{1 +   \sqrt{|\ln (T -t)|}},
\end{equation}
and
\begin{equation}\label{estima-the-imginary-part}
\left\| (T- t)^{\frac{1}{p-1}} |\ln (T-t)|  u_2(.,t) - g_0 \left( \frac{.}{\sqrt{(T -t) |\ln(T -t)|}}  \right)  \right\|_{L^{\infty}(\mathbb{R}^n)}  \leq \frac{C}{1 +  |\ln (T -t)|^{\frac{p_1}{2}}},
\end{equation}
where $f_0 $ is defined in \eqref{defini-f-0} and  $g_0(x)$ is defined as follows
\begin{eqnarray}
\displaystyle g_0(x) &=& \frac{|x|^2}{\left(p-1  + \frac{(p-1)^2}{4p} |x|^2 \right)^{\frac{p}{p-1}}}\label{defini-g-0-z}.
\end{eqnarray} 
\item[$ii)$] There exists  a complex   function 
 $u^*$      in  $  C^{2}(\mathbb{R}^n \backslash \{0\})$    such that
  $u(t) \to u^* = u_1^* + i u_2^*$  as  $t \to T,$ uniformly on compact sets of 
  $\mathbb{R}^n \backslash \{0\},$ and  we have the following asymptotic expansions:
\begin{equation}\label{asymp-u-start-near-0-profile-complex}
u^*(x) \sim \left[ \frac{(p-1)^2 |x|^2}{ 8 p |\ln|x||}\right]^{-\frac{1}{p-1}}, \text{ as } x \to 0,
\end{equation}
and  
\begin{equation}\label{asymp-u-start-near-0-profile-complex-imaginary-part}
 u^*_2(x)  \sim \frac{2 p}{(p-1)^2}  \left[ \frac{ (p-1)^2|x|^2}{ 8p |\ln|x||}\right]^{-\frac{1}{p-1}}\frac{1}{ |\ln|x||} , \text{ as } x \to 0.
\end{equation}
\end{itemize}
\end{theorem}

\begin{remark}
We remark    that the    condition  on the   parameter  $p_1  < \min \left(  \frac{p-1}{4}, \frac{1}{2} \right)$ comes  from  the  definition of the set $V_A (s)$ (see in  item  $(i)$ of Definition \ref{defini-shrinking-set}),  Proposition \ref{prop-dynamic-q-1-2-alpha-beta} and Lemma \ref{lemma-quadratic-term-B-1-2}. Indeed,   this condition  ensures  that   the projections of   the quadratic term $B_2(q_2,q_2)$     on the  negative   and  outer parts  are smaller than  the  conditions  in $V_A(s)$. Then, we can conclude    \eqref{Estimat-q-2--} and \eqref{outer-Phi-e} by using Lemma  \ref{lemma-quadratic-term-B-1-2} and definition of   $V_A(s)$.
\end{remark}
\begin{remark}    We can show   that the constructed  solution in the above Theorem    satisfies  the following asymptotics:
\begin{eqnarray}  
 u(0,t)  & \sim &  \kappa (T -t)^{-\frac{1}{p-1}},\\
 u_2(0,t) &\sim  &     - \frac{2n \kappa}{ (p-1)  }  \frac{(T-t)^{-\frac{1}{p-1}}}{ |\ln(T-t)|^2},\label{limit-u-2-0-t-t-to-T}
 \end{eqnarray}
 as  $t \to T$, (see  \eqref{u-0-t-sim} and  \eqref{u-2-0-t-sim}).  Therefore,  we  deduce  that $u$  blows up at time    $T$    only at $0$.  Note that,   the real and imaginary  parts simultaneously blow up.  Moreover,  from item  $(ii)$   of  Theorem  \ref{Theorem-profile-complex},   the blowup speed of $u_2$ is   softer  than $u_1$ because of  the quantity  $ \frac{1}{|\ln |x||}$ (see \eqref{asymp-u-start-near-0-profile-complex}  and  \eqref{asymp-u-start-near-0-profile-complex-imaginary-part}).  
\end{remark}
\begin{remark}[A strong singularity  of the imaginary part]  We observe from \eqref{asymp-u-start-near-0-profile-complex-imaginary-part} and      \eqref{limit-u-2-0-t-t-to-T} that there is a strong sigularity at   the  neighborhood of  $a$  as $t \to T;$  when  $x$  close  to $0,$  we have  $u_2(x,t) $ which  becomes  large and   positive  as $t \to T$,    however,   we  always  have $u_2(0,t) \to - \infty$ as $t \to T.$  Thus  the imaginary part has  no constant    sign  near   the singularity.   In particular,    if   $t$ is near   $T$,     there exists $b(t) > 0$ in  $\mathbb{R}^n$ and $b(t) \to 0$ as $t \to T$ such that  at time $t,$  $u_2(.,t)$ vanishes on some  surface close to the sphere   of center  $0$  and radius  $b(t)$.  Therefore, we don't have    $|u_2(x,t) | \to + \infty  $ as  $(x,t) \to (0,T)$. This  non constant property  for the imaginary  part  is  very surprising to us. In the frame work of  semilinear heat equation, such  a property  can  be encountered  for  phase  invariant  complex  equations,  such as    the Complex   Ginzburg-Landau (CGL) equation (see    Zaag  in \cite{ZAAihn98},  Masmoudi and Zaag  in \cite{MZjfa08},  Nouaili-Zaag \cite{NZ2017}). As for complex   parabolic  equation  with no phase  invariance, this  is the first time   such a sign change   in available, up to our knowledge. We would like to mention that such a  sign change near  the singularity  was already  observed  for the   semilinear wave equation non  characteristic blowup point  (see Merle and Zaag   in \cite{MZajm12}, \cite{MZdm12}) and   C\^ote and Zaag  in \cite{CZcpam13}.
\end{remark}

\begin{remark}
For each      $a \in \mathbb{R}^n,$   by using   the translation $u_a(.,t)  = u (. - a,t),$ we can prove  that   $u_a$ also  satisfies    equation   \eqref{equ:problem}  and   the solution    blows up  at time   $T$   only   at the  point $a$. We can derive that    $u_a$  satisfies  all   estimates    
\eqref{esttima-theorem-profile-complex} -  \eqref{asymp-u-start-near-0-profile-complex-imaginary-part}   by  replacing   $x$ by   $x- a$.
\end{remark}

\begin{remark}\label{remark-initial-data} In Theorem  \eqref{Theorem-profile-complex}, the initial data  $u (0) $ is given exactly as follows 
$$ u (0)  = u_1 (0)+ i u_2(0) ,$$
where 
\begin{eqnarray*}
& &u_1 (x, 0)  = T^{ - \frac{1}{p -1}} \left\{  \left( p-1 + \frac{(p-1)^2|x|^2}{ 4 p T |\ln T|}   \right)^{-\frac{1}{p-1}}     + \frac{n \kappa }{2 p |\ln T|} \right.\\
& +& \left. \frac{A}{|\ln T|^2} \left(  d_{1,0} + d_{1,1} \cdot \frac{x}{\sqrt T}  \right) \chi_0 \left(\frac{16|x|}{K_0 \sqrt{T |\ln T| }} \right) \right\} \chi_0\left(\frac{|x|}{\sqrt{T} |\ln T|} \right)    +  U^*(x) \left(1 - \chi_0\left(\frac{|x|}{\sqrt{T} |\ln T|} \right)   \right),\\
&+&1,\\
& & u_2 (x, 0) =  T^{ - \frac{1}{p -1}} \chi_0\left(\frac{|x|}{\sqrt{T} |\ln T|} \right)  \left\{ \frac{|x|^2}{ T |\ln T|^2} \left( p-1 + \frac{(p-1)^2|x|^2}{ 4 p T |\ln T|}   \right)^{-\frac{p}{p-1}}     - \frac{2 n \kappa }{(p-1) |\ln T|^2} \right.\\
& +& \left. \left[  \frac{A^2}{|\ln T|^{p_1 +2}} \left(  d_{2,0} + d_{2,1} \cdot \frac{x}{\sqrt T}  \right)   + \frac{A^5 \ln( |\ln(T)|)}{ |\ln T|^{p_1 + 2}} \left( \frac{1}{2} \frac{x^{\mathcal {T} }}{\sqrt T}\cdot  d_{2,2} \cdot 
\frac{x}{\sqrt T} - \text{Tr} (d_{2,2})  \right)  \right]\chi_0\left(\frac{2x}{K_0 \sqrt{T |\ln T| }} \right) \right\}.
\end{eqnarray*}
with $\kappa = (p-1)^{- \frac{1}{p-1}}$, $K_0, A$ are positive constants fixed   large enough, $d_1 = (d_{1,0}, d_{1,1}), d_2 = (d_{2,0}, d_{2,1}, d_{2,2}) $   are parametes we fine tune in our proof, and  $\chi_0 \in C^{\infty}_{0}[0,+\infty), \|\chi_0\|_{L^{\infty}} \leq 1,  \text{ supp} \chi_0 \subset [0,2]$ and   $\chi_0 (x) = 1$ for all  $|x| \leq  1,$ and  $U^*$ is given in  \eqref{defini-U^*} and  related  to the final profile given in  item $(ii)$ of Theorem  \ref{Theorem-profile-complex}. 
Note that when  $p \in    \mathbb{N},$   we took in \cite{DU2017}    a simpler expression  for initial 
 data,  not in   involving   the  final profile   $U^*,$  nor   the    $(+1)$  term  in $u_1 (0)$.  In particular,          adding  this   $(+1)$  term  in our idea  to ensure   that  the real part  of the  solution straps   positive. 
\end{remark}
\begin{remark}
We see in \eqref{equation-satisfied-u_1-u_2}   that  the equation satisfied  by  of $u_2$  is almost 'linear' in $u_2$. Hence,   given an arbitrary   $c_0  \neq   0,$    we can  change a little in our proof to construct   a solution $u_{c_0} = u_{1,c_0} + i u_{2,c_0}$ in  $t \in [0,T) $, which  blows up in finite time $T$  only at the  origin such that  \eqref{esttima-theorem-profile-complex} and  \eqref{asymp-u-start-near-0-profile-complex}     hold    and  the following  holds
\begin{equation}
\left\| (T- t)^{\frac{1}{p-1}} |\ln (T-t)|  u_{2,c_0}(.,t) - c_0g_{0}\left( \frac{.}{\sqrt{(T -t) |\ln(T -t)|}}  \right)  \right\|_{L^{\infty}(\mathbb{R}^n)}  \leq \frac{C}{ |\ln (T -t)|^{\frac{p_1}{2}}},
\end{equation}
and 
\begin{equation}
 u^*_2(x)  \sim \frac{2 p c_0}{(p-1)^2}  \left[ \frac{ (p-1)^2|x|^2}{ 8p |\ln|x||}\right]^{-\frac{1}{p-1}}\frac{1}{ |\ln|x||} , \text{ as } x \to 0,
\end{equation}
\end{remark}
\begin{remark}
As in the case $p=2$ treated  by  Nouaili and Zaag \cite{NZCPDE2015}, and we  also  mentioned   we suspect the  behavior  in Theorem  \ref{Theorem-profile-complex}  to be unstable. This is due to the fact that the number of parameters in the initial data  we  consider below  in Definition \ref{initial-data-bar-u} (see also Remark \ref{remark-initial-data} above) is higher than the dimension of
the blowup parameters which is $n+1$ ($n$ for the blowup points and $1$ for the blowup time).    
\end{remark}

Besides that, we can use the technique of Merle \cite{Mercpam92} 
to construct a solution which blows up at arbitrary given points.
 More precisely, we have  the following Corollary: 
\begin{corollary}[Blowing up at $k$ distinct points]
For any  given points, $x_1,...,x_k$, there exists a solution of \eqref{equ:problem} which blows up exactly at $x_1,...,x_k$. Moreover, the local behavior at each 
blowup point $x_j$ is also given by  \eqref{esttima-theorem-profile-complex}, \eqref{estima-the-imginary-part}, \eqref{asymp-u-start-near-0-profile-complex}, \eqref{asymp-u-start-near-0-profile-complex-imaginary-part}    by replacing $x$ by  $x- x_j$ and $L^\infty (\mathbb{R}^n)$  by $L^{\infty} (|x - x_j| \leq \epsilon_0),$  for some    $\epsilon_0 >0$.
\end{corollary}

\subsection{The strategy  of  the proof}\label{strategy-of-proof}
 From  the singularity of   the   nonlinear term ($u^p$)  when  $p \notin \mathbb{N}$,  we can not  apply the techniques   we used  in  \cite{DU2017}  when  $p \in \mathbb{N}$ (also used  in \cite{MZdm97}, \cite{NZcpde15}, ...). We   need to modify   this  method.  We see that, although our nonlinear  term in not continuous in general, it is  continuous  in the following  half plane
$$ \{ u  \quad  |   \text{Re} (u) > 0 \}.$$
Relying on  this property, our  problem will   be derived    by using   the techniques  which  were  used   in \cite{DU2017} and  the fine  control of   the positivity of the real part.  We treat this challenge  by relying  on the ideas   of the work of  Merle and Zaag in \cite{MZnon97} (or   the work of   Ghoul, Nguyen and Zaag  in \cite{GNZpre16a}  with    inherited  ideas from \cite{MZnon97}) for the   construction of the initial data. We define a  shrinking set  $S(t)$ (see in Definition  \ref{defini-shrinking-set}) which allows a very fine control of  the positivity  of the real part.     More precisely,  it is procceed    to control our solution on  three regions $ P_1(t), P_2(t)$ and   $P_3(t)$  which are  given in subsection  \ref{subsection-shrinking-set}  and which we recall  here: 

- $P_1(t),$ called   the blowup region, i.e $|x| \leq  K_0 \sqrt{(T-t)|\ln (T-t)|}$:  We control our solution   as a perturbation of the intermadiate   blowup  profiles  (for  $t\in [0,T)$) $f_0$ and $ g_0$ given in   \eqref{esttima-theorem-profile-complex} and  \eqref{estima-the-imginary-part}.

-  $P_2(t),$ called the intermediate region, i.e $\frac{K_0}{4} \sqrt{(T-t) |\ln(T-t)|} \leq  |x| \leq  \epsilon_0$:  In this  region,  we will  control our solution  by control  the rescaled  function $U$ of $u$ (see more  \eqref{def-mathcal-U})  to  approach  $\hat U_{K_0}(\tau)$ (see in  \eqref{solution-ODE-hat-U}), by using   a classical parabolic estimates. Roughly  speaking, we control  our solution    as a  perturbation  of the final profiles  for  $t = T$ given in   \eqref{asymp-u-start-near-0-profile-complex} and  \eqref{asymp-u-start-near-0-profile-complex-imaginary-part}.

- $P_3(t),$ called  the regular region, i.e $|x| \geq  \frac{\epsilon_0}{4}$: In this  region, we  control  the solution  as a perturbation  of initial data ($t=0$). Indeed, $T$ will be chosen small by the end of the proof. 

Fixing some constants involved  in the definition  $S(t)$,   we can  prove that  our problem  will be solved    by the control of  the solution  in  $S(t)$. Moreover, we   prove via a priori estimates in the different regions  $P_1, P_2, P_3$ that  the control is reduced to the  control  of a finite dimensional component of the solution. Finally, we  may apply the techniques in \cite{DU2017} to get our conclusion.

We will  organize  our paper   as follows:

- In Section   \ref{section-approach-formal}:  We  give   a formal approach to  explain  how the profiles we have in Theorem \ref{Theorem-profile-complex} appear naturally. Moreover,   we also   approach  our problem  through  two  independant   directions:  \textit{  Inner expansion}   and  \textit{ Outer expansion}, in order to  show that   our profiles   are reasonable.

- In Section \ref{section-existence-solution}:    We   give a formulation   for   our problem (see equation  \eqref{equation-satisfied-by-q-1-2})   and,    step by step   we give   the  rigorous  proof for Theorem \ref{Theorem-profile-complex}, assuming some technical estimates.

- In Section  \ref{the proof of proposion-reduction-finite-dimensional}, we  prove the techical estimates assumed in   Section \ref{section-existence-solution}.

\section{Derivation of the profile (formal approach)}\label{section-approach-formal}
In this section, we  would like to give  a   formal  approach to our problem which  explains how we derive  the profiles for   the solution of equation  \eqref{equ:problem} given  in Theorem \eqref{Theorem-profile-complex}, as well the   asymptotics of the solution.  In particular,  we would   like  to mention that  the main  difference between  the case   $p \in \mathbb{N}$  and $p \notin  \mathbb{N}$ resides   in the  way we handle      the nonlinear      term   $u^p$.  For that reason, we will give  a  lot of    care     for the estimates  involving    the nonlinear term, and go  quickly while  giving  estimates related  to other terms, kindly refering  the reader to  \cite{DU2017}  where   the case   $p \in \mathbb{N}$ was treated.

\subsection{Modeling  the problem}\label{subsection-pro-L}
In this part, we will give definitions and special symbols important for    our work and  explain how the  functions  $f_0, g_0$ arise as blowup profiles for the solution of equation \eqref{equ:problem} as stated in \eqref{esttima-theorem-profile-complex} and \eqref{estima-the-imginary-part}.      Our aim in this section is to give solid (though formal) hints for the existence of a solution $u(t) = u_1 (t)  + i u_2(t)$ to equation  \eqref{equ:problem}  such that 
\begin{equation}\label{lim-u-t-=+-infty}
\lim_{t \to T} \|u(t) \|_{L^\infty(\mathbb{R}^n)} = +\infty,
\end{equation} 
and  $u$ obeys the profiles in \eqref{esttima-theorem-profile-complex} and  \eqref{estima-the-imginary-part}, for some $T > 0$. As we have pointed out   in the introduction, we  are interested in the case where  
$$ p \notin  \mathbb{N},$$
noting  that in this case, we already  have a difficulty to  properly define the nonlinear  term  $u^p$ as a continuous term. In order to overcome  this difficulty, we will restrict ourselves  to the case where 
\begin{equation}\label{condition-Re-u-geq-0}
\text{Re}(u) >  0.
\end{equation}
Our main challenge in this work will  be to show that \eqref{condition-Re-u-geq-0} is propagated  by the flow, at least  for the  initial data  we are suggesting  (see  Definition \ref{initial-data-bar-u} below). Therefore, under the  condition  \eqref{condition-Re-u-geq-0},    by using equation  \eqref{equ:problem}, we deduce that $u_1,u_2$ solve: 
\begin{equation}\label{equation-satisfied-u_1-u_2}
\left\{   \begin{array}{rcl}
\partial_t u_1 &=& \Delta u_1 + F_1(u_1,u_2),\\
\partial_t u_2 &=& \Delta u_2 + F_2(u_1, u_2). 
\end{array} \right.
\end{equation}
where  $F_1(0,0) = F_2 (0,0) = 0$ and  for all $(u_1, u_2) \neq 0$ we have 
\begin{equation}\label{defi-mathbb-A-1-2}
\left\{   \begin{array}{rcl}
F_1(u_1,u_2) &=& \text{ Re} \left[( u_1 + i u_2)^p\right]  =  |u|^{p} \cos\left[ p \text{ Arg }(u_1,u_2) \right]   ,\\
 F_2(u_1, u_2) &=&  \text{ Im} \left[( u_1 + i u_2)^p\right]  =  |u|^{p} \sin\left[ p \text{ Arg }(u_1,u_2) \right]   ,  
\end{array} \right.
\end{equation}
with  $|u| =( u_1^2 + u_2^2)^{\frac{1}{2}} $  and $\text{ Arg}(u_1,u_2),  u_1 > 0 $ is    defined as follows:   
\begin{eqnarray}\label{argument-u-1-u-2}
\text{ Arg}(u_1, u_2) =  
\arcsin \left[ \frac{u_2}{ \sqrt{u_1^2  + u_2^2}}  \right].
\end{eqnarray} 
Note that, in the case where  $p \in \mathbb{N}$, we had the following simple expressions for $F_1, F_2$
\begin{equation}\label{defi-mathbb-A-1-2-entiere}
\left\{   \begin{array}{rcl}
F_1(u_1,u_2) &=& \text{ Re} \left[( u_1 + i u_2)^p\right]  = \sum_{j=0}^{\left[ \frac{p}{2} \right]}  C^{2j }_p ( -1)^{j} u_1^{p-2j} u_2^{2j}  ,\\
 F_2(u_1, u_2) &=&  \text{ Im} \left[( u_1 + i u_2)^p\right]  = \sum_{j=0}^{\left[ \frac{p-1}{2} \right]}  C^{2j + 1}_p ( -1)^{j} u_1^{p-2j -1} u_2^{2j +1}.
\end{array} \right.
\end{equation}
Of course, both expressions  \eqref{defi-mathbb-A-1-2} and  \eqref{defi-mathbb-A-1-2-entiere}   coincide  when  $p  \in \mathbb{N}$. In fact, we will follow our strategy in \cite{DU2017}  for $p \in \mathbb{N}$ and focus  mainly on how we  handle  the nonlinear  terms, since we have a different expression when $p \notin \mathbb{N}.$

\noindent
  Let us introduce \textit{the similarity-variables} for  $u = u_1 + i u_2$ as follows:
\begin{equation}\label{similarity-variales}
w_1 (y,s)  =  (T-t)^{\frac{1}{p-1}}u_1 (x,t), w_2 (y,s)  =  (T-t)^{\frac{1}{p-1}}u_2 (x,t) ,y = \frac{x}{\sqrt {T- t}} , s = - \ln(T-  t).
\end{equation}
By using  \eqref{equation-satisfied-u_1-u_2},   we   obtain a system satisfied by    $( w_1, w_2),$  for all $y \in \mathbb{R}^n$ and $s \geq  -\ln T $ as follows:
\begin{equation}\label{equation-satisfied-by-w-1-2}
\left\{  \begin{array}{rcl}
\partial_s w_1  & = &   \Delta w_1 - \frac{1}{2} y \cdot \nabla w_1  - \frac{w_1}{p-1}   +  F_1(w_1,w_2), \\
\partial_s w_2 &=& \Delta w_2 - \frac{1}{2} y \cdot \nabla w_2  - \frac{w_2}{p-1} + F_2 (w_1,w_2). 
\end{array} \right.
\end{equation}
Then   note that  studying  the    asymptotics  of  $u_1 + i u_2$ as  $t \to T$ is    equivalent to  studying the  asymptotics of $w_1 + i w_2$  in long time.   We  are first interested in the set of constant solutions  of 
 \eqref{equation-satisfied-by-w-1-2}, denoted by 
 $$\mathcal{S} = \{(0,0)\} \cup  \left\{  \kappa  \left( \cos\left(\frac{2 k \pi }{p-1}\right), \sin\left(\frac{2 k \pi }{p-1} \right)   \right)     \text{ where } \kappa 
 = (p-1)^{-\frac{1}{p-1}},   \text{ and }    k \in \mathbb{N}     \right\}.$$ 
We remark    that    $\mathcal{S}$ is infinity  if  $p$ is not integer.  However,   from  the transformation  \eqref{similarity-variales}, we slightly precise  our goal in \eqref{lim-u-t-=+-infty} by requiring   in addition that
 $$ (w_1, w_2) \to (\kappa, 0)  \text{  as }  s  \to +\infty.$$
Introducing  $w_1 = \kappa + \bar w_1, $ our goal because to get 
$$ (\bar w_1, w_2) \to (0,0) \text{ as } s \to + \infty. $$
From \eqref{equation-satisfied-by-w-1-2}, we deduce  that $\bar w_1, w_2$  satisfy the following system

\begin{equation}\label{system-bar R-varphi}
\left\{   \begin{array}{rcl}
\partial_s \bar w_1 &=& \mathcal{L} \bar w_1  + \bar B_1(\bar w_1, w_2),\\
\partial_s w_2 &=& \mathcal{L} w_2  + \bar B_2 (\bar w_1, w_2).
\end{array} \right.
\end{equation}
where
\begin{eqnarray}
\mathcal{L} &=& \Delta - \frac{1}{2} y \cdot \nabla + Id,\label{define-operator-L}\\
\bar B_1 ( \bar w_1, w_2 )  & = & F_1(\kappa + \bar w_1, w_2)- \kappa^p - \frac{p}{p-1} \bar w_1 , \label{defini-bar-B-1}\\ 
\bar B_2 (\bar w_1, w_2) &=& F_2(\kappa + \bar w_1, w_2) - \frac{p}{p-1} w_2.\label{defini-bar-B-2}
\end{eqnarray}
It is  important to study  the linear operator $\mathcal{L}$  and the asymptotics of $\bar B_1, \bar B_2$ as  $(\bar w_1, w_2) \to (0,0)$ which will appear as quadratic. 

$\bullet $ \textit{ The properties of $\mathcal{L}$:}

We observe that the   operator $\mathcal{L}$  plays an important 
role in our analysis. It is easy to find an analysis space such that $\mathcal{L}$ is
 self-adjoint. Indeed,  $\mathcal{L} $  is self-adjoint in $L^2_\rho(\mathbb{R}^n)$, where $L^2_\rho$ is the weighted space associated to the weight $\rho$ defined by
\begin{equation}\label{def-rho-rho-j}
\rho(y)   =\frac{e^{- \frac{|y|^2}{4}}}{(4 \pi)^{\frac{n}{2}}}  = \prod_{j=1}^n \rho(y_j), \text{ with } \rho(\xi) = \frac{e^{- \frac{|\xi|^2}{4}}}{(4 \pi)^{\frac{1}{2}}},
\end{equation}
and  the spectrum set of $\mathcal{L}$
\begin{equation}\label{spectrum-set}
\textup{spec}(\mathcal{L}) = \displaystyle \left\{1 - \frac m2, m \in \mathbb{N}\right\}.
\end{equation}
Moreover,  we can find  eigenfunctions which  correspond  to each  eigenvalue $ 1 - \frac{m}{2}, m \in \mathbb{N}$: 
\begin{itemize}
\item[-] The  one space dimensional case:  the eigenfunction corresponding to
 the eigenvalue $1 - \frac m2$   is  $h_m$, the rescaled Hermite polynomial given as follows
 \begin{equation}\label{Hermite}
h_m(y)   = \sum_{j=0}^{\left[ \frac{m}{2}\right]} \frac{(-1)^jm! y^{m - 2j}}{j! ( m -2j)!}.
\end{equation}
  In particular, we have the following orthogonality property:
 $$\int_{\mathbb{R}} h_i h_j  \rho dy  =  i! 2^i \delta_{i,j}, \quad \forall (i,j) \in \mathbb{N}^2. $$
\item[-] The higher dimensional case:  $n \geq 2$, the eigenspace  $\mathcal{E}_{m}$, corresponding
 to the eigenvalue $1 - \frac {m}{2}$  is defined as follows:
\begin{equation}\label{eigenspace}
 \mathcal{E}_m   = \left\{ h_{\beta} = h_{\beta_1} \cdots h_{\beta_n}, \text{ for all } \beta \in \mathbb{N}^n, |\beta|  = m , |\beta| = \beta_1 + \cdots +\beta_n  \right\}.
\end{equation}
\end{itemize}
Accordingly,   we can represent  an arbitrary function $r \in  L^{2}_\rho$ as follows
\begin{eqnarray*}
r  &=& \displaystyle \sum_{\beta, \beta \in \mathbb{N}^n} r_\beta h_\beta  (y),\\
\end{eqnarray*}
where: $r_\beta$ is the projection of $r$ on $h_\beta $ for any $ \beta \in \mathbb{R}^n$ which is defined as follows:
\begin{equation}\label{projector-h-beta}
r_\beta =  \mathbb{P}_\beta (r)  =  \int r k_\beta \rho dy, \forall  \beta \in \mathbb{N}^n, 
 \end{equation}
with 
\begin{equation}\label{note-k-beta-hermite}
 k_\beta (y) = \frac{ h_\beta}{\|h_\beta\|^2_{L^2_\rho}}.
\end{equation}
$\bullet $ \textit{ The  asymptotics of $\bar B_1(\bar w_1, w_2), \bar B_2(\bar w_1, w_2)$:}
The following asymptotics   hold: 
 \begin{eqnarray} 
\bar B_1 (\bar w_1, w_2) &=& \frac{p}{2 \kappa} \bar w_1^2  + O (|\bar w_1|^3 + |w_2|^2),\label{asymptotic-bar-B-1-1}\\
\bar B_2( \bar w_1, w_2) &=& \frac{p}{ \kappa} \bar w_1 w_2  + O \left( |\bar w_1|^2 |w_2| \right) + O  \left( |w_2|^3 \right), \label{asymptotic-bar-B-2-1}                                 
\end{eqnarray}
as  $(\bar w_1, w_2) \to (0,0).$  Note that although we have here  the   expressions  of the  nonlinear  terms $F_1, F_2$  which  are different   from  the case  $p \in \mathbb{N}$ (see \eqref{defi-mathbb-A-1-2} and  \eqref{defi-mathbb-A-1-2-entiere}), the expressions  coincide, since we have  $u \sim  \kappa = (p - 1)^{- \frac{1}{p-1}}$   in all case      (see Lemma \ref{asymptotic-bar-B-1-2} below). 
\subsection{Inner expansion}\label{subsection-inner-expan}
 In this part, we study    the asymptotics  of the solution 
 in $L^2_\rho(\mathbb{R}^n).$  Moreover, for simplicity  we suppose that  $n =1$, and we recall that  we aim at  constructing a  solution of
  \eqref{system-bar R-varphi} such that 
  $( \bar w_1,  w_2)  \to (0,0) $.   Note first that  the spectrum  of $\mathcal{L}$  contains  two   positive eigenvalues $1, \frac{1}{2}$, a   neutral  eigenvalue $0$ and all  the other ones are  strictly negative. So, in the representation of the solution in $L^2_\rho,$  it is reasonable to think that    the part corresponding to the negative spectrum  is easily  controlled.  Imposing  a symmetry condition  on the solution   with  respect  of $y$,   it is reasonable to look for  a  solution  $ \bar w_1, w_2  $ of the form:
\begin{eqnarray*}
\bar w_1 & = & \bar w_{1,0} h_0  + \bar w_{1,2} h_2, \\
w_2 &=& w_{2,0} h_0 + w_{2,2} h_2. 
\end{eqnarray*}
From the assumption that   $ (\bar w_1, w_2)  \to (0,0)$,  we see that    $ \bar w_{1,0}, \bar w_{1,2}, w_{2,0} , w_{2,2} \to 0$ as $s \to + \infty$. 
We  see also  that we can understand the asymptotics of the solution $\bar w_1, w_2$ in $L^2_\rho$  from  the study of the asymptotics of $\bar w_{1,0}, \bar w_{1,2}, w_{2,0}, w_{2,2}.$
 We  now  project  equations   \eqref{system-bar R-varphi} on $h_0$ and
  $h_2.$ Using   the asymptotics of $\bar B_1, \bar  B_2$  in  \eqref{asymptotic-bar-B-1-1} and \eqref{asymptotic-bar-B-2-1},  we get the following ODEs for $ \bar w_{1,0}, \bar w_{1,2}, w_{2,0} , w_{2,2}:$
\begin{eqnarray}
\partial_s \bar w_{1,0} &=& \bar w_{1,0} + \frac{p}{2 \kappa} \left( \bar w_{1,0}^2 + 8 \bar w_{1,2}^2\right)+ O (|\bar w_{1,0}|^3 + |\bar w_{1,2}|^3) + O( |w_{2,0}|^2 + |w_{2,2}|^2),\label{ODe-bar w_1-0} \\
\partial_s \bar w_{1,2} &=& \frac{p}{\kappa} \left( \bar w_{1,0} \bar w_{1,2} + 4 \bar w_{1,2}^2\right) + O (|\bar w_{1,0}|^3 + |\bar w_{1,2}|^3) + O( |w_{2,0}|^2 + |w_{2,2}|^2) ,\label{ODe-bar w_1-2}\\
\partial_s w_{2,0} &=&  w_{2,0} +  \frac{p}{\kappa}\left[\bar w_{1,0} w_{2,0} + 8 \bar w_{1,2} w_{2,2} \right] + O ((|\bar w_{1,0}|^2 + |\bar w_{1,2}|^2)(|w_{2,0}|+ |w_{2,2}|)) \label{ODe- w_2-0}\\
&+& O( |w_{2,0}|^3 + |w_{2,2}|^3) ,\nonumber\\
\partial_s w_{2,2} &=&   \frac{p}{\kappa} \left[  \bar w_{1,0} w_{2,2} + \bar w_{1,2} w_{2,0} + 8 \bar w_{1,2} w_{2,2}\right] + O ((|\bar w_{1,0}|^2 + |\bar w_{1,2}|^2)(|w_{2,0}|+ |w_{2,2}|)) \label{ODe- w_2-2}\\
&+& O( |w_{2,0}|^3 + |w_{2,2}|^3).   \nonumber
\end{eqnarray}
Assuming that 
\begin{equation}\label{asumption-bar-w-1-0-lesthan-bar-w-1-2}
\bar w_{1,0}, w_{2,0}, w_{2,2} \ll \bar w_{1,2},
\end{equation}
and 
\begin{equation}\label{asuming-bar-w-1-w-2-2-leq-s^2}
\bar w_{1,0}, w_{2,0}, w_{2,2} \lesssim \frac{1}{s^2}, 
\end{equation}
as   $s \to  + \infty$.  Similarly as   in \cite{DU2017}, where    we have   $p \in \mathbb{N},$    we obtain the  following asymptotics of     $\bar w_{1,0}, \bar w_{1,2}, w_{2,0}, w_{2,2}:$

\begin{eqnarray*}
\bar w_{1,0}   & = &  O\left( \frac{1}{s^2}\right),\\
\bar  w_{1,2}  &  =&  - \frac{\kappa}{ 4 ps}  + O \left(  \frac{\ln s}{s^2} \right),\\
w_{2,0}  &  = &   O \left(  \frac{1}{s^3}\right), \\
w_{2,2}  &  =  &    \frac{c_{2,2}}{s^2} +      O \left(  \frac{ \ln s}{s^3} \right), c_{2,2} \neq 0, 
\end{eqnarray*}
as   $s \to + \infty $ which  satisfiy  the assumption   in   \eqref{asumption-bar-w-1-0-lesthan-bar-w-1-2}  and   \eqref{asuming-bar-w-1-w-2-2-leq-s^2}. Then,  we have

\begin{eqnarray}
w_1&=& \kappa - \frac{\kappa}{4p s } ( y^2 - 2) + O\left(\frac{1}{s^2} \right), \label{asymptotic-w-1}\\
w_2  &=& \frac{ c_{2,2}}{s^2} (y^2 - 2)  + O \left( \frac{\ln s}{s^3}\right)\label{asymptotic-w-2},
\end{eqnarray}
in   $L^2_\rho(\mathbb{R})$  for some  $c_{2,2}$ in $\mathbb{R}^*$. Note that,
by using parabolic  regularity, we can derive  that the asymptotics \eqref{asymptotic-w-1}, \eqref{asymptotic-w-2}   also hold for all $|y| \leq K,$ where $K$ is an arbitrary positive constant.
\subsection{Outer expansion}As for the inner expansion, we here assume that    $n = 1$.  We   see  that asymptotics \eqref{asymptotic-w-1} and  \eqref{asymptotic-w-2}
 can not give us a shape, since they  hold uniformly  on compact sets (where  we only see the constant solutio  $(\kappa, 0)$)  and not in larger sets.  Fortunately,  we observe from  \eqref{asymptotic-w-1} and \eqref{asymptotic-w-2} that the profile may be based on the following variable:    
\begin{equation}\label{the-variable-z-profile}
z  = \frac{y}{\sqrt s}.
\end{equation}
This  motivates us  to look for  solutions  of the form:
\begin{eqnarray*}
w_1(y,s) &=& \sum_{j=0}^{\infty} \frac{R_{1,j} (z)}{s^j}, \\
w_2 (y,s)  &=& \sum_{j=1}^{\infty} \frac{R_{2,j}(z)}{s^j}.
\end{eqnarray*}
  Note that,  our purpose     is  to  construct      a  solution   where  the real part  is positive. So, it is   reasonnable  to  assume that  $w_1 > 0$ and   $R_{1,0} (z)  > 0$ for all  $z \in \mathbb{R}$. Besides that,  we also assume that  $R_{1,j} , R_{2,j}$ are   smooth  and  have  bounded   derivatives.  From the definitions of  $F_1, F_2,$ given in \eqref{defi-mathbb-A-1-2},  we have the following
  \begin{eqnarray*}
\left|  F_1 \left(  \sum_{j=0}^{\infty} \frac{R_{1,j} (z)}{s^j} ,  \sum_{j=1}^{\infty} \frac{R_{2,j}(z)}{s^j}   \right)    -  R_{1,0}^{p}  (z) -  \frac{p  R_{1,0}^{p-1} (z) R_{1,1} (z)}{s}     \right| &\leq &  \frac{C (z)}{s^2},\\
  \left|  F_2 \left(  \sum_{j=0}^{\infty} \frac{R_{1,j} (z)}{s^j} ,  \sum_{j=1}^{\infty} \frac{R_{2,j}(z)}{s^j}   \right)    - \frac{p R_{1,0}^{p-1} (z) R_{2,1}(z)}{s}   \right. \\
  \left.  -  \frac{1}{s^2}  \left(  p R_{1,0}^{p-1} (z) R_{2,2}  + p(p-1) R_{1,0}^{p-2}(z) R_{1,1} (z) R_{2,1} (z) \right) \right| &\leq  & \frac{C(z)}{s^3}.
\end{eqnarray*}
   Thus,  for each   $z\in \mathbb{R},$  by using   system  \eqref{equation-satisfied-by-w-1-2},   taking  $s \to  +\infty,$  we obtain  the following system:
\begin{eqnarray}
0 &=&  - \frac{1}{2} R_{1,0}' (z) \cdot z - \frac{R_{1,0} (z)}{p-1} + R_{1,0}^p (z), \label{equa-R-1-0} \\
0 &=&  - \frac{1}{2} z R_{1,1}'(z) - \frac{R_{1,1}}{p-1} (z) + p R_{1,0}^{p-1} (z)R_{1,1} (z)+ R_{1,0}'' (z)
 + \frac{z R_{1,0}'(z)}{2}, \label{equa-R-1-1}  \\
0 &=&  - \frac{1}{2}  R_{2,1}' (z) \cdot z - \frac{R_{2,1}}{p-1}(z) + p R_{1,0}^{p-1} (z)R_{2,1}(z), \label{equa-R-2-1}\\
0 &=&  - \frac{1}{2}  R_{2,2}'(z). z - \frac{R_{2,2}(z)}{p-1} + p R^{p-1}_{1,0}(z) R_{2,2} (z)+ R''_{2,1}(z) + R_{2,1}(z) + \frac{1}{2} R_{2,1}'(z) \cdot z \label{equa-R-2-2} \\
&+&   p (p-1) R^{p-2}_{1,0} (z)R_{1,1} (z)R_{2,1}(z).\nonumber
\end{eqnarray}
This system  is quite similar to  \cite{DU2017} (where   $p \in \mathbb{N}$), and we can find   the fomulas of  $R_{1,0}, R_{1,1} , R_{2,1}, R_{2,2}$ as follows:
\begin{eqnarray}
R_{1,0} (z)  & =&   \left(   p-1  +  b |z|^2\right)^{-\frac{1}{p-1}},\label{solu-R-0}\\
R_{1,1} (z)&=& \frac{(p-1)}{ 2 p} (p-1 + bz ^2)^{- \frac{p}{p-1}} -  \frac{p-1}{4 p} z^2 \ln (p-1 + b z^2) (p-1 + b z^2)^{- \frac{p}{p-1}}, \label{solu-R-1-1}\\
R_{2,1} (z)  & =&   \frac{z^2}{(p-1 + bz^2)^{\frac{p}{p-1}}},\label{solu-varphi_1}\\
R_{2,2} (z) &=& - 2 (p-1 + b z^2) ^{- \frac{p}{p-1}} + H_{2,2} (z),\label{solu-R-2-2}
\end{eqnarray}
where    $b = \frac{(p-1)^2}{4p}$ and 
\begin{eqnarray*}
H_{2,2} (z)  &=& C_{2,1} (p) z^2 (p-1 + b z^2)^{-\frac{2p-1}{p-1}} + C_{2,3} (p) z^2 \ln (p-1 + b z^2) (p-1 + b z^2)^{-\frac{p}{p-1}} \\
&+& C_{2,3} (p) z^2 \ln (p-1 + b z^2) (p-1 + b z^2)^{-\frac{2p-1}{p-1}}.
\end{eqnarray*}

\subsection{Matching  asymptotics}
 By comparing  the          inner expansion and   the outer expansion,  then     fixing  several constants,    we      have    the   following   profiles  for  $w_1$ and $w_2 $
 \begin{equation}\label{equavalent-w-1-2-Phi-1-2}
\left\{   \begin{array}{rcl}
w_1 (y,s) &\sim &  \Phi_1(y,s),\\
w_2 (y,s) &\sim & \Phi_2(y,s),
\end{array} \right.
\end{equation}
 where 
\begin{eqnarray}
\Phi_1(y,s) &=& \left( p-1 + \frac{(p-1)^2}{4 p} \frac{|y|^2}{s} \right)^{-\frac{1}{p-1}} + \frac{n \kappa}{2 p s},\label{defi-Phi-1}\\
\Phi_2 (y,s) &=&\frac{|y|^2}{s^2} \left( p-1 + \frac{(p-1)^2}{4 p} \frac{|y|^2}{s} \right)^{-\frac{p}{p-1}} -  \frac{2n \kappa}{(p-1) s^2},\label{defi-Phi-2}
\end{eqnarray}
for all $(y,s) \in \mathbb{R}^n \times (0,  + \infty)$.   In this setion,  we will give  a regious  proof for the  existence of a   solution  $(w_1, w_2)$  of equation  \eqref{equation-satisfied-by-w-1-2} where \eqref{equavalent-w-1-2-Phi-1-2} holds. 
\section{Existence of  a blowup  solution  in Theorem  \ref{Theorem-profile-complex}}\label{section-existence-solution}
In Section \ref{section-approach-formal}, we adopted a formal approach in order to justify how the profiles $f_0, g_0$   arise as blowup profiles for the solution of   equation \eqref{equ:problem}, given  in Theorem  \ref{Theorem-profile-complex}.  In this section, we give a rigorous  proof to justify the existence  of a solution approaching those profiles.
\subsection{Formulation of the problem}
 In this subsection, we aim at giving a complete  formulation of   our problem in order  to justify  the  formal  approach which is given in the previous  section. We introduce
 \begin{equation}\label{defini-q-1-2}
 \left\{   \begin{array}{rcl}
 w_1 &=& \Phi_1  + q_1, \\
w_2 &=& \Phi_2  + q_2,
\end{array}  \right.
 \end{equation}
where      $\Phi_1, \Phi_2$ are defined in \eqref{defi-Phi-1} and \eqref{defi-Phi-2} respectively.     Then,  by  using  \eqref{equation-satisfied-by-w-1-2},  we  derive  the following system, satisfied  by  $(q_1, q_2) :$ 
\begin{equation}\label{equation-satisfied-by-q-1-2}
\partial_s \binom{q_1}{q_2} = \left( \begin{matrix}
\mathcal{L} + V  & 0 \\
  0 & \mathcal{L} + V
\end{matrix} \right) \binom{q_1}{q_2} +   \left( \begin{matrix}
 V_{1,1}  &  V_{1,2} \\
 V_{2,1 }& V_{2,2}
\end{matrix} \right)  \binom{q_1 }{q_2}  + \binom{B_1(q_1, q_2)}{B_2(q_1,q_2)} + \binom{R_1 }{R_{2} },
\end{equation}
where linear operator $\mathcal{L}$ is defined   in \eqref{define-operator-L} and:\\ 
\medskip
\noindent
 - The potential functions $V, V_{1,1}, V_{1,2}, V_{2,1}, V_{2,2} $ are defined  as follows 
\begin{eqnarray}
V(y,s) &=&  p \left( \Phi_1^{p- 1}  - \frac{1}{p-1}\right)\label{defini-potentian-V}, \\
V_{1,1} (y,s) & = &   \partial_{u_1} F_1( u_1,  u_2)|_{(u_1,u_2) =  (\Phi_1, \Phi_2)}  - p \Phi_1^{p-1} ,\label{defini-V-1-1} \\
V_{1,2} (y,s) & = &  \partial_{u_2} F_1( u_1,  u_2)|_{(u_1,u_2) =  (\Phi_1, \Phi_2)}, \label{defini-V-1-2} \\
V_{2,1} (y,s) &  = &    \partial_{u_1} F_2( u_1,  u_2)|_{(u_1,u_2) =  (\Phi_1, \Phi_2)}, \label{defini-V-2-1}  \\
V_{2,2} (y,s) & = & \partial_{u_2} F_2( u_1,  u_2)|_{(u_1,u_2) =  (\Phi_1, \Phi_2)} - p \Phi_1^{p-1}. \label{defini-V-2-2}
\end{eqnarray}
\medskip
\noindent
 - The quadratic terms $B_1 (q_1, q_2), B_2 (q_1,q_2)$  are defined as follows:
\begin{eqnarray}
B_1 (q_1,q_2) & = &  F_1 \left(  \Phi_1 + q_1, \Phi_2 + q_2 \right) - F_1(\Phi_1, \Phi_2)  -   \partial_{u_1} F_1( u_1,  u_2)|_{(u_1,u_2) =  (\Phi_1, \Phi_2)}  q_1   \label{defini-quadratic-B-1}\\
&-&  \partial_{u_2} F_1( u_1,  u_2)|_{(u_1,u_2) =  (\Phi_1, \Phi_2)} q_2,\nonumber\\
B_2(q_1, q_2) & = &   F_2 \left(  \Phi_1  + q_1 , \Phi_2 + q_2 \right) - F_2(\Phi_1, \Phi_2) -   \partial_{u_1} F_2( u_1,  u_2)|_{(u_1,u_2) =  (\Phi_1, \Phi_2)} q_1  \nonumber\\
&  -  & \partial_{u_2} F_2( u_1,  u_2)|_{(u_1,u_2) =  (\Phi_1, \Phi_2)}  q_2.\label{defini-term-under-linear-B-2}
\end{eqnarray}

\medskip
\noindent
 -  The rest terms $R_1(y,s), R_2(y,s)$ are defined as follows:
\begin{eqnarray}
R_1 (y,s) &=& \Delta \Phi_1 - \frac{1}{2} y \cdot \nabla \Phi_1 - \frac{\Phi_1}{p-1} + F_1 (\Phi_1, \Phi_2) - \partial_s \Phi_1 , \label{defini-the-rest-term-R-1}\\
R_2 (y,s) &=& \Delta \Phi_2 - \frac{1}{2} y \cdot \nabla \Phi_2  - \frac{\Phi_2}{p-1} + F_2 (\Phi_1, \Phi_2) - \partial_s \Phi_2, \label{defini-the-rest-term-R-2}
\end{eqnarray}
where   $ F_1, F_2$ are defined in  \eqref{defi-mathbb-A-1-2}.

By the  linearization around $\Phi_1, \Phi_2,$  our  problem is reduced to  constructing  a solution $(q_1,q_2)$ of system \eqref{equation-satisfied-by-q-1-2}, satisfying
$$ \|q_1\|_{L^{\infty}(\mathbb{R}^n)} + \|q_2\|_{L^{\infty}(\mathbb{R}^n)} \to 0 \text{ as } s \to +\infty.$$
 Looking at system \eqref{equation-satisfied-by-q-1-2}, we  already know   some of  the   main properties of  the  linear operator $\mathcal{L}$  (see page \pageref{define-operator-L}).  As  for     the  potentials $V_{j,k}$  where $  j,k  \in  \{1,2\},$ they  admit the following asymptotics:
\begin{eqnarray*}
\|V_{1,1} (.,s)\|_{L^\infty}   + \| V_{2,2} (.,s)\|_{L^\infty}   &\leq & \frac{C}{s^2},\\ 
  \| V_{1,2}(.,s)\|_{L^\infty}   + \| V_{2,1}(.,s)\|_{L^\infty}  & \leq  & \frac{C}{s}, \forall  s\geq 1,
\end{eqnarray*}
(see   Lemma \ref{lemmas-potentials}  below). 

\noindent
Regarding  the terms $B_1,B_2$  which  are   quadratic, we have  these  estimates    
\begin{eqnarray*}
\|B_1(q_1,q_2)\|_{L^\infty} &\leq &   \frac{C A^{4}}{s^{\frac{p}{2}}}, \\
\| B_2 (q_1,q_2)\|_{L^\infty} & \leq & \frac{C A^{2}}{s^{ 1 + \min \left(   \frac{p-1}{4}, \frac{1}{2}\right)}},
\end{eqnarray*}
 if $q_1,  q_2$   are   small in some sene (see   Lemma  \ref{lemma-quadratic-term-B-1-2}  below). 
 
 \noindent
 In addition to that, the rest terms $R_1, R_2$ satisfy  the following asymptotics
 \begin{eqnarray*}
\| R_1(.,s)\|_{L^\infty(\mathbb{R}^n)}  & \leq  & \frac{C}{s},\\
\| R_2(.,s)\|_{L^\infty(\mathbb{R}^n)}  & \leq & \frac{C}{s^2},
\end{eqnarray*}
 (see Lemma \ref{lemma-rest-term-R-1-2} below).
 
 \noindent
 As a matter of  fact,  the dynamics of  equation  \eqref{equation-satisfied-by-q-1-2} will mainly depend on the  main  linear operator 
$$ \left(   \begin{matrix}
\mathcal{L} + V  & 0\\
0  & \mathcal{L} + V
\end{matrix}  \right), $$
and  the effects  of the orther terms  will be less important  except on the zero mode  of this equation.   For that reason, we need to understand  the dynamics of  $\mathcal{L} + V$. Since  the spectral properties  of  $\mathcal{L}$ were already  introduced  in Section \ref{subsection-pro-L}, we will focus here on the effect of $V$. 

$i)$ Effect of $V$ inside the blowup region $\{|y| \leq K_0\sqrt s\}$ with $K_0>0:$ It satisfies the following  estimate:
$$  V  \to 0  \text{ in  } L^2_\rho(|y| \leq K_0 \sqrt s ) \text{ as } s \to + \infty,$$
which means that the effect of $V$ will be negligeable with  respect of the effect  of $\mathcal{L},$ except perhaps on the null mode of $\mathcal{L}$ (see item $(ii)$  of Proposition  \ref{prop-dynamic-q-1-2-alpha-beta}   below). 

$ii)$ Effect  of  $V$ outside   the blowup region:   For each  $\epsilon > 0,$ there exist $K_{\epsilon} >0$ and $ s_{\epsilon} >0$ such that
$$ \sup_{\frac{y}{\sqrt s} \geq K_{\epsilon}, s \geq s_{\epsilon}} \left| V(y,s)  - \left(- \frac{p}{p-1} \right) \right| \leq    \epsilon.$$ 
Since $1$ is the biggest eigenvalue of $\mathcal{L}$ (see \eqref{spectrum-set}),    the operator $\mathcal{L}+ V$  behaves as one with with a fully negative spectrum  outside blowup region $\{|y| \geq K_\epsilon\sqrt s\}$, which makes the control of the solution in this region easy.

\medskip
Since the behavior of the potential $V$ inside and outside the blowup region is different,
 we will consider the dynamics of the solution for   $|y| \leq 2K_0\sqrt s$ and for $|y| \geq K_0\sqrt s$ separately for some   $K_0$  to be fixed large.
For  that purpose,     we introduce the following cut-off function
\begin{equation}\label{def-chi}
\chi(y,s)  = \chi_0\left(\frac{|y|}{K_0 \sqrt s} \right),
\end{equation}
where     $\chi_0$ is defined as  a cut-off function:   
\begin{equation}\label{defini-chi-0}
\chi_0 \in C^{\infty}_{0}[0,+\infty),  \chi_0(x) = \left\{  \begin{array}{l}
 1 \quad \text{ for } x  \leq 1,\\
 0 \quad   \text{ for }  x  \geq 2,
\end{array} \right.   \text{ and }   \|\chi_0\|_{L^{\infty}} \leq 1.
\end{equation}
 Hence,  it is  reasonable  to  consider  separately  the solution  in   the blowup region  $\{ |y| \leq 2 K_0 \sqrt s \}$   and  in the  regular region  $\{| y| \geq  K_0 \sqrt s \}$.    More precisely,  let us define the following notation for all functions   $r$ in   $L^\infty$ as follows 
\begin{equation}\label{defini-q-1-1-e}
 r =  r_b  + r_e     \text{ with }  r_b =   \chi  r \text{ and  }   r_e   = (1 - \chi) r.
\end{equation}
 Note  in particular   that $\text{ supp} (r_b) 	\subset \mathbb{B} ( 0, 2 K_0 \sqrt s)$ and  $  \text{ supp} (r_e) 	     \subset  \mathbb{R}^n \setminus   \mathbb{B} ( 0,  K_0 \sqrt s)$.
Besides that,   we  also expand    $r_b $ in $L^2_\rho$    according to the spectrum of $\mathcal{L}$  (see Section \ref{subsection-pro-L} above):
\begin{equation}\label{representation-q-1-L-2-rho}
 r_b (y) =    r_0 + r_1 \cdot y  +  \frac{1}{2} y^{\mathcal{T}}  \cdot  r_2  \cdot y    -  \text{ Tr} \left(  r_2 \right)  + r_- (y) ,
\end{equation}
where  $r_0$ is a scalar, $r_1$ is a vector  in  $\mathbb{R}^n$ and  $r_2$ is a  $n \times n$ matrix  defined by  
 \begin{eqnarray*}
r_0  &  =  &  \int_{\mathbb{R}^n}  r_b \rho (y) d y, \\
r_1  &=&   \int_{\mathbb{R}^n}  r_b \frac{y}{2}  \rho (y) d y, \\
r_2 & =&    \left (  \int_{\mathbb{R}^n}  r_b  \left(  \frac{1}{4} y_j y_k -  \frac{1}{2} \delta_{j,k} \right)  \rho (y) d y \right)_{1 \leq j,k \leq n},\\
\end{eqnarray*}
with   $\text{ Tr}(r_2)$ being  the trace of  matrix $r_2$.
The reader   should  keep  in mind that $r_0, r_1,r_2$ are only the   coordinates of $r_b$, not for  $r$.  Note that $r_m$ is  the projection of  $r_b$ on  the eigenspace  of $\mathcal{L}$ corresponding to the eigenvalue  $\lambda = 1 - \frac{m}{2}.$ Accordingly, $r_-$ is the projection of $r_b$  on the negative part of the spectrum  of $\mathcal{L}.$ As a consequence of \eqref{defini-q-1-1-e}  and \eqref{representation-q-1-L-2-rho}, we see that every $r \in L^\infty (\mathbb{R}^n)$  can be decomposed into $5$ components as follows:
\begin{equation}\label{decom-5-parts}
r = r_b + r_e  = r_0 + r_1 \cdot y + \frac{1}{2} y^\mathcal{T}  \cdot r_2 \cdot y - \text{Tr} (r_2) + r_- + r_e.
\end{equation} 
\subsection{The shrinking set}\label{subsection-shrinking-set}

According to  \eqref{similarity-variales} and \eqref{defini-q-1-2}, our goal is to construct a solution  $(q_1, q_2) $ of system \eqref{equation-satisfied-by-q-1-2} such that they satisfiy the following estimates:
\begin{equation}\label{equa-q-1q-2-infty-to-0}
\|q_1(.,s)\|_{L^\infty}  +  \|q_2(.,s)\|_{L^\infty}  \to 0 \text{ as } s \to +\infty.
\end{equation}

Here, we   aim at  constructing  a shrinking set to $0$.   Then,   the control  of $(q_1,q_2) \to 0,$  will be a consequence of      the  control  of     $(q_1,q_2)$ in this shrinking set.   In addition to that,  we      have   to  control  the solution  $q_1$  so that   
\begin{equation}\label{control-w-1-positive}
w_1 =   q_1  + \Phi_1 > 0, 
\end{equation}
 (this is equivalent to   have   $u_1  > 0$)  and it is  one of the main  difficults  in our analysis.    As a matter of fact,    the shrinking sets which were  constructed   in   \cite{MZdm97}  by  Merle and Zaag or even in   \cite{DU2017},   are  not  sharp enough to ensure \eqref{control-w-1-positive}.    In other words, our  set has to shrink to $0$ as  $s \to  + \infty$ and ensure that the real part   of the  solution to  \eqref{equation-satisfied-by-w-1-2} is always positive.     In fact,   the positivity is the first  thing to be solved.     For the control         of  the positivity  of   the real part,  we     rely on   the ideas,        given by Merle and Zaag in \cite{MZnon97} for the control  of the  solution of the  following  equation:
 \begin{equation}\label{equa-vortex-conected}
 \partial_t  u = \Delta u - \eta \frac{\left| \nabla  u \right|^2}{u}  + |u|^{p-1}u, u \in \mathbb{R}.
\end{equation}  
 In \cite{MZnon97}, the authors  needed  a sharp control of  $u$ and  $| \nabla u |$  near zero, in order to bound  the  term  $\frac{|\nabla u|^2}{u}.$ Here, we will use  their ideas  in order to control $u_1$ near  zero and  ensure  its  positivity. As  in  \cite{MZnon97},  we will  control  the solution differently in 3  overlapping  regions defined  as follows: 
 
 \noindent
   For    $K_0 > 0, \alpha_0 > 0, \epsilon_0 >0,    t \in [0,T), s \in [- \ln T, +\infty),  s = - \ln (T-t)$, we introduce  a cover of $\mathbb{R}^n$  as follows
$$ \mathbb{R}^n   \subset P_1(t) \cup  P_2(t) \cup P_3(t),$$
where  
\begin{eqnarray*}
P_1(t) &=& \{ x |\quad   |x|  \leq  K_0 \sqrt{(T-t)|\ln(T-t)|}\}  = \{x | \quad |y| \leq  K_0  \sqrt s \} = \{ x | \quad |z| \leq K_0\},\\
P_2(t) &=&  \left\{ x | \quad  \frac{K_0}{4} \sqrt{(T-t)|\ln(T-t)|} \leq |x| \leq \epsilon_0  \right\} =  \left\{ x | \quad  \frac{K_0}{4} \sqrt s \leq |y| \leq \epsilon_0 e^{\frac{s}{2}} \right\}\\
&=&  \left\{ x | \quad  \frac{K_0}{4}  \leq |z| \leq \frac{\epsilon_0 }{\sqrt s} e^{\frac{s}{2}}  \right\},\\
P_3(t) &=&  \left\{ x | \quad  |x| \geq  \frac{\epsilon_0}{4}  \right\} =  \left\{ x | \quad   |y| \geq \frac{\epsilon_0 e^{\frac{s}{2}}}{4} \right\} =  \left\{ x | \quad   |z| \geq \frac{\epsilon_0 }{4 \sqrt s} e^{\frac{s}{2}} \right\},
\end{eqnarray*}
with 
$$  y  = \frac{x}{\sqrt{T-t}} \text{ and  }   z = \frac{y}{ \sqrt s} = \frac{x}{ \sqrt{(T-t) |\ln(T-t)|}}.$$

In the following, let us  explain how  we derive  the positivity  condition from the various  estimate we impose   on the solution  in the $3$ regions. Then
\begin{itemize}
\item[$a)$]  In  $P_1(t)$, the  \textit{blowup region}:  In this region, we control the positivity of  $u_1$ by controlling   the positivity  of  $w_1$ (see  the similarity variables given in  \eqref{similarity-variales}).  More precisely,  as we mentioned in Subsection \ref{strategy-of-proof},    $w$  will be controlled   as a pertubation of   the profiles  $\Phi_1, \Phi_2$ (\eqref{defi-Phi-1} and  \eqref{defi-Phi-2}).  By using the positivity of  $\Phi_1 $ and a good estimate of the distance   of  $w_1$ to these profiles, we  may deduce   the positivity of $w_1,$ which  leads to the positivity of  $u_1.$
\item[$b)$] In $P_2(t)$,  the \textit{intermediate  region}:  In this region,  we control  $u$ via  a rescaled  function $U$ of  $u$  as follows: 
\begin{equation}\label{def-mathcal-U}
U (x,\xi, \tau) =  (T - t(x))^{- \frac{1}{p-1}} u( x + \xi \sqrt{T - t(x)}, t(x) + \tau (T - t(x))),
\end{equation}
where $t(x)$ is uniquely defined  for $|x| $ small enough by
\begin{equation}\label{def-t(x)}
|x| =  \frac{K_0}{4} \sqrt{(T -t(x)) \left|\ln(T -t(x))\right|}.
\end{equation}
We also introduce 
\begin{equation}\label{defini-theta-x}
\theta (x)  =   T  - t(x).
\end{equation}
We see that,  on the domain $ (\xi, \tau) \in \mathbb{R}^n \times \left[ - \frac{t(x)}{T -t(x)}, 1 \right)$, $U$ satisfies  the following  equation:
\begin{equation}\label{equation-xi-tau-U}
\partial_\tau U  = \Delta_\xi U + U^p.
\end{equation}
By using   classical parabolic estimates on $U,$ we   can  prove the following the  rescaled  $U$ at time  $\tau(x,t)$,   has a behavior similar  to  $\hat U_{K_0}(\tau(x,t)),$ for all $  |\xi|  \leq \alpha_0  \sqrt{ | \ln (T - t(x) |} $ where 
$$ \tau (x,t)  = \frac{t - t(x)}{T-t(x)},$$
and $\hat U_{K_0}(\tau)$ is unique solution of  the following ODE
\begin{equation}\label{ODE-hat-U}
 \left\{ \begin{array}{rcl}
 \partial_\tau \hat U_{K_0}  &=& \hat U^p_{K_0}(\tau),\\
 \hat U_{K_0}(0) &=& \left( p-1 + \frac{(p-1)^2 K_0^2}{64 p} \right)^{- \frac{1}{p-1}}.
\end{array}  
 \right.
\end{equation}
In particular,  we can    solve        \eqref{ODE-hat-U} with  an explicit solution:
\begin{equation}\label{solution-ODE-hat-U}
\hat U_{K_0}(\tau) = \left(  (p-1) (1 - \tau) + \frac{(p-1)^2K_0^2}{64 p}\right)^{-\frac{1}{p-1}}, \forall \tau \in [0,1).
\end{equation}
Then,     by using the positivity of  $\hat U_{K_0},$   we derive  that  $u_1  > 0,$ in this  region.
\item[$c)$] In  $P_3(t),$ the \textit{regular region}:   We control  the  solution in this region  as a perturbation of the initial data,        thanks  to  the  well-posedness property  of the Cauchy problem for equation \eqref{equ:problem},   to derive that our solution  is close to  the  initial data, 
(in fact,  $T$ will be  taken small enough). Therefore,  if the initial data  is  strictly larger than some constant, we will  derive the positivity of  $u_1$.
\end{itemize}
The above strategy makes  the real part   of our  solution  becomes  positive.  Therefore,  it remains  to control   the solution in order to get
$$ \|q_1 (.,s)\|_{L^\infty}   + \|q_2(.,s)\|_{L^\infty}  \to  +\infty, $$
(see \eqref{defini-q-1-2}).  This  part  is in fact   quite similar  to  the integer case, done in \cite{DU2017}. 

\noindent
From   the above arguments,   we give  in  the following  our definition   of the shrinking set.
\begin{definition}[A shrinking set to 0]\label{defini-shrinking-set} \textit{For all $T> 0, K_0 >0, \alpha_0 >0,  \epsilon_0 >0, A > 0, \delta_0 > 0, \eta_0 >0,  p_1 \in \left(0, \min \left( \frac{p-1}{4}, \frac{1}{2}\right) \right) $ for all $t \in [0, T)$, we define the set $S(T, K_0, \alpha_0, \epsilon_0, A, \delta_0, \eta_0, t) \subset  C( [0,t], L^\infty(\mathbb{R}^n, \mathbb{C}))$ (or $S(t)$ for short)  as follows: $u = u_1 +   i u_2 \in S(t)$ if  the  following  condition hold: 
\begin{itemize}
\item[$(i)$] \textit{Control in the blowup region $P_1(t)$:}   We have $ (q_1, q_2) (s) \in V_{p_1, K_0, A} (s)$ where  $s  =   - \ln (T-t),$ $(q_1, q_2)$      is  defined as in  \eqref{defini-q-1-2}  and     $V_{p_1, K_0,  A} (s) = V_A (s)  \in (L^\infty (\mathbb{R}^n))^2 $  is   the set of  	all function $(q_1, q_2) \in (L^\infty)^2$   such that   the following holds:
\begin{eqnarray*}
|q_{1,0} (s)| \leq \frac{A}{s^2} &\text{ and }&   |q_{2,0} (s)| \leq \frac{A^2}{s^{p_1 + 2}},\\
|q_{1,j} (s)| \leq \frac{A}{s^2} &\text{ and }&   |q_{2,j} (s)| \leq \frac{A^2}{s^{p_1 + 2}}, \forall j \leq n,\\
|q_{1,j,k} (s)| \leq \frac{A^2 \ln s}{s^2} &\text{ and }&   |q_{2,j,k} (s)| \leq \frac{A^5 \ln s}{s^{p_1 + 2}}, \forall j,k \leq n,\\
\left\| \frac{q_{1,-} (y,s)}{1 + |y|^3} \right\|_{L^\infty} \leq \frac{A}{s^{2}} &\text{ and }&  \left\| \frac{q_{2,-} (y,s)}{1 + |y|^{3}} \right\|_{L^\infty} \leq \frac{A^2}{ s^{\frac{p_1 + 5}{2}}},\\
\|q_{1,e} (.,s)\|_{L^\infty} \leq \frac{A^2}{\sqrt s} & \text{ and } &  \|q_{2,e} (.,s)\|_{L^\infty} \leq \frac{A^3}{ s^{\frac{p_1 + 2}{2}}},
\end{eqnarray*}
where    the coordinates of $q_1 $ and $ q_2$  are  introduced in  \eqref{decom-5-parts} with $r= q_1$ or $r=q_2$.
\item[$(ii)$] \textit{Control in the intermediate region $P_2(t)$:} For all $|x| \in \left[ \frac{K_0}{4} \sqrt{(T -t)|\ln(T -t)|}, \epsilon_0 \right],  \tau(x,t) = \frac{t -t(x)}{T -t(x)}$ and $|\xi | \leq \alpha_0 \sqrt{|\ln(T -t(x))|},$ we have
$$ \left|U(x,\xi,\tau(x,t)) - \hat{U}(\tau(x,t)) \right| \leq \delta_0,$$
where   $\hat U_{K_0}$ defined  in \eqref{solution-ODE-hat-U}.
\item[$iii$] \textit{Control in the regular region $P_3(t)$}: For all $|x| \geq \frac{\epsilon_0}{4}$,
$$ \left|  u (x,t) - u(x,0) \right| \leq \eta_0, \forall i =0,1.$$
\end{itemize}
Finally, we also define  the set  $S^*(T, K_0, \alpha_0, \epsilon_0, A, \delta_0, \eta_0)   \subset  C([0,T), L^\infty (\mathbb{R}^n, \mathbb{C}))$ as the set of all  $u  \in   C([0,T), L^\infty (\mathbb{R}^n, \mathbb{C}))$  such that 
$$u \in S(T, K_0, \alpha_0,  \epsilon_0,A, \delta_0, \eta_0, t), \forall  t \in [0,T).$$}
\end{definition}

The following  lemma, we  show the estimates  of the fuction   being   in $V_A(s)$ and  this  lemma is given in  \cite{DU2017}: 
\begin{lemma}\label{lemma-analysis-V-A}
For all $A \geq 1, s \geq 1,$  if we have   $(q_1, q_2) \in V_A (s)$, then the following  estimates  hold:
\begin{itemize}
\item[$(i)$] $\|q_1\|_{L^\infty (\mathbb{R}^n) } \leq \frac{C A^2}{ \sqrt s} \text{ and }  \|q_2\|_{L^\infty(\mathbb{R}^n)} \leq \frac{CA^3}{s^{\frac{p_1 + 2 }{2}}}.$
\item[$(ii)$] 
$$ |q_{1,b} (y) | \leq \frac{CA^2 \ln s}{s^2} (1 + |y|^3), \quad |q_{1,e} (y)| \leq \frac{C A^2}{s^2} (1 + |y|^3) \text{ and } 	 |q_1| \leq  \frac{C A^2 \ln s}{ s^2} (1 + |y|^3),$$
and
$$ |q_{2,b} (y) | \leq \frac{CA }{s^{\frac{p_1 +5}{2}}} (1 + |y|^3), \quad |q_{2,e} (y)| \leq \frac{C A^3}{s^{\frac{p_1 + 5}{2}}} (1 + |y|^3) \text{ and } 	 |q_2| \leq  \frac{C A^3 }{ s^{\frac{p_1 + 5}{2}}} (1 + |y|^3).$$
\item[$(iii)$]  For all $y \in \mathbb{R}^n$ we have
$$   |q_1| \leq  C  \left[   \frac{A}{s^2}(  1 + |y| )  + \frac{A^2 \ln s}{s^2} (1 + |y|^2)  +  \frac{A^2}{s^2} (1  + |y|^3) \right],$$
and 
$$ |q_2| \leq C \left[   \frac{A^2}{s^{p_1 + 2}} (1 + |y|)  +  \frac{A^5 \ln s}{s^{p_1 + 2}} (1  + |y|^2)    +  \frac{A^3}{ s^{\frac{p_1 + 5}{2}}} (1  + |y|^3)      \right] .$$
\end{itemize}
where  $C$ will henceforth  be an  constant which depends only on $K_0.$
\end{lemma}
\begin{proof}
See Lemma 3.2, given in \cite{DU2017}.
\end{proof}
 As matter of fact,  if  $ u \in S_A (t) $ then, from  item $(i)$ of  Lemma \ref{lemma-analysis-V-A},  the similarity variables \eqref{similarity-variales} and  \eqref{defini-q-1-2}, we derive the following
\begin{eqnarray}
\left\| (T -t)^{\frac{1}{p-1}} u (.,t)   -   f_0 \left( \frac{.}{\sqrt{ (T-t) |\ln (T-t) |}}\right)\right\|_{L^\infty(\mathbb{R}^n)} &\leq & \frac{CA^2}{ 1 + \sqrt{ |\ln(T-t)|}},\label{conclu-u-f-0}\\
\left\| (T -t)^{\frac{1}{p-1}} |\ln (T-t)| u_2 (.,t)   -   g_0 \left( \frac{.}{\sqrt{ (T-t) |\ln (T-t) |}}\right)\right\|_{L^\infty(\mathbb{R}^n)} &\leq & \frac{CA^3}{ 1 + |\ln (T-t)|^{\frac{p_1}{2}}}.\label{conclu-g-0-u-2}
\end{eqnarray}

We see   in the definition of  $S(t)$ that  there are many parameters, so  the dependence  of the constants on them is very important in our analysis.  We would like to mention that, we use   the notation  $C$ for these constants which depend at most    on $K_0.$ Otherwise, if the constant depends   on $K_0, A_1, A_2,...$ we will write    $C(A_1,A_2,...)$.  

We  now   prove in the following  lemma   the positivity  of $\text{Re}(u)$ at time  $t$ if  $u$ belongs  to $S(t)$ (this is a crucial  estimate  in our argument):
\begin{lemma}[The positivity of the real part of functions trapped in $S(t)$]\label{lemma-positive-real-part} For all $K_0, A \geq 1 $   $ \alpha_0 > 0,  \delta_0 < \frac{\hat U(0)}{2} , \eta_0 < \frac{1}{2},$  there exists  $\epsilon_{1} (K_0) > 0$ such that for all $\epsilon_0 \leq \epsilon_{1}$ there exists  $T_1 (A, K_0, \epsilon_0)$     such that    for all  $ T \leq T_1$ the following holds:   if  $u \in S (T, K_0, \alpha_0,  \epsilon_0,  A, \delta_0, \eta_0, t) $ for all  $ t \in [0, t_1] $ for some   $t_1 \in [0, T),$  and   $\text{ Re}(u(0)) \geq 1$ for all  $|x| \geq  \frac{\epsilon_0}{4}$, then   
 $$  \text{ Re}(u)(x,t) \geq \frac{1}{2},  \forall  x \in \mathbb{R}^n, \forall t \in [0,t_1].$$  
 \end{lemma}
 \begin{proof}
We write   that $u = u_1 + i u_2,$  with  $\text{ Re} (u) = u_1.$ Then,    we  estimate  $u_1$ on  the 3  regions $P_1 (t), P_2(t)$ and $P_3 (t)$. 

+ \textit{ The estimate in  $P_1 (t)$:}   We use the fact that  $(q_1,q_2) \in V_A (s)$ together with  item  $(i)$  in Lemma \ref{lemma-analysis-V-A}, and the  definition  \eqref{defini-q-1-2}  of  $q_1$  and the definition of  $\Phi_1$  given in \eqref{defi-Phi-1},   to derive the following: for all $|y| \leq K_0 \sqrt s,$     
$$ \left|  w_1(y,s) -  f_0 \left( \frac{y}{\sqrt s}\right)  \right| \leq  \frac{C A^2}{ \sqrt s}.$$
Using the   definition  \eqref{defi-Phi-1} of $\Phi_1,$ we write   for all $|y| \leq K_0 \sqrt s$  
\begin{eqnarray*}
w_1 (y,s) &   \geq &  f_0\left( \frac{y}{\sqrt s}\right)  - \frac{CA^2}{\sqrt s}\\
&  \geq  &   \left(  p-1   + \frac{(p-1)^2}{ 4p} K_0^2\right)^{-\frac{1}{p-1} }   - \frac{CA^2}{\sqrt s},
\end{eqnarray*}
By  definition  \eqref{similarity-variales} of the similarity variables, we    implies that:  for all $|x| \leq K_0 \sqrt{ (T-t) |\ln(T-t)|},$ 
$$ (T-t)^\frac{1}{p-1}  u_1 (x,t) \geq   \left(  p-1   + \frac{(p-1)^2}{ 4p} K_0^2\right)^{-\frac{1}{p-1} }  - \frac{CA^2}{ \sqrt{|\ln (T-t)|}}.$$
Therefore,
$$  u_1(x,t) \geq  (T-t)^{-\frac{1}{p-1}}  \left[    \left(  p-1   + \frac{(p-1)^2}{ 4p} K_0^2\right)^{-\frac{1}{p-1} }  - \frac{CA^2}{ \sqrt{|\ln (T-t)|}}  \right]  \geq \frac{1}{2},$$
provided  that   $T \leq T_{1,1} (K_0, A).$

+ \textit{ The estimate in $P_2 (t)$:}  Since we have  $u \in S (t),$  using  item $(ii)$  in the Definition  \ref{defini-shrinking-set}, we derive that:   for all $x \in \left[ \frac{K_0}{4} \sqrt{ (T-t) |\ln (T-t)|}, \epsilon_0\right]$ 
$$ \left|  U(x, 0, \tau (x,t))  - \hat U_{K_0} (\tau(x,t))   \right|  \leq \delta_0,$$  
where  $\tau(x,t) = \frac{t -t(x)}{T-t(x)} $. In particular, by using the definition  of  $t(x)$ given  in     \eqref{def-t(x)}  and   the fact that   
$$ |x| \geq   \frac{K_0}{4} \sqrt{(T-t) |\ln(T-t)|},$$
we have  $\tau(x,t) \in [0,1).$ Therefore, 
\begin{eqnarray*}
U_1(x, 0, \tau (x,t))  &\geq & \hat U_{K_0} (\tau (x,t)) - \delta_0\\
& \geq  &  \hat U_{K_0} (0) - \delta_0 \\
& \geq & \frac{1}{2} \hat U_{K_0} (0) = \frac{1}{2}\left(  p-1 + \frac{(p-1)^2}{4 p} \frac{K_0^2}{16}\right)^{-\frac{1}{p-1}},
 \end{eqnarray*}
provided  that  $\delta_0 \leq  \frac{1}{2} \hat U_{K_0} (0).$   By  definition  \eqref{def-mathcal-U}  of $U,$ this implies that    
$$ (T-t(x))^{\frac{1}{p-1}} u_1(x, t)  =  U_1 (x, 0, \tau (x,t))  \geq \frac{1}{2}\left(  p-1 + \frac{(p-1)^2}{4 p} \frac{K_0^2}{16}\right)^{-\frac{1}{p-1}}.$$
Using the definition    of  $t(x)$ in \eqref{def-t(x)}  we write  
  $$   T-t(x)  \sim   \frac{8}{K_0^2} \frac{|x|^2}{|\ln|x||}, \text{ as }  |x| \to 0.$$
Therefore, there exists $\epsilon_{1,1}(K_0) > 0$ such that  for all $ \epsilon_0 \leq  \epsilon_{1,1},$  and  for all $|x| \leq \epsilon_0,$  we have 
$$  (T -t(x))^{-\frac{1}{p-1}} \frac{1}{2} \left(  p-1 + \frac{(p-1)^2}{4p} \frac{K_0^2}{16}\right)^{-\frac{1}{p-1}} \geq \frac{1}{2}.$$ 
Then,   we conclude that     for all $|x| \in \left[  \frac{K_0}{4} \sqrt{ (T-t) |\ln(T-t)|},\epsilon_0\right],$ we have 
$$ u_1(x,t)  \geq \frac{1}{2},$$
provided  that  $T \leq  T_{2,1} (\epsilon_0)$.

+ \textit{ The estimate  in $P_3 (t)$:} This is  very easy to derive. Indeed,   item  $(iii)$ of Definition  \ref{defini-shrinking-set},  we have  for all $|x| \geq \frac{\epsilon_0}{4}$
$$ u_1 (x,t)  \geq  \text{ Re} (u)(x,0) - \eta_0  \geq 1 - \frac{1}{2}= \frac{1}{2},$$
provided that  $\eta_0 \leq \frac{1}{2}$.  This  concludes  the proof  of Lemma  \ref{lemma-positive-real-part}.
\end{proof}

Thanks to  Lemma \ref{lemma-positive-real-part},    we can handle the   singularity of the nonlinear term $u^p$ when    our solution is in  $S(T, A, \alpha_0,  \epsilon_0,A, \delta_0, \eta_0).$   In addition to that,  from  item $(i)$ of  Lemma   \ref{lemma-positive-real-part}, \eqref{conclu-u-f-0}  and  \eqref{conclu-g-0-u-2}     our problem  is reduced    to   finding   parameters $T, K_0, \alpha_0,  \epsilon_0,  A,  \delta_0, \eta_0, $ and    constructing   initial data $u(0) \in  L^\infty (\mathbb{R}^n, \mathbb{C}) $  such that the solution $u$ of  equation  \eqref{equ:problem}, exists   on $[0,T)$ and  satisfies 
\begin{equation}\label{purpose-u-t-S-t-0-T}
u  \in  S^* (T, K_0, \alpha_0,  \epsilon_0, A, \delta_0, \eta_0 ).
\end{equation}
\subsection{Preparing  initial data and the existence of a solution trapped in S(t)}
In  this subsection, we   would like  to define   initial data  $u(0)$, which depend  on some parameters to be fine-tuned in order to get a good solution. The following is our definition:
\begin{definition}[Preparing of  initial data]\label{initial-data-bar-u}
\textit{For each $ A \geq 1,  T > 0, $  $d_1= (d_{1,0}, d_{1,1}) \in \mathbb{R}^{1} \times \mathbb{R}^n, $ and  $ d_2 = (d_{2,0},d_{2,1}, d_{2,1}) \in \mathbb{R}^{ 1 + n}  \times \mathbb{R}^{\frac{n(n+1)}{2}},$   we   introduce the following functions defined at  $s_0 = - \ln T:$
\begin{eqnarray*}
\phi_{1,K_0, A,d_1} (y,s_0) &=& \frac{A}{s_0^2} \left(  d_{1,0} + d_{1,1} \cdot  y \right) \chi_0\left(\frac{16 |y|}{ K_0 \sqrt s_0}\right),	 \\
  \phi_{2, K_0, A, d_2} (y,s_0) &= &\left( \frac{A^2}{s_0^{p_1+2}} \left(   d_{2,0} +  d_{2,1} \cdot y \right) + \frac{A^5 \ln s_0}{s^{p_1+2}_0}\left(  y^\mathcal{T} \cdot d_{2,2}\cdot y - \textup{Tr }(d_{2,2}) \right) \right)\chi_0\left(\frac{16 |y|}{ K_0 \sqrt s_0}\right).
\end{eqnarray*}
We also      define    initial data    $ u_{K_0, A,d_1,d_2} (0) =   u_{1, K_0, A, d_1}(0) + i  u_{2, K_0, A,d_2} (0)$    for equation \eqref{equ:problem}    as follows:
\begin{eqnarray}
 u_{1,K_0,  A, d_1} (x,0) & = &  T^{-\frac{1}{p-1}} \left\{   \phi_{1, K_0, A,d_1}\left( \frac{x}{\sqrt{T}}, - \ln T\right)  + \Phi_1 \left(\frac{x}{ \sqrt{T}}, - \ln T  \right)  \right\} \chi_1 \left( x \right)  \label{parti-real-initial} \\
&+&U^*(x) (1  - \chi_1(x))   + 1,\nonumber\\
 u_{2, K_0,  A,  d_2} (x,0) & = &  T^{-\frac{1}{p-1}} \left\{   \phi_{2, K_0, A,d_2} \left( \frac{x}{\sqrt T}, - \ln T\right)  +    \Phi_2 \left( \frac{x}{ \sqrt{T}}, - \ln T  \right)  \right\}  \chi_1 (x ),\label{parti-imaginary-inital}
\end{eqnarray}
where $\Phi_1, \Phi_2$ are defined in \eqref{defi-Phi-1}, \eqref{defi-Phi-2}  and $\chi_1 (x) $ is defined as follows  
\begin{equation}\label{defi-chi-1-x-t-0}
\chi_1 (x)  = \chi_0 \left( \frac{|x|}{  \sqrt{T} |\ln T|}\right),
\end{equation}
  with   $\chi_0$ defined in   \eqref{defini-chi-0},    and $U^*  \in C^1(\mathbb{R}^n \backslash \{0\}, \mathbb{R})$ is defined   for all   $ x \in \mathbb{R}^n, x \neq 0 $
\begin{equation}\label{defini-U^*}
U^*(x) = \left\{  \begin{array}{rcl}
&\left[ \frac{(p-1)^2 |x|^2}{8 p |\ln |x||} \right]^{- \frac{1}{p-1}} & \quad  \text{ if }  |x| \leq C^*, \\[0.3cm]
&\frac{1}{ 1  +   |x|^2 } &   \quad  \text{ if  }   |x| \geq 1,\\[0.3cm]
 &U^*(x) > 0 &  \text{ for all } x \neq 0,
\end{array} \right.
\end{equation} 
 where   $C^*$ is a fixed constant  strictly less than $1$ enough,  and  $U^*$ satisfies the  following property:    for each  $\epsilon_0 \leq  \frac{C^*}{2}$ we have
 \begin{equation}\label{defi-proper-U-*}
 U^*(x) \leq U^*(\epsilon_0), \text{ for all }  |x| \geq \epsilon_0.
 \end{equation}}
\end{definition}
\begin{remark}
Roughly speaking, the critical data we done  here  are  superposition of  two items:
\begin{itemize}
\item[-] $T^{-\frac{1}{p-1}} \left\{  \phi_1  + \Phi_1 \right\}$ in  $P_1 (0)$
\item[-] $U^*$  in  $P_2(0).$
\end{itemize}
The first  form is  well-known  in previous construction problems. As for the second, we borrowed it from   Merle and Zaag in \cite{MZnon97}.  Note that $U^*$ is  the candidate  for the  final profile  of the  real part, as we can see   from own main result in Theorem  \ref{Theorem-profile-complex}.  More crucially, we draw your attention to the fact  that in comparision with  \cite{MZnon97}, we add here  $+1$ to the expression in \eqref{parti-real-initial}, and  this term will allow us to have  the initial condition
$$  \text{Re}(u(0))  \geq  1,$$
which is  essential to make the  nonlinear term  $u^p$ well-defined, and  the Cauchy problem solvable (see Appendix  \ref{appendix-Cauchy-problem}). This  is an important idea of ours. 
\end{remark}
From  the above definition, we  show  in the following lemma  some  rough properties of  the initial data.
\begin{lemma}\label{lemma-intial-rough}    For all  $K_0 \geq 1, A \geq 1, $  $|d_1|_\infty \leq 2,  |d_2|_\infty  \leq   2,$  and  for all $ \epsilon_0 \leq \frac{C^*}{2}$ (where  $C^*$ is introduced in \eqref{defi-proper-U-*}),  there exists $T_{2} (\epsilon_0, K_0, A) > 0$ such that   for all  $T \leq T_2,$   if $ u(0)  = u_{K_0, A,d_1,d_2}(0)$ is defined as in  Definition  \ref{initial-data-bar-u}, then the following holds:
\begin{itemize}
\item[$(i)$] The initial data   belongs to  $L^\infty$ and   satisfies the following
$$ \|u (., 0)\|_{L^\infty(|x| \geq \epsilon_0)}  \leq 1 + \left( \frac{(p-1)^2 |\epsilon_0|^2}{ 8p |\ln \epsilon_0| } \right)^{-\frac{1}{p-1}}.$$
\item[$(ii)$] The real part of the initial data, $\text{Re}(u(0))$  is positive. In particular, 
$$   \text{Re}(u(x,0))  \geq   1,   \forall   x \in \mathbb{R}^n.$$
\end{itemize} 
\end{lemma}
\begin{proof}

 \item[$(i)$]  It is obvious to see that  the initial data   belongs to  $L^\infty$ with the assumptions in this Lemma. It remains to  prove the estimate in item $(i)$. We now  take  $\epsilon_0  \leq  \frac{C^*}{2},$   and  we  use definition of $ \chi_1 $  in \eqref{defi-chi-1-x-t-0} to deduce that   $\text{supp} (\chi_1)   \subset  \{ |x| \leq  2 \sqrt{T} |\ln T|  \}$.  Moreover, we have
$$ \sqrt{T} |\ln T|  \to 0 \text{ as }  T \to 0.$$ 
Then,  we have
$$ \sqrt{ T} |\ln T | \leq \frac{\epsilon_0}{4},$$
provided  that    $T \leq  T_{2,1} (\epsilon_0)$.
Hence, 
$$\text{supp} (\chi_1)   \subset  \{ |x| \leq  \frac{\epsilon_0}{2} \},$$
  Hence,   it follows the defintion   of    $u(0)$ that:  for all $|x| \geq   \epsilon_0,$ we have 
$$ u(x,0) =    U^* (x)  + 1,  $$
 Using   \eqref{defi-proper-U-*},  our result follows. 

\item[$(ii)$]   We see in the definition  of  $u(0)$ that     we  have   $\text{supp}(\phi_{1, K_0, A, d_1})  \subset  \{ |y|  \leq \frac{K_0}{8} \sqrt{ s_0}  \}   $ and  we have the following   
$$ |\phi_{1, K_0, A, d_1}  \left( \frac{x}{\sqrt T}, - \ln T \right)|_{L^\infty}   \leq    \frac{C A}{ |\ln T|^\frac{3}{2}}. $$
 In addition to that,   in  the region  $\{ |x|  \leq  \frac{K_0}{8} \sqrt{T |\ln T|} \},$ the  function $\Phi_1\left(   \frac{x}{\sqrt T}, - \ln T \right)$   is bounded  from   below  by a positive constant which  depends only on   $K_0$. Therefore,  there exists  $T_{2,2} (A, K_0) > 0$ such that   for all   $T \leq  T_{2,2}$   for all   $|x|  \leq  \frac{K_0}{8} \sqrt{T |\ln T|}  $ we have 
$$ \phi_{1, K_0, A, d_1}  \left( \frac{x}{\sqrt T}, - \ln T \right) +\Phi_1\left(   \frac{x}{\sqrt T}, - \ln T \right)  > 0.$$
 Therefore:    for all   $|x|  \leq  \frac{K_0}{8}  \sqrt{ T |\ln T| },$  we have
$$ \text{ Re}(u(x,0)) \geq 1.$$
Now,  if    $|x| \geq \frac{K_0}{8}  \sqrt{T |\ln T|},$ then we have    $\phi_{1,K_0, A, d_1}(y,s_0) = 0.$ Since  $\Phi_1 (y,s_0) > 0$ from  \eqref{defi-Phi-1} and  $U^*(x) > 0$ from   \eqref{defi-proper-U-*}, we  directly see from  the definition \eqref{parti-real-initial} for $\text{Re}(u(0))$ that
$$   \text{ Re}(u(x,0))  \geq 1.$$
This concludes  the proof of Lemma  \ref{lemma-intial-rough}.
\end{proof}
 Following the above lemma,  we will prove that  there exists  a  domain  $\mathcal{D}_{K_0, A, s_0 },  $ with $s_0 = - \ln T$   such that  for all  $(d_1,d_2) \in  \mathcal{D}_{K_0,A,s_0},$ the initial  $u_{K_0, A,d_1,d_2}(0) $  is  trapped in 
 $$S (T, K_0, \alpha_0,  \epsilon_0, A, \delta_0, \eta_0, 0) = S(0).$$ 
 In particular, we show that the initial data   strictly  satisfies  almost the conditions of $S(0)$  except a few of  the conditions in item $(i)$ of Definition  \ref{defini-shrinking-set}.  More precisely,  these conditions concern   the following modes
   $$(q_{1,0}, (q_{1,j})_{j\leq n}, q_{2,0}, (q_{2,j})_{j \leq n}, (q_{2,j,k})_{j,k \leq n})(s_0).$$   
   The following is our lemma:
\begin{lemma}[Control of  initial data]\label{lemma-control-initial-data}
There exists $K_{3}  \geq 1$   such that  for all each  $K_0 \geq K_{3} , A \geq 1$  and $\delta_1> 0,$ there exists $\alpha_{3} (K_0, \delta_1) $ such that  for all  $\alpha_0   \leq   \alpha_{3},$   there exists  $\epsilon_{3} (K_0, \alpha_0,  \delta_1) > 0  $ such that for all  $\epsilon_0  \leq \epsilon_3 ,  \eta_0 > 0,$ there exists  $T_3(K_0,  \alpha_0,   \epsilon_1, A, \delta_1,\eta_1 )  > 0$  such that   for all  $ T  \leq  T_3$ and $ s_0 =  - \ln T$, there exists $\mathcal{D}_{K_0, A, s_0}  \subset   [-2,2]^{1 + n} \times  [-2,2]^{1+ n} \times   [-2,2]^{\frac{n(n+1)}{2}}$ such that the following holds: if  $u (0) =  u_{K_0, A_0,d_1,d_2} (0)$ (see  Definition  \ref{initial-data-bar-u}), then

\noindent
$(I)$   For all  $(d_1,d_2) \in \mathcal{D}_{K_0, A, s_0},$ we have  $ u(0) \in  S(T, K_0, \alpha_0, \epsilon_0, A, \delta_1,  \eta_0, 0)$. In particular, we have:
\begin{itemize}
\item[$(i)$]  Estimates in  $P_1(0)$:  We have  $ (q_1, q_2) (s_0) \in V_{A} (s_0)$ where  $(q_1, q_2)(s_0) $   are defined  in  \eqref{similarity-variales}  and  \eqref{defini-q-1-2}, satisfy the following estimates:
\begin{eqnarray*}
| q_{1,j,k} (s_0)| &\leq &\frac{A^2 \ln s_0}{2 s_0^2}, \forall  1  \leq j,k \leq n\\
\left\| \frac{q_{1,-} (.,s_0)}{1 + |y|^{3}} \right\|_{L^\infty} \leq \frac{A}{2s^2_0} &\text{ and }&  \left\| \frac{q_{2,-} (.,s_0)}{1 + |y|^{3}} \right\|_{L^\infty} \leq \frac{A^2 }{2s_0^{\frac{p_1 + 5}{2}}},\\
\|q_{1,e}(.,s_0)\|_{L^\infty}    \leq \frac{A^2}{2 \sqrt {s_0}} &\text{ and } &  \| q_{2,e} (.,s_0)\|_{L^\infty} \leq \frac{A^3}{2 s_0^{\frac{p_1 + 2}{2}}}.\\
\end{eqnarray*}
\item[$(ii)$] Estimates in  $P_2(0)$: For all $|x| \in \left[  \frac{K_0}{4} \sqrt{T  |\ln  T |}, \epsilon_0\right],  \tau_0 (x) = \frac{ -  t(x)}{ \theta (x)}$  with $\theta (x) = T -t(x)$  and  $|\xi| \leq  \alpha_0 \sqrt{|\ln (T - t(x))|},$ we have 
$$ |  U(x, \xi, \tau_0(x))   -  \hat U_{K_0} (\tau_0(x))| \leq \delta_{1},$$
\end{itemize}
where  $ U (x, \xi, \tau) $ is defined in  \eqref{def-mathcal-U} and $\hat U_{K_0} (\tau)$ is defined in \eqref{solution-ODE-hat-U}.

\bigskip
$(II)$ There exists  a maping  $\Psi_1$   such that    
\begin{eqnarray*}
 \Psi_1:  \mathbb{R}^{ 1 + n } \times \mathbb{R}^{ 1 + n }  \times \mathbb{R}^{\frac{n(n+1)}{2}}   &\to & \mathbb{R}^{ 1 + n } \times \mathbb{R}^{ 1 + n }  \times \mathbb{R}^{\frac{n(n+1)}{2}}\\
(d_1,  d_2) & \mapsto &  (q_{1,0}, (q_{1,j})_{j\leq n}, q_{2,0}, (q_{2, j})_{j\leq n}, (q_{2,j,k})_{j,k \leq n}  )(s_0) 
\end{eqnarray*} 
is linear, one to one  from $\mathcal{D}_{K_0, A,s_0}$ to   $\hat V_A(s_0),$ where  
 \begin{eqnarray}
\hat V_A(s) = \left[ - \frac{A}{s^2},\frac{A}{s^2} \right]^{1 + n} \times \left[ - \frac{A^2}{s^{p_1 + 2}},\frac{A^2}{s^{p_1 + 2}} \right]^{1 + n} \times \left[ - \frac{A^5 \ln s}{s^{p_1+2}}, \frac{A^5 \ln s}{s^{ p_1+2}}\right]^{\frac{n(n+1)}{2}}.\label{defini-of-hat-V-A}
\end{eqnarray} 
 Moreover, 
$$ \Psi_1 (\partial \mathcal{D}_{K_0, A,s_0}) \subset \partial  \hat V_A(s_0),$$
and 
\begin{equation}\label{degree-Psi_1-neq-0}
\text{deg} \left( \Psi_1 |_{\mathcal{D}_{K_0, A, s_0}} \right) \neq  0.
\end{equation} 
\end{lemma}
\begin{proof}
If  we forget about the terms involving $U^*$ and  the  $+1$ term in our definition  \eqref{parti-real-initial} -  \eqref{parti-imaginary-inital} of initial data, then we are  exactly  in the  framework of the case $p$ integer treated in \cite{DU2017} (see Lemma 3.4 in \cite{DU2017}).  Therefore, when $p$ is not integer, we only need to understand the effect of  $U^*$ and  the $+1$ term in order to complete the proof.  The argument is only technical. For that reason, we leave it to Appendix \ref{appendix-preparation-initial-dada}. 
\end{proof} 
Now, we   give   a key-proposition for our argument. More precisely, in the following proposition, we     prove  the  existence of a solution of equation \eqref{equation-satisfied-by-q-1-2} trapped in the shrinking set:
\begin{proposition}[Existence of a solution trapped in $S^*(T, K_0, \alpha_0, \epsilon_0, A, \delta_0, \eta_0)$]\label{pro-existence-d-1-d-1} We can chose  the parameters    $T, K_0, \alpha_0, \epsilon_0,  A$,  $  \delta_0, \eta_0$  such that  there  exist   $(d_1, d_2) $  such that  the   solution u  of equation \eqref{equ:problem} with initial data  given in Definition \ref{initial-data-bar-u}, exists on  $[0,T)$ and  satisfies 
$$    u  \in S (T,K_0, \alpha_0, \epsilon_0,  A, \delta_0, \eta_0).$$
 \end{proposition}
 \begin{proof}
The proof of  this Proposition  is   given   2 steps:
\begin{itemize}
\item  The first step: We    reduce   our problem  to a finite dimensional one.  In  other words,  we aim at proving that the  control of $u(t)$ in the shrinking set $S(t)$   reduces to  the control of the components 
$$( q_{1,0},  (q_{1,j} )_{j \leq n }  , q_{2,0},  (q_{2,j})_{j\leq n}, (q_{2,j,k} )_{j,k \leq n} )(s)$$ 
 in  $\hat V_A (s),$  defined in \eqref{defini-of-hat-V-A}.
 \item The second step:  We get the  conclusion of  Proposition \ref{pro-existence-d-1-d-1}   by using   a topological argument   in finite dimension.
\end{itemize}

\medskip
- \textit{Step 1: Reduction to a finite dimensional problem:}
Using \textit{a priori estimates},  our problem will be reduced to the control of a finite number of  components. 
\begin{proposition}[Reduction to a  finite dimensional problem]\label{pro-reduction- to-finit-dimensional}
There exist  parameters $K_0, \alpha_0,  \epsilon_0, A, \delta_0, \eta_0$ and  $  T> 0$
 such that the  following holds:
\begin{itemize}
\item[$(a)$]   Assume that    initial data    $u(0) =  u_{K_0, A, d_1,d_2} (0) $ is given  in  Definition \ref{initial-data-bar-u} with $(d_1, d_2) \in \mathcal{D}_{K_0, A, s_0}$
\item[$(b)$] Assume    furthemore    that   the solution  $u$  of equation  \eqref{equ:problem} satisfies:  $u  \in S(T, K_0,  \alpha_0, \epsilon_0,A, \delta_0, \eta_0, t)  $ for all $  t \in [0, t_*],$  for some $t_*  \in [0,T)$   and
$$  u \in  \partial S(T, K_0, \alpha_0, \epsilon_0,  A, \delta_0, \eta_0, t_*).$$
\end{itemize}

  Then, we have: 
\begin{itemize}
\item[$(i)$] (\textit{Reduction to finite dimensions}): It holds that  $( q_{1,0},  (q_{1,j} )_{j \leq n }  , q_{2,0},  (q_{2,j})_{j\leq n}, (q_{2,j,k} )_{j,k \leq n} )(s_*) \in \partial \hat V_A (s_*) ,$ where  $(q_1, q_2) (s)$    are defined  in  \eqref{similarity-variales}  and  \eqref{defini-q-1-2},  $\hat V_A(s)$ is defined as in   \eqref{defini-of-hat-V-A},  and $s_*  = - \ln (T -t_*)$.
\item[$(ii)$] (\textit{Transverse outgoing crossing}): There exists  $\nu_0 > 0$ such that 
\begin{equation}\label{traverse-outgoing crossing}
\forall \nu \in (0,\nu_0), ( q_{1,0},  (q_{1,j} )_{j \leq n }  , q_{2,0},  (q_{2,j})_{j\leq n}, (q_{2,j,k} )_{j,k \leq n} )(s_*+\nu) \notin \hat V_A (s_*+\nu),
\end{equation}
which implies  that there   exists   $\nu_1 > 0$ such that    $u$ exists  on  $[0, t_* +  \nu_1)$  and         for all  $\nu \in (0, \nu_1)   $ 

$$u(t_* + \nu) \notin   S (T, K_0, \alpha_0,  \epsilon_0,  A, \delta_0, \eta_0, t _*+ \nu ).$$
\end{itemize}
\end{proposition}
   The proof of this Lemma     uses  techniques   given  in \cite{MZnon97}   which  were developed      from    \cite{BKnon94} and   \cite{MZdm97} in  the real case.   However,  it is true  that   our shrinking set involves   more conditions    than   the shrinking set  used  in    \cite{BKnon94},  \cite{MZdm97},  \cite{DU2017}.   In fact, the additional  conditions are useful to     ensure that   our solution always stays    positive.  In particular,   the set $V_A(s)$ plays an important role.     Indeed,   as for the integer case  in \cite{DU2017},   only  the nonnegative   modes  
   $ ( q_{1,0},  (q_{1,j} )_{j \leq n }  , q_{2,0},  (q_{2,j})_{j\leq n}, (q_{2,j,k} )_{j,k \leq n} )(s_*)$
may  touch  the boundary  of   $\hat V_A (s_*)$ and   leave  in short time later.      However,  the control     of the sulution with the   positive   real part  is also   our   highlight and   of course	 it is the main difficulty in our work.   This proposition  makes the heart of the paper and needs many steps to be proved. For that reason, we dedicate  a whole  section to its proof (Section \ref{the proof of proposion-reduction-finite-dimensional} below).  Let us admit it here, and       get to  the conclusion of     Proposition \ref{pro-existence-d-1-d-1} in the second step.
 
 \medskip
\textit{ - Step 2: Conclusion of Proposition \ref{pro-existence-d-1-d-1} by a topological argument.}
In this step, we  give  the  proof  of    Proposition \ref{pro-existence-d-1-d-1}  assuming  that Proposition \ref{pro-reduction- to-finit-dimensional} holds. In  fact,   we aim at  proving  the existence of a parameter $(d_1,d_2) \in \mathcal{D}_{K_0, A,s_0}$ such that  the solution  $ u$ of equation \eqref{equ:problem}   with   initial  data $ u_{K_0, A, d_1, d_2}(0)$ (given in Definition \ref{initial-data-bar-u}),  exists  on  $[0,T)$ and satisfies
$$  u  \in  S^*(T, K_0, \alpha_0,  \epsilon_0, A, \delta_0, \eta_0),$$
where   the parameters  will be  suitably chosen.  Our  argument is  analogous  to the   argument of  Merle and Zaag in \cite{MZdm97}.  For that reason,   we only give  a brief proof. Let us fix $T,K_0, \delta_0, \alpha_0, \epsilon_0,   A, \alpha_0, \eta_0 $  such that Lemma  \ref{lemma-control-initial-data}, Proposition \ref{pro-reduction- to-finit-dimensional}  and   Lemma  \ref{lemma-positive-real-part}  hold. Then,  for all  $(d_1, d_2)  \in \mathcal{D}_{K_0, A, s_0}$ and   from   Lemma  \ref{lemma-control-initial-data} we have  the initial  data  
$$ u_{K_0, A, d_1, d_2} (0)  \in     S(T,K_0,\alpha_0,  \epsilon_0, A,   \delta_0,  \eta_0, 0).$$
Thanks to  Lemmas  \ref{lemma-positive-real-part} and   \ref{lemma-control-initial-data},    for each  $(d_1, d_2)  \in \mathcal{D}_{K_0, A, s_0}$ we can define   $t_*(d_1,d_2) \in [0,T)$ as the maximum time  such that    the solution    $u_{d_1,d_2}$  of equation   \eqref{equ:problem},   with  initial data   $u_{K_0, A, d_1, d_2} (0)$     trapped in  $S(T,K_0, \alpha_0, \epsilon_0, A, \delta_0, \eta_0,t)$ for all   $t\in [0,t_*(d_1,d_2)).$  We have the  two  following cases:

\noindent
+ Case 1:  If  there exists   $(d_1,d_2)$ such that   $t_*(d_1,d_2) = T$ then our   problem  is  solved 

\noindent
+  Case 2:     For all  $(d_1,d_2) \in \mathcal{D}_{K_0, A, s_0}, $   we have
$$       t_*(d_1,d_2)    < T. $$ 

\noindent
By  contradiction,  we can  prove that  the  second case   can  not   occur. Indded,    if   it is true, by using the continuity of  the solution  $u$ in time and the definition   of   $t_*=t_* (d_1,d_2),$ we can deduce that    $u  \in \partial S(t_*).$  Using    item  $(i)$ of Proposition \ref{pro-reduction- to-finit-dimensional},      we derive  
$$( q_{1,0},  (q_{1,j} )_{j \leq n }  , q_{2,0},  (q_{2,j})_{j\leq n}, (q_{2,j,k} )_{j,k \leq n} )(s_*) \in \partial \hat V_A(s_*), $$
where   $s_*  =  - \ln (T - t_*).$
\noindent
Then,  the    following mapping $\Gamma$ is well-defined: 
\begin{eqnarray*}
\Gamma : \mathcal{D}_{K_0, A, s_0} & \to & \partial \left(  [-1, 1]^{1 + n} \times [-1, 1]^{1 + n}  \times [-1, 1]^{\frac{n (n+1)}{2}}\right)\\
(d_1,d_1)  &\mapsto  & \left(  \frac{s_*^2}{A}(q_{1,0}, (q_{1,j})_{j\leq n})(s_*), \frac{s_*^{p_1 + 2}}{A^2}  (q_{2,0}, (q_{2,j})_{j\leq n})(s_*),
 \frac{s_*^{p_1+2}}{A^5 \ln s_*}  (q_{2,j,k})_{j,k \leq n }(s_*)\right).
\end{eqnarray*}
Moreover, it       satisfies  the two  following properties:
\begin{itemize}
\item[(i)] $\Gamma$ is continuous from $\mathcal{D}_{K_0, A, s_0}$ to $ \partial \left(   [-1, 1]^{1 + n} \times [-1, 1]^{1 + n}  \times [-1, 1]^{\frac{n (n+1)}{2}}\right).$  This is a consequence  of item $(ii)$ in Proposition \eqref{pro-reduction- to-finit-dimensional}.
\item[(ii)] The degree  of the   restriction $\Gamma \left|\right. _{\partial \mathcal{D}_{A,s_0}}$ is   non zero.  
Indeed, again  by item $(ii)$ in Proposition \ref{pro-reduction- to-finit-dimensional}, we have 
$$ s^* (d_1,d_2) = s_0,$$
in this case. Applying  \eqref{degree-Psi_1-neq-0}, we get the conclusion.
\end{itemize}
In fact, such  a mapping $\Gamma$  can not  exist by Index theorem and  this is a contradiction. Thus,   Proposition \ref{pro-existence-d-1-d-1}  follows, assuming  that  Proposition  \ref{pro-reduction- to-finit-dimensional} holds   (see Section  \ref{the proof of proposion-reduction-finite-dimensional} for the proof of latter).
\end{proof}
\subsection{ The proof of Theorem \ref{Theorem-profile-complex}} 
In this section, we aim at giving the  proof of Theorem \ref{Theorem-profile-complex} by using  Proposition \ref{pro-existence-d-1-d-1}.

\medskip
The  proof of Theorem \ref{Theorem-profile-complex}: Except for the treatment of the nonlinear term, this part is quite similar to what we did  in \cite{DU2017}  when $p$ is integer. Nevertheless, for the reader's convenience, we give the proof here, insisting on the way we handle the nonlinear term.

\noindent
+ \textit{The proof of  item $(i)$ of Theorem \ref{Theorem-profile-complex}:} 
 Using Proposition \ref{pro-existence-d-1-d-1},  there exist   $(d_1, d_2) $ such that  the solution   $u$ of equation  \eqref{equ:problem} with     initial data  $u_{K_0, A, d_1, d_2}(0)$ (given in Definition \ref{initial-data-bar-u}),  exists     on $[0, T)$ and   satisfies:  
 $$ u  \in S^* (T, K_0,  \alpha_0,\epsilon_0, A, \delta_0, \eta_0).$$
 Thanks to  item  $(i)$  in Definition \ref{defini-shrinking-set},  item $(i)$  of Lemma  \ref{lemma-analysis-V-A}, and definition  \eqref{similarity-variales} and   definition \eqref{defini-q-1-2} of  $(w_1,w_2)$ and   $(q_1, q_2)$  we conclude    \eqref{esttima-theorem-profile-complex} and \eqref{estima-the-imginary-part}.  In addition to that we have  $\text{Re}(u) > 0$.     Moreover,  we use again  the  definition of    $V_A(s)$ to conclude the following asymptotics:
 \begin{eqnarray}
 u(0,t)  & \sim &  \kappa (T -t)^{-\frac{1}{p-1}},\label{u-0-t-sim}\\
 u_2(0,t) &\sim  &     - \frac{2n \kappa}{ (p-1)  }  \frac{(T-t)^{-\frac{1}{p-1}}}{ |\ln(T-t)|^2},\label{u-2-0-t-sim}
 \end{eqnarray}
as   $t \to  T$, 
which means  that  $u$ blows up  at time $T$ and   the origin is a blowup point.    Moreover,   the real and   imaginary parts simultaneously  blow up   .  It remains  to prove that  for all $x \neq 0,$  $x$ is not a blowup point of $u$. The following Lemma  allows us  to conclude. 
 \begin{lemma}[No blow-up  under  some  threshold; Giga and Kohn \cite{GKcpam89}]\label{lemma-no-blowup-solution}
For all $C_0 > 0, 0 \leq  T_1 <  T$ and $\sigma>0$ small enough,  there exists $\epsilon_0 (C_0,T, \sigma)> 0$  such that  if   $u(\xi, \tau)$ satisfies the following  estimates    for all $|\xi| \leq   \sigma, \tau \in \left[T_1,T\right)$:
 $$ \left| \partial_\tau u - \Delta u  \right| \leq C_0 |u|^p,$$
 and 
 $$ |u(\xi,\tau)|       \leq \epsilon_0 (1 -\tau)^{-\frac{1}{p-1}}.$$
 Then,  $u$ does not  blow up at $\xi = 0, \tau = T$. 
 \end{lemma}
 \begin{proof}
  See   Theorem   2.1     in Giga and Kohn   \cite{GKcpam89}.        Although  the proof  of  \cite{GKcpam89} was  given in the  real case, it extends naturally  to the complex  valued  case.    
 \end{proof}
 
 \noindent We next use   Lemma  \ref{lemma-no-blowup-solution}  to conclude that $u$ does not blow up at  $x_0 \neq  0.$   Since  from   \eqref{estima-the-imginary-part}, we have  
 $$ (T-t)^{-\frac{1}{p-1}} \|u_2(.,t)\|_{L^\infty}  \leq   \frac{C}{|\ln(T-t)|},$$
   if   $x_0 \neq 0$ we use  \eqref{esttima-theorem-profile-complex}   to deduce the following: 
 \begin{equation}\label{estima-u-x-0-neq-0}
 \sup_{|x - x_0| \leq \frac{|x_0|}{2}}   (T - t)^{\frac{1}{p -1}} | u(x,t) |  \leq  \left|  f_0 \left(   \frac{ \frac{|x_0|}{2}}{  \sqrt{{(T -t )}|\ln (T -t)| } }\right)  \right| + \frac{C}{ \sqrt{ | \ln (T - t)|}} \to 0, \text{ as }  t \to T. 
 \end{equation}
 Applying Lemma  \ref{lemma-no-blowup-solution}  to $u(x - x_0, t),$ with some $\sigma$ small enough such that $\sigma \leq \frac{|x_0|}{2},$ and $T_1$  close enough to $T,$ we see that $u(x - x_0, t)$ does not blow up at time   $T$ and  $x = 0$. Hence, $x_0 $ is not a blow-up point of  $u$. This concludes  the proof of item $(i)$ in Theorem \ref{Theorem-profile-complex}.  
 
 \medskip
+  \textit{The proof of item $(ii)$ of Theorem \ref{Theorem-profile-complex}:}
 Here,  we   use the argument  of Merle in   \cite{Mercpam92}    to deduce  the existence of  $u^* = u_1^* + i u_2^*$   such that   $u(t) \to u^* $  as  $t \to T$ uniformly on compact sets of $\mathbb{R}^n \backslash \{0\}$. In addition to that, we use   the techniques in  Zaag \cite{Zcpam01}, Masmoudi and Zaag  \cite{ MZjfa08}, Tayachi and Zaag \cite{TZpre15} for the proofs of \eqref{asymp-u-start-near-0-profile-complex} and \eqref{asymp-u-start-near-0-profile-complex-imaginary-part}. 
 
\noindent 
 Indeed,   for all   $x_0 \in \mathbb{R}^n , x_0 \neq 0 $, we deduce   from   \eqref{esttima-theorem-profile-complex},  \eqref{estima-the-imginary-part}   that  not only  \eqref{estima-u-x-0-neq-0} holds but also the following is  satisfied:
 \begin{eqnarray}
  \sup_{|x - x_0| \leq \frac{|x_0|}{2}}   (T - t)^{\frac{1}{p -1}} |\ln(T -t)|| u_2(x,t) |  &\leq & \left|  \frac{9|x_0|^2}{4 (T -t ) |\ln (T -t)|    } f_0 ^p\left(  \frac{ \frac{|x_0|}{2}}{  \sqrt{{(T -t )}|\ln (T -t)| } }\right)  \right|\label{estimates-T-t-u-leq-epsilon}  \\
  &+&  \frac{C}{  | \ln (T - t)|^{\frac{p_1}{2}}} \to 0,\nonumber \text{ as } t \to T.
 \end{eqnarray}
  We now consider  $x_0$ such that $|x_0|$ is  small enough,   and $K$ to be fixed later. We   define  $t_0(x_0)$  by 
  \begin{equation}\label{x-0-leq-delta-t-x-0}
  |x_0|   =  K \sqrt{ (T -t_0(x_0)) |\ln (T -t_0(x_0))|} . 
  \end{equation}
Note that  $t_0 (x_0)$  is unique  when  $|x_0|$  is small enough and $t_0 (x_0)\to T$  as  $x_0 \to 0$.  We  introduce  the rescaled functions  $U(x_0, \xi, \tau)$  and   $V_2 (x_0, \xi, \tau)$        as follows:
   \begin{equation}\label{equa-upsilon-xi-tau}
U (x_0, \xi, \tau)  = \left( T - t_0 (x_0)\right)^{\frac{1}{p-1}} u(x,t).
\end{equation}
and  
\begin{equation}\label{defini-V-2-x-0-xi-tau}
V_2 (x_0, \xi, \tau)  = |\ln (T- t_0(x_0))| U_2 (x_0, \xi, \tau),
\end{equation}
where  $U_2 (x_0, \xi, \tau)$   is  defined by
$$ U (x_0, \xi, \tau)  = U_1 (x_0, \xi, \tau) + i U_2 (x_0, \xi, \tau),$$
and  
\begin{equation}\label{relation-x-and-xi-tau}
(x,t) = \big(x_0 + \xi\sqrt{T - t_0(x_0)},t_0(x_0) + \tau (T - t_0(x_0))\big), \text{ and } (\xi, \tau) \in \mathbb{R}^n \times \left[ - \frac{t_0(x_0)}{T - t_0(x_0)}, 1 \right).
\end{equation}
We can see that  with these notations, we derive from item $(i)$ in Theorem \ref{Theorem-profile-complex} the following estimates for initial data at $\tau = 0 $ of $U$ and $V_2$
\begin{eqnarray}
\sup_{|\xi| \leq   |\ln(T - t_0(x_0))|^{\frac{1}{4}}} \left| U(x_0, \xi, 0) - f_0(K_0)\right| & \leq & \frac{C}{ 1 + (|\ln(T - t_0(x_0))|^{\frac{1}{4}})} \to 0 \quad \text{ as } x_0 \to 0,\label{condition-initial-K-0-f-0}\\
\sup_{|\xi| \leq  |\ln(T - t_0(x_0))|^{\frac{1}{4}}} \left| V_2(x_0, \xi, 0)- g_0(K_0)\right| &\leq &\frac{C}{ 1 + (|\ln(T - t_0(x_0))|^{ \gamma_1})} \to 0 \quad \text{ as } x_0 \to 0.\label{condition-initial-K-0-g-0}
\end{eqnarray}
where $f_0 (x), g_0 (x)$ are defined as in  \eqref{defini-f-0} and \eqref{defini-g-0-z} respectively, and  $\gamma_1 = \min \left(  \frac{1}{4} ,   \frac{p_1}{2} \right) $. Moreover, using equations \eqref{equation-satisfied-u_1-u_2}, we derive the following equations for $U, V_2$: for all   $\xi \in \mathbb{R}^n, \tau \in\left[ 0, 1 \right)$ 
\begin{eqnarray}
\partial_\tau U &= & \Delta_\xi  U +  U^p,\label{equa-U-x-0-xi-tau}\\
\partial_\tau V_2 &= & \Delta_\xi V_2 + \left|  \ln (T-t_0 (x_0))\right| F_2 (U_1, U_2), \label{equa-V-2-x-0-xi-tau}
\end{eqnarray}
where   $F_2$ is defined in \eqref{defi-mathbb-A-1-2}.

\noindent
 Besides that,  from  \eqref{estima-u-x-0-neq-0}  and \eqref{equa-U-x-0-xi-tau}, we can apply Lemma  \ref{lemma-no-blowup-solution}   to $U$ when $|\xi| \leq |\ln (T - t_0(x_0))|^{\frac{1}{4}}$  and obtain:
\begin{equation}\label{bound-U-xi-tau-x-0}
\sup_{|\xi| \leq \frac{1}{2}|\ln (T -t_0(x_0))|^\frac{1}{4}, \tau \in [0,1) }  |U (x_0, \xi,\tau)|     \leq C. 
\end{equation}
and we  aim at  proving    for $V_2 (x_0, \xi, \tau)$  that
\begin{equation}\label{bound-V-2-xi-tau-x-0}
\sup_{|\xi| \leq  \frac{1}{16}|\ln (T -t_0(x_0))|^\frac{1}{4}, \tau \in [0,1) }  |V_2(x_0, \xi,\tau)|     \leq C.
\end{equation}
+  \textit{The proof for \eqref{bound-V-2-xi-tau-x-0}:}  We first use \eqref{bound-U-xi-tau-x-0} to derive  the following rough estimate:
\begin{equation}\label{estima-V-2-1-step-0}
\sup_{|\xi| \leq \frac{1}{2} |\ln (T -t_0(x_0))|^\frac{1}{4}, \tau \in [0,1) }  |V_2(x_0, \xi,\tau)|     \leq C |\ln(T -t_0(x_0))|.
\end{equation}
We first introduce $\psi $ a cut-off function $\psi \in C^\infty_0 (\mathbb{R}^n), 0 \leq  \psi \leq 1, supp(\psi ) \subset B(0,1), \psi = 1   $ on $B( 0, \frac{1}{2}).$  Introducing
\begin{equation}\label{U-2-1-psi-1-x-0}
\psi_1 (\xi) = \psi \left(  \frac{2\xi}{  |\ln (T -t_0 (x_0))|^{\frac{1}{4}}} \right) \text{ and  }  V_{2,1} (x_0, \xi, \tau) = \psi_1 (\xi) V_2 (x_0,\xi, \tau).
\end{equation}
Then, we deduce from  \eqref{equa-V-2-x-0-xi-tau}  an  equation satisfied by $V_{2,1}$
\begin{equation}\label{equa-V-2-1-x-0}
\partial_\tau  V_{2,1}  =  \Delta_\xi V_{2,1} - 2 \text{ div} (V_2 \nabla \psi_1)  + V_2 \Delta \psi_1 + |\ln (T-t_0(x_0))|\psi_1 F_2 (U_1,U_2) .
\end{equation}
Hence, we can write $V_{2,1} $ with a integral  equation as follows
\begin{equation}\label{equa-integral-V-2-1}
V_{2,1} (\tau)  =   e^{\Delta \tau} (V_{2,1}(0)) +  \int_0^\tau e^{(\tau - \tau')\Delta} \left(  - 2 \text{ div } (V_2 \nabla \psi_1) + V_2 \Delta\psi_1 +  |\ln(T-t_0(x_0))|  \psi_1 F_2(U_1, U_2))(\tau')  \right) d \tau'.
\end{equation}
Besides that, using \eqref{bound-U-xi-tau-x-0} and \eqref{estima-V-2-1-step-0} and  the fact  that 
\begin{eqnarray*}
| \nabla \psi_1|  \leq  \frac{C}{ | \ln (T -t_0(x_0))|^{\frac{1}{4}}}, 
| \Delta \psi_1|  \leq  \frac{C}{ | \ln (T -t_0(x_0))|^{\frac{1}{2}}},
\end{eqnarray*}
we deduce that
\begin{eqnarray*}
\left|   \int_0^\tau e^{(\tau - \tau')\Delta} \left(  - 2 \text{ div } (V_2 \nabla \psi_1) \right) d\tau' \right| &\leq & C \int_{0}^\tau \frac{\| V_2 \nabla \psi_1\|_{L^\infty} (\tau')}{ \sqrt { \tau - \tau'}}  d\tau'       \leq C |\ln (T - t_0 (x_0))|^{\frac{3}{4}},\\
\left|   \int_0^\tau e^{(\tau - \tau')\Delta} \left(  V_2 (\tau') \Delta \psi_1  \right) d\tau' \right| &\leq & C  \int_0^\tau  \| V_2 \Delta \psi_1\|_{\infty}   (\tau') d\tau' \leq C |\ln (T -t_0 (x_0))|^\frac{1}{2}, \\
\left|   \int_0^\tau e^{(\tau - \tau')\Delta} \left(   \psi_1  |\ln (T-t_0(x_0))| F_2(U_1, U_2)(\tau')   \right) d\tau' \right|   &\leq  & C \int_0^\tau \|  |\ln(T-t_0(x_0))|  \psi_1 F_2 (U_1,U_2)\|_{L^\infty} (\tau') d\tau'.
\end{eqnarray*}
   Since  the last term  in    \eqref{equa-integral-V-2-1}     involves   the nonlinear term $F_2(U_1,U_2),$ we need to handle  it differently  from the case  where $p$ is integer: using  the definition \eqref{defi-mathbb-A-1-2}  of   $F_2,$  and \eqref{bound-U-xi-tau-x-0} and    the fact that   $U_1$  is positive,    we write    from  for all   $|\xi| \leq \frac{1}{2} |\ln(T -t_0(x_0))|^\frac{1}{4}, \tau \in [0,1)$  we have
 \begin{eqnarray*}
 |\psi_1  \ln (T-t_0(x_0))  F_2 (U_1, U_2)(\tau)    |  \leq   C \left( U_1^2  + U_2^2 \right)^\frac{p-1}{2} |\psi_1  \ln (T-t_0(x_0))  U_2(\tau)| \leq   C \|V_{2,1}(\tau) \|_{L^\infty}.
 \end{eqnarray*}
Hence, from  \eqref{equa-integral-V-2-1} and the above estimates, we derive
$$ \| V_{2,1}(\tau)\|_{L^\infty} \leq C | \ln (T -t_0 (x_0)) |^{\frac{3}{4}} +  C \int_0^\tau   \|V_{2,1}(\tau')\|_{L^\infty} d \tau'. $$
Thanks to Gronwall Lemma,  we deduce that 
$$ \|V_{2,1} (\tau)\|_{L^\infty}  \leq C |\ln(T -t_0(x_0))|^{\frac{3}{4}}, \forall  \tau \in [0,1),$$
which yields
\begin{equation}\label{estima-V-2-1-step-1}
\sup_{|\xi| \leq \frac{1}{4} |\ln (T -t_0(x_0))|^\frac{1}{4}, \tau \in [0,1) }  |V_2(x_0, \xi,\tau)|     \leq C |\ln(T -t_0(x_0))|^{\frac{3}{4}}.
\end{equation}
We apply  iteratively for 
$$ V_{2,2} (x_0, \xi, \tau) =   \psi_2 (\xi) V_{2} (x_0,\xi, \tau)  \text{ where }  \psi_2 (\xi)  = \psi \left( \frac{4 \xi}{ |\ln(T -t_0 (x_0))|^{\frac{1}{4}}} \right).$$
Similarly, we  deduce that
$$ \sup_{|\xi| \leq \frac{1}{8} |\ln (T -t_0(x_0))|^\frac{1}{4}, \tau \in [0,1) }  |V_2(x_0, \xi,\tau)|     \leq C |\ln(T -t_0(x_0))|^{\frac{1}{2}}.$$
We apply this  process a  finite   number of  steps to  obtain \eqref{bound-V-2-xi-tau-x-0}.  We now come back to our problem, and aim at proving that:  
\begin{eqnarray}
\sup_{|\xi| \leq \frac{1}{16}  |\ln(T - t_0(x_0))|^{\frac{1}{4}}, \tau \in [0,1)} \left| U (x_0, \xi, \tau) - \hat U_{K_0} (\tau) \right| &\leq & \frac{C}{ 1 + |\ln (T- t_0(x_0) )|^{\gamma_2}}, \label{sup-v-xi-tau-apro-1}\\
\sup_{|\xi| \leq \frac{1}{32}|\ln(T - t_0(x_0))|^{\frac{1}{4}}, \tau \in [0,1)} \left| V_2 (x_0, \xi, \tau) - \hat V_{2,K_0} (\tau) \right| &\leq & \frac{C}{ 1 + |\ln (T- t_0(x_0) )|^{\gamma_3}}, \label{sup-V-2-xi-tau-apro-1}
\end{eqnarray}
where  $\gamma_2, \gamma_3$ are positive small enough  and   $( \hat U_{K_0} , \hat V_{2,K_0}) (\tau)  $  is  the  solution of  the following system:
\begin{eqnarray}
\partial_\tau \hat U_{K_0} &=& \hat U_{K_0}^p,\label{ODE-hat-U-K-0}\\
 \partial_\tau \hat V_{2, K_0} &=& p \hat U_{K_0}^{p-1} \hat V_{2,K_0}\label{ODE-hat-U-K-0}. 
\end{eqnarray}
with initial data  at $\tau = 0$
\begin{eqnarray*}
\hat U_{K_0} (0) &=& f_0 (K_0),\\
\hat V_{2,K_0} (0) &=& g_0 (K_0). 
\end{eqnarray*}
given by 
\begin{eqnarray}
\hat U_{K_0} (\tau) &=& \left(  (p-1) (1 - \tau)  + \frac{(p-1)^2 K_0^2}{4 p}\right)^{-\frac{1}{p-1}}  ,\label{defini-hat-U-K-0-tau}\\
\hat V_{2,K_0} (\tau) &=& K_0 ^2 \left( (p-1) (1 - \tau) + \frac{(p-1)^2 K_0^2}{4 p}\right)^{-\frac{p}{p-1}}.  \label{defini-hat-V-2-K-0-tau}
\end{eqnarray}
for all $\tau \in [0,1)$.  The proof of  is cited to Section 5 of Tayachi and Zaag \cite{TZpre15} and, here  we  will use  \eqref{sup-v-xi-tau-apro-1}   to prove    \eqref{sup-V-2-xi-tau-apro-1}.   For the  reader's  convenience, we give it here.  Let us consider 
\begin{equation}\label{defini-mathcal-V-2}
\mathcal{V}_2 = V_2 - \hat V_{2,K_0} (\tau).
\end{equation}
  Using  \eqref{bound-V-2-xi-tau-x-0}, we deduce the following   
\begin{equation}\label{estima-mathcal-V-2}
\sup_{|\xi| \leq \frac{1}{ 16} | \ln (T - t_0(x_0))|^{\frac{1}{4}}, \tau \in [0,1)} | \mathcal{V}_{2}| \leq C.
\end{equation}
 In addition to that, from   \eqref{equa-V-2-x-0-xi-tau}  we write    an equation  on  $\mathcal{V}_2$ as follows:
\begin{equation}\label{equat-mathcal-V-2}
\partial_\tau \mathcal{V}_2 = \Delta \mathcal{V}_2 +  p \hat U_{K_0}^{p-1} \mathcal{V}_2 +  p (U_1^{p-1}  - \hat U_{K_0}^{p-1} ) V_2 + \mathcal{G}_2(x_0, \xi, \tau),
\end{equation}
where 
$$ \mathcal{G}_2 (x_0, \xi,\tau) =   |\ln(T-t_0(x_0))| \left(  F_2 (U_1,U_2)  -  p  U_1^{p-1}  U_2\right).$$
As for  the last term  in  \eqref{equat-mathcal-V-2}, we need  here to carefully handle  this expression, sine it involves  a nonlinear term, which needs a  treatment  different  from the case where  $p$ is integer. 
From the definition   \eqref{defi-mathbb-A-1-2}  of $F_2,$   we have
\begin{eqnarray*}
\left| F_2 (U_1,U_2)  - p  U_1^{p-1} U_2 \right|  & \leq &  \left|    p U_2 \left( (U_1^2  + U_2^2)^{\frac{p-1}{2}}    - U_1^{p-1} \right) \right| \\
&  + &  \left|  (U_1^2  + U_2^2)^{\frac{p}{2}} \left\{    \sin \left(    p  \arcsin \left(  \frac{U_2}{\sqrt{U_1^2 + U_2^2}} \right)    \right)  - \frac{p U_2}{\sqrt{U_1^2  + U_2^2}}\right\} \right|.
\end{eqnarray*} 
And we   deduce  from  \eqref{bound-V-2-xi-tau-x-0} and     \eqref{sup-v-xi-tau-apro-1}    with  $\epsilon_0 > 0$ small  enough that 
$$  \left| F_2 (U_1,U_2)  - p  U_1^{p-1} U_2 \right|   \leq   C |U_2|^3,$$
 Plugging     the above   estimate  and      using   \eqref{defini-V-2-x-0-xi-tau} and \eqref{bound-V-2-xi-tau-x-0}, we have the following    
\begin{equation}\label{estima-mathcal-G-2}
  \sup_{|\xi| \leq  \frac{1}{16} |\ln(T -t_0)|^{\frac{1}{4}}, \tau \in [0,1) } | \mathcal{G}_2 (x_0, \xi, \tau)| \leq   \frac{C}{  | \ln (T -t_0 (x_0))|^2}.
\end{equation}
Introducing 
$$\bar{\mathcal{V} }_2 = \psi_* (\xi) \mathcal{V}_2,$$
where 
$$ \psi_* = \psi \left( \frac{1 6 \xi}{ |\ln (T- t_0 (x_0))|^{\frac{1}{4}}}\right),$$
and $ \psi  $ is the cut-off function  which has been introduced  above. We also note that $\nabla \psi_*, \Delta \psi_*$ satisfy the following estimates
\begin{equation}\label{estima-psi-nabla-xi-psi-x-0}
\| \nabla_\xi  \psi_* \|_{L^\infty} \leq \frac{C}{ |\ln (T- t_0 (x_0))|^{\frac{1}{4}}}  \text{ and }  \| \Delta_\xi  \psi_*\|_{L^\infty} \leq \frac{C}{ |\ln (T- t_0 (x_0))|^{\frac{1}{2}}}.
\end{equation}
In particular,  $\bar{\mathcal{V}}_2$ satisfies
\begin{equation}\label{euqa-bar-mathcal-V-2}
\partial_\tau   \bar {\mathcal{V}}_2  =  \Delta \bar {\mathcal{V}}_2 + p \hat U_{K_0}^{p-1} (\tau)  \bar {\mathcal{V}}_2   - 2 \text{ div } (\mathcal{V}_2 \nabla \psi_*) + \mathcal{V}_2 \Delta \psi_* + p (U_1^{p-1} - \hat U_{K_0}^{p-1}) \psi_* V_2 + \psi_* \mathcal{G}_2,
\end{equation}

By Duhamel principal,  we derive the following integral  equation 
\begin{equation}\label{duhamel-bar-mathcal-V-2}
\bar{ \mathcal{V}}_2 (\tau)  = e^{\tau \Delta} (\bar{ \mathcal{V}}_2 (\tau)  ) + \int_0^\tau  e^{(\tau - \tau')\Delta} \left( p \hat U_{K_0}^{p-1}   \bar {\mathcal{V}}_2   - 2 \text{ div } (\mathcal{V}_2 \nabla \psi_*) + \mathcal{V}_2 \Delta \psi_* + p (U_1^{p-1} - \hat U_{K_0}^{p-1}) \psi_* V_2 + \psi_* \mathcal{G}_2   \right) (\tau') d\tau'.
\end{equation}
Besides that, we use    \eqref{sup-v-xi-tau-apro-1},  \eqref{defini-hat-U-K-0-tau},   \eqref{estima-mathcal-V-2},   \eqref{estima-psi-nabla-xi-psi-x-0},   \eqref{estima-mathcal-G-2} to derive the  following estimates: for all $\tau \in [0,1)$ 
\begin{eqnarray*}
|\hat U_{K_0} (\tau )|  & \leq & C ,\\
\| \mathcal{V}_2 \nabla \psi_* \|_{L^\infty} (\tau)  & \leq  & \frac{C }{ |\ln (T -t_0(x_0))|^{\frac{1}{4}}},\\
\| \mathcal{V}_2 \Delta \psi_* \|_{L^\infty} (\tau)  & \leq  & \frac{C }{ |\ln (T -t_0(x_0))|^{\frac{1}{2}}},\\
\left\|  \left(U_1^{p-1} - \hat U_{K_0}^{p-1} \right) \psi_*   \right\|_{L^\infty} (\tau)& \leq & \frac{C}{ | \ln(T -t_0(x_0))|^{\gamma_2}},  \\
\| \mathcal{G}_2 \psi_*\|_{L^{\infty}} & \leq & \frac{C}{ |\ln (T -t_0(x_0))|^2 }.
\end{eqnarray*}
where $\gamma_2$ given in \eqref{sup-v-xi-tau-apro-1}. Hence, we  derive from the above estimates that:  for all $ 0 \leq  \tau' < \tau  < 1$
\begin{eqnarray*}
|  e^{(\tau - \tau')\Delta}p \hat U_{K_0}^{p-1}   \bar {\mathcal{V}}_2  (\tau') |  & \leq  &C \|\bar {\mathcal{V}}_2  (\tau') \|,\\
| e^{(\tau - \tau')\Delta} (\text{div} (\mathcal{V}_2 \nabla \psi_*)) | &\leq & C \frac{1}{\sqrt{ \tau - \tau'}}  \frac{1}{| \ln (T-  t_0(x_0))|^{\frac{1}{4}}}  ,\\
|e^{(\tau - \tau')\Delta } ( \mathcal{V}_2 \Delta \psi_*)  |   & \leq &   \frac{C}{|\ln(T - t_0(x_0) )|^{\frac{1}{2}} },\\
|e^{(\tau - \tau')\Delta } ( p (U_1^{p-1} - \hat U_{K_0}^{p-1}) \psi_* V_2   )(\tau')  |   & \leq &   \frac{C}{ |\ln (T -t_0(x_0))|^{\gamma_2}},\\
|e^{(\tau - \tau')\Delta }  (\psi_* \mathcal{G}_2  )(\tau')| &\leq & \frac{C}{|\ln (T - t_0(x_0))|}.
\end{eqnarray*}
   Plugging    into   \eqref{duhamel-bar-mathcal-V-2}, we obtain
$$  \|\bar{\mathcal{V}}_2 (\tau)\|_{L^\infty}  \leq  \frac{C}{ |\ln(T -t_0(x_0))|^{\gamma_3}} +  C \int_{0}^\tau  \|\bar{\mathcal{V}}_2 (\tau')\|_{L^\infty}   d \tau' ,$$
where $\gamma_3 = \min (\frac{1}{4}, \gamma_2)$. Then, thanks to Gronwall inequality, we get
$$\| \bar{ \mathcal{V}}_2\|_{L^\infty} \leq \frac{C}{|\ln(T -t_0(x_0))|^{\gamma_3}}.$$
Hence,  \eqref{sup-V-2-xi-tau-apro-1} follows . Finally, we  easily  find the asymptotics  of $u^*$ and $u_2^*$ as follows, thanks to the definition of $U$ and  $V_2$  and   to estimates \eqref{sup-v-xi-tau-apro-1} and   \eqref{sup-V-2-xi-tau-apro-1}:
  \begin{equation}\label{limit-u-start}
u^* (x_0) = \lim_{t \to T} u(x_0, t) = (T- t_0 (x_0))^{- \frac{1}{p-1}} \lim_{\tau \to 1} U (x_0, 0, \tau) \sim (T- t_0 (x_0))^{- \frac{1}{p-1}}  \left(\frac{(p -  1)^2}{ 4 p} K_0^2 \right)^{- \frac{1}{p-1}},
\end{equation} 
and 
\begin{equation}\label{limit-im-u-*}
u_2^* (0)= \lim_{t \to T} u_2(x_0, t) = \frac{(T- t_0 (x_0))^{- \frac{1}{p-1}}}{|\ln (T- t_0 (x_0))|} \lim_{\tau \to 1} V_2 (x_0, 0, \tau) \sim \frac{(T- t_0 (x_0))^{- \frac{1}{p-1}}}{|\ln (T- t_0 (x_0))|} \left(\frac{(p -  1)^2}{ 4 p} \right)^{- \frac{p}{p-1}} (K_0^2)^{- \frac{1}{p-1}}.
\end{equation}
Using the relation \eqref{x-0-leq-delta-t-x-0}, we find that
\begin{equation}\label{asymp-T-t-0-x-0} 
  T  - t_0(x_0) \sim \frac{|x_0|^2}{ 2 K_0^2 |\ln |x_0||} \text{ and  }\ln(T - t_0(x_0)) \sim 2 \ln (|x_0|), \quad \text{ as } x_0 \to 0.
  \end{equation}
Plugging  \eqref{asymp-T-t-0-x-0}  into   \eqref{limit-u-start} and   \eqref{limit-im-u-*}, we get the conclusion of  item $(ii)$ of Theorem \ref{Theorem-profile-complex}. 

This concludes the proof of Theorem \ref{Theorem-profile-complex} assuming that Proposition \ref{pro-reduction- to-finit-dimensional} holds. Naturally, we need to prove this propostion  on order  to finish the argument. This will be done  in the next section.

 \section{The proof of Proposition \ref{pro-reduction- to-finit-dimensional}}\label{the proof of proposion-reduction-finite-dimensional}
  This section is devoted to the proof of Proposition \ref{pro-reduction- to-finit-dimensional},
 which is considered   as   central   in    our analysis. We would like  to    proceed  into two parts:
 
 +   In the first part,    we     derive     \textit{a priori estimates}  on  $u$  in  every  component $P_j(t)$  where   $j =  1,2 $ or  $3.$
 
 +    In the second part, we  use    the priori  estimates   to    derive    new bounds    which    improve all the bounds  in Definition   \ref{defini-shrinking-set}, except  for   the   non-negative   modes $( q_{1,0},  (q_{1,j} )_{j \leq n }  , q_{2,0},  (q_{2,j})_{j\leq n}, (q_{2,j,k} )_{j,k \leq n} )$. 

\noindent 
 This means that the problem is reduced to the control of these   components, which is the conclusion of item $(i)$ of Proposition \ref{pro-reduction- to-finit-dimensional}. As for  item $(ii)$ of Proposition \ref{pro-reduction- to-finit-dimensional} is just a direct consequence of the dynamics of these modes.

 \subsection{ A priori  estimates  in $P_1(t), P_2 (t)$ and  $P_3(t)$ }
In this section,    we aim at  giving a \textit{ priori estimates}     to  the solution $u(t)$ on  $P_1(t), P_2 (t)$ and  $P_3 (t)$ which  are   important to get the conclusion of Proposition  \ref{pro-reduction- to-finit-dimensional}:

+ \textit{  A priori    estimates  in  $P_1 (t)$:}  Here  we   give  in the following proposition    some  estimates relevant to the region  $P_1 (t):$  
 \begin{proposition}\label{prop-dynamic-q-1-2-alpha-beta} 
For all $A,K_0 \geq 1$ and $\epsilon_0 > 0, \alpha_0 > 0, \delta_0 >0, \eta_0 > 0,$ there exists $T_4 (K_0, A, \epsilon_0) $ such that  for all $ T \leq T_4,$   if  $u$  is   a  solution of equation \eqref{equ:problem}  on $[0,t_1]$ for some   $t_1 \in [0,T)$ and     $u  \in S(T, K_0,\alpha_0,  \epsilon_0, A, \delta_0, \eta_0,t) $ for all $ t \in \left[ 0, t_1\right]$,  then, the  following holds: for all $s_0 \leq  \tau \leq s \leq s_1$ with $s_1 = \ln (T -t_1),$ we have:
\begin{itemize}
\item[$(i)$] (\textit{ODE satisfied  by the positive modes})  For all $j \in \{1,n\}$ we have
\begin{equation}\label{ODE-q-1-0-1}
\left| q_{1,0}' (s) -  q_{1,0} (s) \right| + \left| q_{1,j}' (s) -  \frac{1}{2}  
q_{1,j} (s) \right|  \leq \frac{C}{s^2},\forall j\leq n.
\end{equation}
\begin{equation}\label{ODE-Phi-0-1}
\left| q_{2,0}' (s) -  q_{2,0}(s)  \right| +  \left| q_{2,j}' (s) -  \frac{1}{2}   q_{2,j}(s)  \right|  \leq \frac{C }{s^{p_1 +2}}, \forall j \leq n.
\end{equation}
\item[$(ii)$] (\text{ODE satisfied by the null modes}) For all $j,k \leq n$
\begin{equation}\label{ODE-q-1-2}
 \left| q_{1,j,k}' (s) + \frac{2}{s} q_{1,j,k} (s) \right|  \leq \frac{C A}{s^3},
\end{equation}
\begin{equation}\label{ODE-Phi-2}
\left| q_{2,j,k} '(s) + \frac{2}{s} q_{2,j,k}(s)\right| \leq \frac{C A^2 \ln s}{s^{p_1 + 3}}.
\end{equation}
\item[$(iii)$] (\textit{Control of  the  negative part})
\begin{equation}\label{Estimata-q-1--}
\left\| \frac{q_{1,-}(.,s)}{1 + |y|^{3}}\right\|_{L^\infty} \leq  Ce^{- \frac{s - \tau }{2}} \left\| \frac{q_{1,-}(.,\tau)}{1 + |y|^{3}}\right\|_{L^\infty}  +  C \frac{ e^{- (s - \tau )^2}}{s^{\frac{3}{2}}}  \|q_{1,e}(.,\tau)\|_{L^\infty} + \frac{C (1 + s -\tau)}{s^2}, 
\end{equation}
\begin{equation}\label{Estimat-q-2--}
\left\| \frac{q_{2,-}(.,s)}{1 + |y|^{3}}\right\|_{L^\infty} \leq  Ce^{- \frac{s - \tau }{2}} \left\| \frac{q_{2,-}(.,\tau)}{1 + |y|^{3}}\right\|_{L^\infty}  +  C \frac{ e^{- (s - \tau )^2}}{s^{\frac{3}{2}}}  \|q_{2,e}(.,\tau)\|_{L^\infty} + \frac{C (1 + s -\tau)}{s^{\frac{p_1 + 5}{2}}}.
\end{equation}
\item[$(v)$] (\textit{Control  of the outer part}) 
\begin{equation}\label{outer-Q-e}
\left\| q_{1,e} (.,s) \right\|_{L^\infty} \leq  C e^{- \frac{(s -\tau)}{p} } \|q_{1,e}(.,\tau)\|_{L^\infty}  + C e^{s - \tau }s^{\frac{3}{2}}  \left\| \frac{q_{1,-}(.,\tau)}{ 1 + |y|^3} \right\|_{L^\infty} +  \frac{C (1 + s - \tau)e^{s - \tau}}{\sqrt s},
\end{equation}
\begin{equation}\label{outer-Phi-e}
\left\| q_{2,e} (.,s) \right\|_{L^\infty} \leq  C e^{- \frac{(s -\tau)}{p} } \|q_{2,e}(.,\tau)\|_{L^\infty}  + C e^{s - \tau }s^{\frac{3}{2}}  \left\| \frac{q_{2,-}(.,\tau)}{ 1 + |y|^3} \right\|_{L^\infty} +  \frac{C (1 + s - \tau)e^{s - \tau}}{s^{\frac{p_1 +2}{2}}}.
\end{equation}
\end{itemize} 
\end{proposition}
\begin{proof}
 By using the fact that    $u(t)  \in S(T, K_0, \alpha_0, \epsilon_0,  A, \delta_0, \eta_0,t) $ for all $ t \in \left[0,  t_1\right],$ we  derive   by the definition that  $(q_1,q_2)(s) \in V_A(s)$ for all $s \in [s_0, s_1]$ and $(q_1,q_2)(s)$ satisfies equation \eqref{equation-satisfied-by-q-1-2}. In addition to that, we deduce also  the fact that  $ q_1 (s) +  \Phi_1 (s)  \geq \frac{e^{-\frac{s}{p-1}}}{2}$ for all $s \in [s_0, s_1]$ (see Lemma  \ref{lemma-positive-real-part}). Although   the  potential  terms  $V_{j,k}$, the quadratic terms $B_1, B_2$ and the rest terms $R_1, R_2$ (see equation \eqref{equation-satisfied-by-q-1-2})  are different from the case where  $p$  is integer, they  behavior  as  in that case (see Lemmas \ref{lemmas-potentials}, \ref{lemma-quadratic-term-B-1-2}, \ref{lemma-rest-term-R-1-2} below).    Hence, the result   is   derived      from   the projection of equation \eqref{equation-satisfied-by-q-1-2} and the dynamics of  the operator $\mathcal{L} + V$.  For   that reason,  we  kindly   refer    the  the reader  to  the proof of Lemma 4.2   given in \cite{DU2017}  for the case  where  $p$ is  integer.   
\end{proof}

+ \textit{ A priori estimates   in  $P_2(t)$:}

In this step, we     aim at  proving the     following lemma    which gives     a priori estimates  on     $u$   in $P_2(t)$. The following is our main result:
\begin{lemma}\label{lema-priori-estima-P-2} 
 For all  $K_0 \geq 1 , \delta_1 \leq 1, \xi_0 \geq 1,  \Lambda_5 > 0, \lambda_5 > 0,$     the following  holds: If   $U (\xi, \tau)$  a solution of equation    \eqref{equa-U-x-0-xi-tau}, for all  $\xi$ and   $\tau \in [\tau_1, \tau_2]$ with  $0 \leq \tau_1 \leq  \tau_2  \leq 1,$  such that   for all $ \tau \in [ \tau_1, \tau_2]$  and for all  $\xi  \in [ -2  \xi_0, 2 \xi_0],$ we have
 \begin{equation}\label{condition-priori-P-2-t}
   |  U (\xi, \tau) |  \leq \Lambda_5    \text{ and }    \text{ Re}\left( U(\xi, \tau )\right)  \geq   \lambda_5     \text{  and   }     \left|  U(\xi, \tau_1 )   -  \hat  U_{K_0} (\tau_1)    \right|  \leq \delta_1,
\end{equation}   
then,  there exists  $\epsilon = \epsilon(K_0,   \Lambda_5, \lambda_5,     \delta_1, \xi_0)$  such  that   for all  $\xi \in [- \xi_0, \xi_0] $ and for all $\tau \in [\tau_1, \tau_2]$ we have 
$$ \left|   U (\xi, \tau)   - \hat U (\tau)\right|\leq  \epsilon, $$ 
where   $\hat U_{K_0}(\tau)$ is given  \eqref{solution-ODE-hat-U}. in  particular,  $\epsilon(K_0,     \Lambda_5, \lambda_5,   \delta_1, \xi_0) \to  0 $ as $( \delta_1, \xi_0) \to (0, + \infty)$.
\end{lemma}
\begin{proof}
We  introduce  $\psi $ as   a  cut-off  function  in $C^\infty_0 (\mathbb{R})$  which    satisfies  the   following: 
$$    \psi(x) = 0 \text{ if } |x| \geq 2 ,  |\psi(x)| \leq 1  \text{ for  all } x \text{ and }  \psi(x)  = 1  \text{ for all } |x| \leq 1, $$  
and we also define  $\psi_1 $  as follows
$$  \psi_1 (\xi)  = \psi \left(\frac{|\xi|}{\xi_0} \right).$$
Then,    we have $\psi_1  \in C^\infty_0 (\mathbb{R}^n),$ and        $\text{supp}(\psi_1) \subset   \{  |\xi|  \text{ such that }  |\xi| \leq  2 \xi_0\}$ and  $\psi_1 (\xi) = 1$ for all  $|\xi|  \leq \xi_0.$
In addtition  to  that,   we let    
$$V_1 (\xi, \tau) =  \psi _1(\xi) \left( U (\xi, \tau) -  \hat U_{K_0} (\tau) \right),  \forall  \tau \in [\tau_1, \tau_2] , \xi   \in \mathbb{R}^n.$$
Thanks to equation \eqref{equa-U-x-0-xi-tau},  we derive  that  $V_1$ satisfies the following equation:
\begin{equation}\label{equa-V-1-priori-P-2}
\partial_\tau V_1   = \Delta_\xi V_1   -   2 \text{ div } ( U \nabla \psi_1 )  +   U \Delta \psi_1   +  \psi_1 (\xi)\left( U^p  - \hat U^p  \right).
\end{equation}
Therefore,    we  can  write  $V_1 (\xi, \tau)$ under  the  following  intergral equation  
\begin{equation}\label{equa-intergral-equa-U}
V_1(\tau )  =  e^{( \tau - \tau_1) \Delta} (V_1(\tau_1))    + \int_{\tau_1}^{\tau} e^{ (\tau- \tau') \Delta} \left(  -   2 \text{ div } ( U \nabla \psi_1 )  +   U \Delta \psi_1   +  \psi_1\left( U^p  - \hat U^p  \right) \right)(\tau') d \tau'.
\end{equation}
 In addition to that, we have the  following  fact from  \eqref{condition-priori-P-2-t} (in particular the estimate   $\text{Re}(U(\xi, \tau))  \geq \lambda_5$ in   \eqref{condition-priori-P-2-t}  is  crucial for the $4^{th}$ term   in \eqref{equa-intergral-equa-U}):  for all $\tau \in  [\tau_1, \tau_2]$
\begin{eqnarray*}
\| V_1(\tau_1)  \|_{L^\infty}  &\leq & \delta_1,\\
\left\|    U \nabla \psi_1   \right\|_{L^\infty}(\tau)   &\leq & \frac{C(\Lambda_5)}{ \xi_0},\\
 \left\|     U \Delta  \psi_1 \right\|_{L^\infty}(\tau)  & \leq &   \frac{C(\Lambda_5)}{ \xi_0^2},\\
 \left\|   \psi_1 (U^p - \hat U^p)    \right\|_{L^\infty}(\tau) &  \leq  & C(K_0, \Lambda_5, \lambda_5) \| V_1\|_{L^\infty} (\tau),
\end{eqnarray*}
which yields  when    $ \tau_1  \leq  \tau' <   \tau \leq  \tau_2,$
\begin{eqnarray*}
\left\|  e^{(\tau - \tau_1 ) \Delta} (V_1(\tau_1)) \right\|  &\leq &  \delta_1,\\
\left\|   e^{(\tau - \tau')\Delta} (\text{div } ( U \nabla \psi_1  )(\tau')) \right\|_{L^\infty}  &\leq & \frac{C(\Lambda_5)}{ \xi_0} \frac{1}{ \sqrt{\tau - \tau'}},\\
 \left\|   e^{(\tau - \tau')\Delta} (  U \Delta  \psi_1(\tau') ) \right\|_{L^\infty}  & \leq &   \frac{C(\Lambda_5)}{ \xi_0^2},\\
 \left\|   e^{(\tau - \tau')\Delta} (\psi_1 (U^p - \hat U^p)(\tau')  )  \right\|_{L^\infty} &  \leq  & C(K_0, \Lambda_5, \lambda_5) \| V_1\|_{L^\infty} (\tau').
\end{eqnarray*}
Plugging  into  \eqref{equa-intergral-equa-U}, we have    for all  $\tau \in [\tau_1, \tau_2]$ 
$$ \left\| V_1 (\tau)\right\|_{L^\infty} \leq  C(K_0, \Lambda_5, \lambda_5) \left(  \delta_1  +   \frac{1}{\xi_0}  \right)  + C(K_0, \Lambda_5, \lambda_5)\int_{\tau_1}^\tau   \left\|  V_1 (\tau')\right\|_{L^\infty} d \tau' .$$
Thanks to Gronwall lemma, we obtain  the following 
$$ \| V_1 (\tau)\|_{L^\infty}  \leq  C(K_0, \Lambda_5, \lambda_5)\left( \delta_1  + \frac{1}{ \xi_0} \right) , \forall  \tau \in [\tau_1, \tau_2].$$
Since $V_1(\tau)   =  U(\tau)  -  \hat U(\tau)$ for all  $\xi \in [- \xi_0,  \xi_0]$ and for all  $\tau \in [\tau_1, \tau_2],$   this concludes  our lemma.  
\end{proof}

+ \textit{ A proiori estimates in   $P_3(t)$:} We aim at  proving the following lemma which gives   a priori estimates  on  $u$ in  $P_3(t).$
\begin{lemma}[A priori estimates in  $P_3(t)$]\label{lemma-control-P-3}
For all $K_0 \geq 1,   A \geq 1,    \eta > 0, \epsilon_0 > 0, \sigma \geq 1 $   and $|d_1|_{\infty} , |d_2|_{\infty} \leq 2 ,$   there exists $T_{6}( K_0, A,  \epsilon_0,\eta, \sigma ) >0, $ such that  for all  $T \leq  T_6$  the following holds:  if $u$ is a solution of  equation \eqref{equ:problem}  for all $t \in [0, t_*]$  for  some $t_* \in [0, T)$  with  the initial data $u(0)  = u_{K_0, A, d_1,d_2} (0)$  (see Definition \ref{initial-data-bar-u})  and  
\begin{equation}\label{condition-priori-P-3}
|u(x,t)|  \leq   \sigma, \forall    |x| \in   \left[\frac{\epsilon_0}{8}, + \infty \right),  t\in [0,t_*],
\end{equation}
then,   
$$ | u(x,t) - u(x,0)| \leq \eta, \forall |x| \geq \frac{\epsilon_0}{4}, t \in [0, t_*]. $$
\end{lemma}
\begin{proof}
We  introduce   $\psi ,$ a  cut-off  function  in $C^\infty (\mathbb{R})$  defined as follows
$$    \psi(r) = 0 \text{ if } |r| \leq \frac{1}{2}, \quad   \psi(r)  = 1  \text{ for all } |r| \geq 1 \text{ and }  |\psi(r)| \leq 1  \text{ for  all } r,  $$  
and  we also introduce $\psi_{\epsilon_0} \in C^\infty (\mathbb{R}^n)$ as follows 
$$  \psi_{\epsilon_0} (x)   =  \psi \left(  \frac{4|x|}{\epsilon_0}\right).$$
Then,  $\psi_{\epsilon_0} \in  C^\infty(\mathbb{R}^n),$  and     $ \psi_{\epsilon_0}(x)  =   1   $ for all  $|x| \geq  \frac{\epsilon_0}{ 4}$ and $\psi_{\epsilon_0} = 0$ for all $|x| \leq \frac{\epsilon_0}{8}$. We define   as well 
$$   v =    \psi_{\epsilon_0} u .$$
Thanks to equation \eqref{equ:problem},   we derive an    equation  satisfied by $v$
\begin{equation}\label{equa-u-1}
\partial_t v  =  \Delta v      - 2 \text{ div} (u\nabla \psi_{\epsilon_0} ) +  u \Delta   \psi_{\epsilon_0}  + \psi_{\epsilon_0} u^p   =  \Delta v   -  2 \text{div} \left( u \nabla \psi_{\epsilon_0}   \right)  +   G ( u),
 \end{equation}
where 
$$  G( u)   =   u \Delta   \psi_{\epsilon_0}  + \psi_{\epsilon_0} u^p. $$
 Using    \eqref{condition-priori-P-3}, we get 
$$ \| G(t,u(t))\|_{L^\infty(\mathbb{R}^n)}   \leq C(\sigma, \epsilon_0), \forall  t \in [0, t_*].$$
By Duhamel formula, we derive    
\begin{equation}\label{duhamel-u-t-0-t-*}
v(t) =  e^{t\Delta} (v (0))   +  \int_{0}^t e^{(t-s) \Delta} (G(s,u(s))) ds, 
\end{equation}
which yields
$$ v (t) - v (0)   =   e^{t \Delta} (v (0)) - v(0) +  \int_{0}^t e^{(t-s) \Delta} (G(s,u(s))) ds. $$
Thus,  
\begin{eqnarray*}
\| v (t) -  v (0)\|_{L^\infty (\mathbb{R}^n)}  \leq   \|e^{t \Delta} (v (0)) - v(0) \|_{L^\infty}  + \left\| \int_{0}^t e^{(t-s) \Delta} (G(s,u(s))) ds \right\|_{L^\infty}.
\end{eqnarray*}
In addition to that,   if    $T \leq   T_{6,1} (\epsilon_0),$ we have   $\chi_1 (x) =  0,   $ for all  $|x| \geq  \frac{\epsilon_0}{8},$  where   $\chi_1  $ defined   in \eqref{defi-proper-U-*} is involved   in  Definition   \ref{defini-shrinking-set} of  initial data $u(0)$. As a matter of fact,      from the definition of   $u(0), $ we deduce   from  this fact   that
$$ v(0)  =  \psi_{\epsilon_0} \left( U^*  + 1 \right).$$
 Since $\Delta v(0) \in L^\infty(\mathbb{R}^n),$ it follows that  
$$ \left\|  e^{t\Delta} (v(0))   - v(0)\right\|_{L^\infty(\mathbb{R}^n)}  \to  0  \text{ as } t \to  0.$$
Besides that,  we have also
$$ \left\| \int_{0}^t e^{(t-s) \Delta} (G(s,u(s))) ds \right\|_{L^\infty(\mathbb{R}^n)}   \to 0 \text{ as }  t \to 0.$$
Therefore,  for all  $t \in [t_0, t_*]$ we have
$$   \| v (t) -  v (0)\|_{L^\infty (\mathbb{R}^n)}    \leq  \eta,$$
provided that  $T \leq  T_{6,2} (K_0, A, \epsilon_0, \eta, \sigma)$. This   concludes our lemma.
\end{proof}
 Finally, we need the following Lemma to get the conclusion  of  our proof: 
\begin{lemma}\label{lemma-implies-Propo-esti-P-3} There exists $ K_{7} \geq 1$ such that  
for all $K_0   \geq K_{7}, A \geq 1,$ and   $   \delta_1 > 0,$    there exists    $\alpha_{7} (K_0,  A, \delta_1) > 0$ such that  for all $\alpha_0 \leq \alpha_{7},$  there  exists  $\epsilon_{7} (K_0, \alpha_0, A, \delta_1) > 0$  such that for all $\epsilon_0 \leq \epsilon_{7}$ there exist $\delta_7  (\delta_1)> 0,      T_{7} (K_0, \epsilon_0, A, \delta_1) > 0,$  $\eta_{7} (K_0, \epsilon_0, A) >0$  such that  for all  $\delta_0 \leq \delta_7,   \eta_0 \leq \eta_{7}$ and  for all $T   \leq T_7 $      if  $u \in S (T, K_0, \alpha_0, \epsilon_0, A, \delta_0, \eta_0, t)$ for all  $t \in [0, t_*] ,$  for some  $t_* \in [0,T),$  then  the following  holds:
$$\text{ whenever } |x| \in \left[  \frac{K_0}{4} \sqrt{ (T -t_*) |\ln(T -t_*)|}, \epsilon_0 \right]$$
\begin{itemize}
\item[$(i)$]    For all $|\xi| \leq 2 \alpha_0 \sqrt{| \ln(T- t(x)) |}   $ and for all 
$$ \tau   \in \left[ \max \left( 0,   \frac{ - t(x)}{ T  - t(x)}  \right),  \frac{t_*  - t(x)}{T - t(x)} \right],$$
if   $U(x, \xi, \tau) $ satisfies equation \eqref{equa-U-x-0-xi-tau},  then  
$$   | U(x, \xi, \tau)|  \leq C^*_7(p)  \text{ and } \text{ Re}\left(  U(\xi, \tau)\right)  \geq C_7^{**}(K_0,p),$$
where  $U(\xi, \tau)$ is defined  as in  \eqref{def-mathcal-U},    $t(x)$ is defined in  \eqref{def-t(x)}, and  $C^*_7$   depends only on the   parameter $p$ and  $C^{**}_7(K_0, p)$ depends  on  the parameters $K_0$ and $p$. 
\item[$(ii)$]  For all $|\xi| \leq  2 \alpha_0 \sqrt{|\ln(T - t(x))|},$     if  we define 
\begin{equation}\label{condition-initial-data-tau-0-x}
\tau_0 (x) =  \max \left( 0, \frac{ - t(x)}{T - t(x)} \right),
\end{equation}
then, we have 
$$ | U (x, \xi, \tau_0) - \hat U_{K_0} (\tau_0)| \leq  \delta_1.$$
\end{itemize}
\end{lemma}
\begin{proof}
  The idea of the proof relies on the argument  in Lemma 2.6, given in  \cite{MZnon97}.

\noindent
+ \textit{The proof of  item $(i)$:}  We aim at proving that for all $|x| \in  \left[  \frac{K_0}{4} \sqrt{(T-t_*) |\ln (T-t_*)|}, \epsilon_0\right],$   $|\xi| \leq  2 \alpha_0  \sqrt{|\ln (T  - t(x))|} $ and $t \in \left[ \max(0, t(x)), t_* \right],$ we have 
\begin{equation}\label{estima-U-xi-tau-hat-u-bound-1-2}
  |U (x, \xi, \tau(x,t))|  \leq     C^*_7,
\end{equation}
and 
\begin{equation}\label{proof-bounded-below-real-part}
\text{Re}  \left(  U(\xi, \tau) \right) \geq   C_7^{**},
\end{equation}
where $\tau (x,t) = \frac{t - t (x) }{T- t(x)}$ and $C_7^*, C_7^{**}> 0$.    Let  us   introduce    a parameter  $\delta  > 0 $ to be  fixed  later  in our proof,    small enough (note that  $\delta$  has nothing to do with the parameters  $\delta_0, \delta_1$  in the statement of our lemma).  We observe that if   we have  $ \alpha_0 \leq \alpha_{1,7} (K_0, \delta)  $ for some  $\alpha_{1,7} > 0$ and  small enough, then for all $|\xi|  \leq  2 \alpha_0 \sqrt{|\ln (T- t(x))|},$ we have 
\begin{equation}\label{x+xi-1-delta-1+delta}
(1 - \delta) |x| \leq |x + \xi \sqrt{ T - t(x)}| \leq  (1+ \delta) |x|.  
\end{equation}
We also recall the definition of rescaled function   $U(x, \xi,  \tau (x,t))$  as follows
$$  U(x, \xi,  \tau ) = \left( T - t(x) \right)^{\frac{1}{p-1}} u(x + \xi \sqrt{T - t(x)}, t(x) + \tau (T-t(x)) ).   $$
Introducing    $X =  x +  \xi \sqrt{T - t(x)},$ we write 
$$ U(x, \xi, \tau (x,t) )  = (T- t(x))^{\frac{1}{p-1}}  u (X,t).$$
We here  consider 3 cases:

+ \textit{ Case 1:}  We consider  the case where 
$$|X|  \leq  \frac{K_0}{4} \sqrt{ (T-t) |\ln (T-t)|}.$$
Using  the fact that  $u \in S(t)$,  in particular item $(i)$ of  Definition \ref{defini-shrinking-set}, we see that Lemma  \ref{lemma-analysis-V-A} and    \eqref{conclu-u-f-0} hold, hence
$$ \left|(T - t)^{\frac{1}{p-1}}  u(X, t)  - f_0 \left( \frac{X}{ \sqrt{(T-t)|\ln (T-t)|}}\right) \right| \leq  \frac{C A^3}{ \sqrt{1 + |\ln (T- t)|}}.$$

\noindent
Then,  we derive  the following 
\begin{eqnarray}
| U(x, \xi, \tau (x,t))| &\leq &\left( \frac{T-t}{T-t(x)}\right)^{-\frac{1}{p-1}} \left( f_0\left( 0\right)   +  \frac{C A^3}{ \sqrt{1 + |\ln (T- t)|}}   \right) \nonumber\\
& = & \left( \frac{T-t}{T-t(x)}\right)^{-\frac{1}{p-1}} \left( \kappa+  \frac{C A^3}{ \sqrt{1 + |\ln (T- t)|}}   \right) ,\label{estima-U-xi-leq-T-1-t(x)-inside-blowup}\\
\text{Re} (U(x, \xi, \tau(x,t))) & \geq  &\left( \frac{T-t}{T-t(x)}\right)^{-\frac{1}{p-1}} \left( f_0\left( 0\right)   - \frac{C A^3}{ \sqrt{1 + |\ln (T- t)|}}   \right) \nonumber\\
& = & \left( \frac{T-t}{T-t(x)}\right)^{-\frac{1}{p-1}} \left( \kappa -  \frac{C A^3}{ \sqrt{1 + |\ln (T- t)|}}   \right).\label{estima-U-1-xi-leq-T-1-t(x)-inside-blowup}
\end{eqnarray}

\noindent
Besides that,  we deduce  the following   from   \eqref{x+xi-1-delta-1+delta}   and   the fact that  $|X| \leq  \frac{K_0}{4} \sqrt{ (T-t)|\ln (T-t)|}:$ 
$$|x| \leq  \frac{K_0}{ 4(1 - \delta)} \sqrt{(T-t) |\ln(T-t)|}.$$
In addition to that, we have    that the function $T - t(x)$ is an increasing function if  $|x|$ small enough. Therefore,  
\begin{equation}\label{estima-Ttheta-leq-theta-K_0-4-T-t}
T-t(x) \leq  T- t\left( \frac{K_0}{4 (1 - \delta)} \sqrt{ (T-t) |\ln (T- t)|}\right) .
\end{equation}
As a matter of fact, we have the following asymptotics of function  $\theta (x) = T -t (x), $ 
\begin{equation}\label{asymptotic of-theta-x}
\ln \theta (x) \sim  2\ln |x|  \text{ and  } \theta (x) \sim \frac{8}{K_0^2} \frac{|x|^2}{ |\ln |x||} \text{ as } |x| \to 0.
 \end{equation}
 Plugging    \eqref{asymptotic of-theta-x} in  \eqref{estima-Ttheta-leq-theta-K_0-4-T-t}, we obtain the following
$$ T-t(x) \leq  T- t\left( \frac{K_0}{4 (1 - \delta)} \sqrt{ (T-t) |\ln (T- t)|}\right) \sim \frac{8 K_0^2 (T-t) |\ln(T-t)|}{K_0^2 16 (1-\delta)^2 \frac{1}{2} |\ln(T-t)|}  = \frac{(T-t)}{(1-\delta)^2}.$$
In particular,  from   $t \in  [\max(0, t(x)), t_*],$  we have the following 
$$  T- t(x)  \geq   T- t.$$
Plugging  into  \eqref{estima-U-xi-leq-T-1-t(x)-inside-blowup}  and  \eqref{estima-U-1-xi-leq-T-1-t(x)-inside-blowup},  we obtain
$$ |U(x, \xi, \tau)| \leq   C^*_{1,7}(p, \delta),$$
and 
$$\text{Re}(U(x, \xi, \tau(x,t)))   \geq  C^{**}_{1,7}(p,\delta),$$
provided that  $\delta$  is small enough,    $K_0 \geq  K_{1,7} (\delta)$ which is large enough and  $T \leq T_{1,7}(K_0, A)$.   Note that   $ C^*_{1,7}(p, \delta)   $ and   $C^{**}_7(p, \delta)$   depend  on $\delta$ and  $p$,  in particular,    $C^*_{1,7}(\delta, p)$ is bounded when   $\delta \to 0$.
 
+  \textit{ The second case: } We consider the case  where 
 $$ |X| \in \left[ \frac{K_0}{4} \sqrt{(T-t) |\ln (T-t)|}, \epsilon_0 \right].$$
By  using  the definition of  $U(x, \xi, \tau(x,t)),$ we deduce that 
 $$ U(x, \xi, \tau (x,t))  = \left( \frac{T-t(x)}{T- t(X)}\right)^{\frac{1}{p-1}} U(X,0, \tau(X,t)).$$
 However,  using  the fact that  $u \in S (t)$, in particular   item $(ii)$ of Definition  \ref{defini-shrinking-set},  we have
 $$ | U(X,0, \tau (X,t) )| \leq  \delta_0 + \hat U(1).  $$
 In addition to that,  we use   \eqref{x+xi-1-delta-1+delta}, the definition of $t(x)$ and  the fact that  $|X| \geq \frac{K_0}{4} \sqrt{(T-t) |\ln (T-t)|}   $  to derive  the following 
 $$  1 \leq  \frac{T-t(x)}{T- t(X)} \leq 2, $$
 provided that   $\delta$ small enough. Therefore, we have 
 $$ \left|U(x, \xi, \tau (x,t))\right| \leq  2^{\frac{1}{p-1}} \left(\delta_0 + \hat  U_{K_0}(1)   \right) \leq  \frac{1}{2},$$
 and
 $$  \text{Re}(U(x, \xi, \tau(x,t))) \geq    \hat U_{K_0} (0) - \delta_0  \geq  \frac{1}{2}\hat U_{K_0}(0),$$
 provided that   $ \delta_0 \leq  \frac{1}{2}\hat U_{K_0}(0) $  and $K_0 \geq  K_{2,7}.$    
 
+   \textit{ The third  case:} We consider  the case where  $ |X|   \geq \epsilon_0.$  Using the fact that $u\in S(t),$ in particular   item  $(iii)$ of Definition \ref{defini-shrinking-set}, we have
\begin{eqnarray*}
 |U(x, \xi, \tau (x,t))| &=&  \left( T - t(x)\right)^{\frac{1}{p-1}} |u(X,t)| \leq   \left( T - t(x)\right)^{\frac{1}{p-1}}  ( |u(X,0)|  + \eta_0 ),\\
 \text{Re}\left( U(x, \xi, \tau(x,t))\right) & = &  \left( T - t(x)\right)^{\frac{1}{p-1}} \text{Re}(u(X,t)) \geq      \left( T - t(x)\right)^{\frac{1}{p-1}} \left(  \text{Re}(u(X,0)) - \eta_0   \right).
\end{eqnarray*}
Using  the definition  \eqref{parti-real-initial}, we have  for all $|X| \geq \epsilon_0$
$$u(X,0)  =  U^*(X) + 1 ,$$
provided that $T \leq  T_{2,7}(\epsilon_0).$
 In addition to that, we have the following fact
 \begin{eqnarray*}
 T- t(x) &\sim &  \frac{ 16  |x|^2}{K_0^2|\ln|x||}, \label{equaivalent-T-t-x}\\
 u(X,0)  &\sim &  U^*(X)   =  \left[ \frac{(p-1)^2 |x|^2}{8 p |\ln |x||} \right]^{- \frac{1}{p-1}},
 \end{eqnarray*}
 as  $(X,x) \to (0,0)$,   and  in particular,   from  \eqref{x+xi-1-delta-1+delta}, we have
 $$      (1 - \delta)|x|  \leq |X|  \leq  (1 + \delta) |x| .$$
Therefore, we have  
\begin{eqnarray*}
 \left| U (x, \xi, \tau (x,t))\right|  \leq C^*_{2,7}(\delta),\\
 \text{Re}(U(x, \xi, \tau(x,t)))  \geq   C^{**}_{2,7} (K_0, \delta),
\end{eqnarray*}
  provided that   $K_0 \geq  K_{3,7},  \eta_0 \leq  \eta_{1,7}(\delta)$ and   $\delta$ is small. 
We conclude item $(i)$.
 
 \noindent
 \textit{The  proof  of item $(ii)$:} We  aim at  proving   that  for all  $|\xi| \leq  2 \alpha_0 \sqrt{ |\ln \theta (x)|}$ and   $\tau_0 (x)  = \max \left(0,  - \frac{t(x)}{ \theta (x)} \right),$ we have 
 \begin{equation}\label{estiam-U-hat-U-tau-0-max}
 \left|   U(x, \xi, \tau_0(x))    - \hat U_{K_0} (\tau_0(x))  \right| \leq  \delta_1.
 \end{equation}
 Considering 2 cases  for  the proof  of  \eqref{estiam-U-hat-U-tau-0-max}:
 
 + Case 1: We consider the case where 
 $$|x| \leq    \frac{K_0}{4} \sqrt{ T |\ln T|},$$
 then, we deduce from the defintion of $t(x)$ given by \eqref{def-t(x)} that $t(x)  \leq 0$. Thus,   by definition    \eqref{condition-initial-data-tau-0-x}, we have 
 $$ \tau_0 (x)   = \frac{-t(x)}{ \theta (x)}.$$
 Therefore,   \eqref{estiam-U-hat-U-tau-0-max}  directly  follows  item $(ii)$ of Lemma  \ref{lemma-control-initial-data} with   $K_0 \geq K_{4,7},   \alpha_0 \leq \alpha_{3,7},   \epsilon_0 \leq \epsilon_{3,7}$ (see in Lemma \ref{lemma-control-initial-data})
 
 + Case  2: We consider the case  where  
 $$ |x| \geq  \frac{K_0}{4} \sqrt{ T |\ln T|},$$
 which  yields  $t(x) \geq 0$. Thus,  by definition    \eqref{condition-initial-data-tau-0-x}, we have   $$\tau_0(x)  =  0.$$ 
  We let $X   = x + \xi \sqrt{ \theta (x)}.$  Accorrding to the  definitions  of $U, \hat U_{K_0}$ which are given  by  \eqref{def-mathcal-U} and  \eqref{solution-ODE-hat-U}, we write
 \begin{eqnarray*}
 \left|  U (x, \xi, 0)   - \hat U_{K_0} (0)  \right| &=&  \left|  \theta^{-\frac{1}{p-1}} (x) u\left(X, t(x) \right)  -  \left(  (p-1)  +  \frac{(p-1)^2}{4p}   \frac{K_0^2}{ 16}\right)^{-\frac{1}{p-1}}  \right| \\
 & = & \left|   \theta^{-\frac{1}{p-1}} (x) u\left(X, t(x) \right)   -  \left(  (p-1)  +  \frac{(p-1)^2}{4p}   \frac{|X|^2}{  \theta (x) |\ln \theta (x)| }\right)^{-\frac{1}{p-1}}      \right.\\
 &+  & \left.   \left(  (p-1)  +  \frac{(p-1)^2}{4p}   \frac{|X|^2}{  \theta (x) |\ln \theta (x)| }\right)^{-\frac{1}{p-1}} - \left(  (p-1)  +  \frac{(p-1)^2}{4p}   \frac{K_0^2}{ 16}\right)^{-\frac{1}{p-1}}     \right|\\
 & \leq  & (I) +  (II),
 \end{eqnarray*}
 where $\theta (x)  =  T-t(x),$ and  
 \begin{eqnarray*}
 (I) &=&  \left|    \theta^{-\frac{1}{p-1}} (x) u\left(X, t(x) \right)   -  \left(  (p-1)  +  \frac{(p-1)^2}{4p}   \frac{|X|^2}{  \theta (X) |\ln \theta (X)| }\right)^{-\frac{1}{p-1}}         \right|,\\
 (II) & = &  \left|   \left(  (p-1)  +  \frac{(p-1)^2}{4p}   \frac{|X|^2}{  \theta (X) |\ln \theta (X)| }\right)^{-\frac{1}{p-1}} - \left(  (p-1)  +  \frac{(p-1)^2}{4p}   \frac{K_0^2}{ 16}\right)^{-\frac{1}{p-1}}  \right| .
 \end{eqnarray*}
 Since
 $$  |X|  \leq (1 + \delta) |x|   \leq   \frac{(1  +\delta ) K_0}{4}  \sqrt{  (T- t(x)) |\ln (T-t(x))|}   \leq K_0 \sqrt{(T-t(x)) |\ln (T-t(x))|},$$
 Using item    $(i)$ of  Definition  \ref{defini-shrinking-set}, taking    $ t = t(x), $ we write 
 $$  (I)  \leq  \frac{C(K_0) A^2}{  \sqrt{|\ln (T-t(x))|}}  \leq   \frac{C(K_0) A^2}{ \sqrt{|\ln T|}} \leq \frac{\delta_1}{2},$$
 provided that  $T \leq T_{4,7} (K_0, A, \delta_1)$.  Besides that,   from  \eqref{x+xi-1-delta-1+delta}  we have
 $$  (1  - \delta)^2 \frac{K_0^2}{ 16} \leq   \frac{|X|^2}{ \theta (X) \left|\ln \theta (X) \right|} \leq (1 + \delta)^2 \frac{K_0^2}{16}.$$
This yields
$$ (II) \leq   \frac{\delta_1}{2},$$
provided  that   $\delta$ is  small enough.  Then,  \eqref{estiam-U-hat-U-tau-0-max} follows.  Finally, we fix  $\delta > 0$ small enough and we conclude our lemma.
\end{proof}
\subsection{The conclusion of  Proposition   \ref{pro-reduction- to-finit-dimensional}}
It this subsection,   we  would like   to     conclude the    proof of  Proposition \ref{pro-reduction- to-finit-dimensional}. As   we  mentioned earlier,   in the   analysis of   the shrinking set   $S(t),$ the heart     is   the set   $V_A(s)$ (see   item  $(i)$   of Defintion  \ref{defini-shrinking-set}   of  $S(t)$). So, let us  first  give  an important argument    related   the analysis of  $V_A(s)$; the reduction  to finite dimensions.    More precisely,        we prove that   if the  solution   $(q_1, q_2)  $ of equation  \eqref{equation-satisfied-by-q-1-2}     satisfies $(q_1, q_2) (s) \in V_A(s) $ for all  $s \in [s_0, s_*]$ and      $(q_1, q_2) (s_*)  \in  \partial V_A (s_*)   $ for some $ s_* \in [s_0, +\infty)$ with $s_0 = - \ln T,$    then,  we    can  directly derive  that
$$ ( q_{1,0},  (q_{1,j} )_{j \leq n }  , q_{2,0},  (q_{2,j})_{j\leq n}, (q_{2,j,k} )_{j,k \leq n} )(s_*) \in \hat \partial  V_A (s_*), $$
where  $\hat V_A (s_*)$ is defined in \eqref{defini-of-hat-V-A}. After that,    we will    use   the  dynamic of  these  modes to derive   that they will  leave    $\hat  V_A$  after that.  The  following is our statement
\begin{proposition}[A reduction to finite  dimensional problem]\label{control-q(s)-V-A-s-1-2}
There exists $A_8 \geq 1, K_8 \geq 1$ such that for all  $ A \geq A_8, K_0 \geq  K_8$, there exists $s_8(A, K_0)\geq 1$ such that for all $s_0 \geq s_8(A, K_0)$, we have the following properties: If the following conditions hold:
\begin{itemize}
\item[$a)$]  We   take    the initial data  $(q_1,q_2)(s_0) $    are     defined     by  $ u_{A, K_0,d_1, d_2} (0)$       with   $s_0 = - \ln T$  (see  Definition \ref{initial-data-bar-u},  \eqref{similarity-variales}    and \eqref{defini-q-1-2})    and    $(d_0,d_1) \in \mathcal{D}_{K_0, A,s_0}$  (see  in Lemma \eqref{lemma-control-initial-data}).
\item[$b)$] For all $s \in [s_0,s_1],$   the solution $(q_1,q_2) $ of  equation \eqref{equation-satisfied-by-q-1-2} satisfies:  $(q_1,q_2)(s) \in V_A(s)$ and  $q_1(s) + \Phi_1 (s) \geq \frac{1}{2} e^{- \frac{s}{p-1}}$.
\end{itemize}

Then, for all $s \in [s_0,s_1]$, we have
\begin{eqnarray}
 \forall i,j \in \{1, \cdots, n\}, \quad |q_{2,i,j}(s)| & \leq & \frac{A^2 \ln s}{2 s^2}, \label{conq_1-2} \\
\left\| \frac{q_{1,-}(.,s)}{1 + |y|^3}\right\|_{L^{\infty}} &\leq &  \frac{A}{2 s^{2}}, \quad \|q_{1,e}(s)\|_{L^{\infty}} \leq  \frac{A^2}{2 \sqrt s}, \label{conq-q-1--and-e}\\
\left\| \frac{q_{2,-}(.,s)}{1 + |y|^3}\right\|_{L^{\infty}}  &\leq &  \frac{A^2}{2 s^{\frac{p_1 + 5}{2}}},\quad \|q_{2,e}(s)\|_{L^{\infty}} \leq  \frac{A^3}{2 s^{\frac{p_1 + 2}{2}}}. \label{conq-q-2--and-e}
\end{eqnarray}
\end{proposition}
\begin{proof}
 The proof   is  quite    similar   to  Proposition 4.4  in \cite{DU2017}. Indeed,  the proof is a consequence    of  Proposition \ref{prop-dynamic-q-1-2-alpha-beta},   exactly  as in  \cite{DU2017}. Thus, we omit  the proof and refer  the reader to  \cite{DU2017}.  
\end{proof}
Here,  we  give the  conclusion     of the proof of    Proposition \ref{pro-reduction- to-finit-dimensional}: 

\textit{ Conclusion of the  proof of  Proposition \ref{pro-reduction- to-finit-dimensional}:} We first   choose    the parameters $K_0, A,  \alpha_0,   \epsilon_0,   \delta_0, \delta_1,  \eta_0, \eta$  and  $ T > 0$  such that  all      the above Lemmas and Propositions which are necessary  to  the proof, are satisfied. In particular,  we also  note that     the parameters $\delta_1 $and $\eta $ which are   introduced   in Lemma \ref{lemma-control-initial-data}  and Lemma  \ref{lemma-control-P-3},   will be   small enough ( $\delta_1   \ll \delta_0       $ and   $\eta \ll \eta_0 $).  Finally, we fix    the constant $T$  small enough,   depending  on all the above parameters, then  we conclude our Proposition.   We   now  assume  the solution $u$ of equation  \eqref{equ:problem} with  initial data  $u_{K_0, A,d_1,d_2} (0)$,  defined  in   Definition \ref{initial-data-bar-u}, satisfies the following  
$$u \in S(T,K_0, \alpha_0, \epsilon_0,  A, \delta_0, \eta_0, t)= S(t), $$
 for all  $ t \in [0, t_*]$ for some $t_* \in [0, T)$ and 
$$ u \in \partial  S(t_*).$$ 
We aim at  proving that 
\begin{equation}\label{q-1-2-V-A-s-*}
 (q_1,q_2) (s_*) \in  \partial  V_A (s_*),
\end{equation}
where $s_* = \ln (T-t_*)$. Indeed, by contradition, we suppose that  \eqref{q-1-2-V-A-s-*}  is not true, then, by using Definition   \ref{defini-shrinking-set}   of $S(t),$    we  derive the following:   

$(I)$  Either,    there exist $x_*, \xi_*$  which  satisfy
\begin{eqnarray*}
|x_*| &\in & \left[ \frac{K_0}{4} \sqrt{ (T-t_*) |\ln(T-t_*)|}, \epsilon_0  \right],\\
|\xi_*| & \leq & \alpha_0 \sqrt{ |\ln (T-t(x_*))|}.
\end{eqnarray*}
and 
$$   |  U (x_*, \xi_*, \tau (x_*, t_*))  - \hat U  (\tau (x_*, \tau_*))| = \delta_0.$$

$(II)$  Or,  there exists  $x^*$  such that    $    |x^*| \geq \frac{\epsilon_0}{4} $     and 
$$ | u(x^*, t_*)  -  u (x^*, 0)| = \eta_0.$$
We would  like to  prove that    $(I)$  and $(II)$ can not occur. Indeed,   if  the first case occurs, then, letting    $\tau_0(x_*)   =  \max \left(   - \frac{t(x_*)}{ \theta(x_*)}, 0\right)$,  it follows  from Lemma  \ref{lemma-implies-Propo-esti-P-3}  that:   For all   $|\xi|   \leq   2 \alpha_0  \sqrt{  |\ln (T-t(x_*))|},$  we have   
$$   \left|    U(x_*, \xi, \tau_0(x_*))   -  \hat U (\tau_0(x_*))  \right|   \leq \delta_1,$$
and  for all  $\tau  \in \left[  \max \left( - \frac{t(x_*)}{T-t(x_*)},  \frac{t_*-t_(x_*)}{T-t(x_*)}\right)  \right] ,$ we have 
\begin{eqnarray*}
|U(x_*,\xi, \tau(x_*) )|    & \leq & C^*_7,\\
 \text{Re}(U(x_*,\xi, \tau(x_*) ))   &\geq   & C^{**}_7,
\end{eqnarray*}
where   $C^*_7, C^{**}_7$ are given in  Lemma \ref{lemma-implies-Propo-esti-P-3}.
Then,  we apply   Lemma  \ref{lema-priori-estima-P-2}, with   $\xi_0   =  \alpha_0 \sqrt{|\ln(T-t(x_*))|}, \tau_1 = \tau_0(x_*), \tau_2  =    \frac{t_* - t(x_*)}{T-t(x_*)},  \lambda_5 =  C^{**}_7$ and  $\Lambda_5 = C^*_7,$ to  derive that:    for all   $\xi \in [-\xi_0, \xi_0]$
$$ \left|   U(x_*,   \xi,   \tau (x_*, t_*))   - \hat U (\tau (x_*, t_*)) \right|  \leq C(K_0, \Lambda_5   \lambda_5, \delta_1, \xi_0),$$
where  $C(K_0, \Lambda_5, \lambda_5, \delta_1, \xi_0) \to 0$ as  $(\delta_1, \xi_0) \to (0, +\infty)$.  Taking    $(\delta_1, \xi_0)  \to (0, +\infty),$ (note that  $\xi_0 \to  + \infty$  as $\epsilon_0 \to 0$),    we write
$$  \left|  U(x_*,   \xi_*,   \tau (x_*, t_*))   - \hat U (\tau (x_*, t_*))     \right| \leq   \frac{\delta_0}{2},$$
this  is   a  contradiction.

 If   $(II)$  occurs,    we  have     for all  $ |x| \in \left[ \frac{\epsilon_0}{8}, + \infty \right)$
 $$ |u(x,t)|  \leq   C( \epsilon_0, A, \delta_0, \eta_0), \forall  t \in [0, t_*].$$
 Indeed,  we consider  the two  following cases:
 
 + The case where  $|x| \geq  \frac{\epsilon_0}{4},$ using   item  $(iii)$  if the definition of  $S(t),$ we derive the following
 $$|u(x,t)|  \leq  |u(x,0)|  + \eta_0 \leq  C(A, \eta_0), \forall  t \in [0, t_*].$$
 
 + The case where   $|x| \in \left[  \frac{\epsilon_0}{8}, \frac{\epsilon_0}{4} \right],$ using  item  $(ii)$  in the definition of  $S(t),$ we have
$$  |u(x,t)| \leq   C(\delta_0) \left(  T-t(x)\right)^{-\frac{1}{p-1}}   \leq  C(\epsilon_0, \delta_0), \forall t \in [0, t_*].$$ 
\noindent
Then,  we apply Lemma \ref{lemma-control-P-3} with  $\eta  \leq \frac{\eta_0}{2}$ and   $\sigma=C( \epsilon_0, A, \delta_0, \eta_0),  $ to derive the following  
$$  | u(x^*, t_*)  -  u (x^*, 0)| \leq  \frac{\eta_0}{2}.$$
Therefore, $(II)$  can not occurs.  Thus, \eqref{q-1-2-V-A-s-*} follows.  In addition to that,  from   \eqref{q-1-2-V-A-s-*},    Proposition   \ref{prop-dynamic-q-1-2-alpha-beta} and        Lemma \ref{control-q(s)-V-A-s-1-2},  we  conclude  the proof of item  $(i)$ of Proposition   \ref{pro-reduction- to-finit-dimensional}.  Since,   item  $(ii)$  follows from  item  $(i)$    (see  for instance the proof of  Proposition 3.6, given  in \cite{DU2017}). This concludes  the proof  of  Proposition \ref{pro-reduction- to-finit-dimensional}.

\appendix

\section{ Cauchy problem for  equation \eqref{equ:problem}}\label{appendix-Cauchy-problem}
In this section,  we giva a proof to a local  Cauchy problem in time. 
\begin{lemma}[A local Cauchy problem for a  complex heat equation]\label{local-Chauchy-exitece} Let $u_0$  be   any function  in $L^\infty \left( \mathbb{R}^n, \mathbb{C} \right)$ such that  
\begin{equation}\label{condition-Re-u_0-geq-lambda}
 \text{ Re}(u_0(x)) \geq \lambda, \forall  x \in \mathbb{R}^n, 
\end{equation}
for some constant  $\lambda > 0$. Then, there exists   $T_1> 0 $ such that  equation \eqref{equ:problem} with initial data $u_0,$  has   a unique  solution   on  $\left(0,T_1\right].$ Moreover, $u \in C\left( \left(0,T_1\right], L^\infty (\mathbb{R}^n)\right)$  and 
$$\text{ Re}(u(t))  \geq \frac{\lambda}{2}, \forall  (t,x) \in [0,T_1] \times  \mathbb{R}^n.$$
\end{lemma} 
 \begin{proof}
 The proof  relies on a  fixed-point  argument. Indeed, we consider  the  space
 $$ X =  C\left((0,T_1], L^\infty(\mathbb{R}^n, \mathbb{C}) \right).$$
 It is easy to check that  $X$ is an Banach space with the following norm
 $$ \| u \|_X = \sup_{t \in  (0,T_1]} \| u (t)\|_{L^\infty} , \forall u = (u(t))_{t\in (0,T_1]} \in X.$$
We also introduce   the closed set    $B^+_{\lambda}\left( 0, 2 \|u_0\|_{L^\infty} \right) \subset X$  defined as follows
$$ B^+_{\lambda}\left( 0, 2 \|u_0\|_{L^\infty} \right) = \left\{ u \in X \text{ such that }  \|u \|_X   \leq  2 \|u_0\|_{L^\infty}   \right\} \cap \left \{  u \in X |  \forall  t \in (0,T_1],  \text{Re}(u(t,x))  \geq  \frac{\lambda}{2}  \text{ a. e}    \right\}  $$
 Let  $\mathcal{Y}$  be   the following mapping
$$  \mathcal{Y}:   B^+_{\lambda}\left( 0, 2 \|u_0\|_{L^\infty} \right) \to X ,$$
where  $\mathcal{Y}(u) =  (\mathcal{Y}(u)(t))_{t\in (0,T_1]}$   is  defined  by 
\begin{equation}\label{equation-u_0-maping-mathcal-Y}
\mathcal{Y} (u) (t)=  e^{t \Delta } (u_0) + \int_{0}^t e^{(t-s) \Delta } (u^p(s)) d s.
\end{equation}
Note that,  when    $u  \in B^+_{\lambda}\left( 0, 2 \|u_0\|_{L^\infty} \right),$  $u^p$  is well  defined  as in        \eqref{defi-mathbb-A-1-2} and   \eqref{argument-u-1-u-2}.
We  claim  that  there exists $T^*= T^* (\|u_0\|_{L^\infty}, \lambda) > 0$ such that for all   $0< T_1 \leq T^*$,  the   following   assertion  hold:
\begin{itemize} 
\item[$(i)$]      The mapping is reflexive on  $B^+_{\lambda}\left( 0, 2 \|u_0\|_{L^\infty} \right),$  meaning     that 
$$\mathcal{Y} :  B^+_{\lambda}\left( 0, 2 \|u_0\|_{L^\infty} \right)   \to   B^+_{\lambda}\left( 0, 2 \|u_0\|_{L^\infty} \right) .$$
\item[$(ii)$]   The mapping  $\mathcal{Y}$   is  a contraction mapping  on $B^+_{\lambda}\left( 0, 2 \|u_0\|_{L^\infty} \right)$ :
$$   \| \mathcal{Y} (u_1)  -  \mathcal{Y}(u_2)\|_X     \leq   \frac{1}{2} \|u_1 - u_2\|_{X},$$
for all  $u_1,u_2  \in   B^+_{\lambda}\left(0,  2 \|u_0\|_{L^\infty} \right) $.
\end{itemize} 
\textit{The proof for $(i)$:} By observe that,   by using  the regular  property  of  operator $e^{t \Delta},$  we conclude  that   $\mathcal{Y}(u) \in C\left(  (0,T_1], L^\infty(\mathbb{R}^n, \mathbb{C})  \cap  C(\mathbb{R}^n, \mathbb{C})   \right).$     Besides that,  for all  $u \in  B^+_{\lambda}\left( 0, 2 \|u_0\|_{L^\infty} \right) $  we derive from \eqref{equation-u_0-maping-mathcal-Y} that for all $t \in (0,T_1]$
\begin{eqnarray*}
\|  \mathcal{Y} (u) (t) \|_{L^\infty} &=& \left\|  e^{t \Delta}  (u_0)   +   \int_{0}^t e^{(t-s) \Delta } (u^p(s))ds \right\|_{L^\infty}\\
&\leq &  \left\|  e^{t \Delta}   (u_0)  \right\|_{L^\infty}  +    \left\| \int_{0}^t e^{(t-s) \Delta } (u^p(s)) ds \right\|_{L^\infty}\\
&\leq &  \|u_0\|_{L^\infty}  +  t 2^p\|u_0\|_{L^\infty}^p    .
\end{eqnarray*}
Hence,  if we take  $T_1 \leq  \frac{1}{ 2^p \|u_0\|^{p-1}_{L^\infty}} $   then  we have  
$$  \|\mathcal{Y}(u)\|_X  =     \sup_{t\in (0,T_1]} \|  \mathcal{Y} (u)  \|_{L^\infty}  \leq   2 \| u_0\|_{L^\infty}.$$
  Now, let us   note   from    \eqref{condition-Re-u_0-geq-lambda} that 
  $$  \text{Re}\left(e^{t \Delta} (u_0) \right)  =  e^{t \Delta} \left(  \text{Re}(u_0) \right) \geq  e^{t \Delta} \left(  \lambda \right)   =  \lambda.$$
Therefore,    from   \eqref{equation-u_0-maping-mathcal-Y}    for all $(t,x) \in (0,T_1] \times  \mathbb{R}^n$
$$  \text{Re} (\mathcal{Y}(u)(t,x))  \geq  \lambda  -   \left\|  \int_{0}^t e^{(t-\tau)\Delta}(u^p)(\tau ) d \tau\right\|_{L^\infty}.$$
Note that,   
$$\left\|   \int_{0}^t e^{(t-\tau)\Delta}(u^p)(\tau ) d \tau \right\|_{L^\infty}  \leq   t 2^p \|u_0\|_{L^\infty}^p.$$
So,   if  $T_1  \leq  \frac{\lambda}{2^{p+1} \|u_0\|_{L^\infty}},$ then  for all $t \in (0, T_1] \times  \mathbb{R}^n$
$$\text{Re} (\mathcal{Y}(u)(t,x)) \geq  \frac{\lambda}{2}.$$
Therefore, 
$$ \mathcal{Y} (u) \in  B^+_{\lambda}\left(0,  2 \|u_0\|_{L^\infty} \right)  .$$

\textit{The proof  of   $(ii)$:} We first recall that the function  $G(u) = u^p, u \in \mathbb{C} $ is analytic on  
$$\left\{ u \in \mathbb{C} \text{ such that }   \text{ Re}(u) \geq \frac{\lambda}{ 2}\right\}.$$
Then,  there exists $C_2( \|u_0\|_{L^\infty}, \lambda) >0$ such that  
\begin{eqnarray*}
\| \mathcal{Y} (u_1)  -  \mathcal{Y}(u_2) \|_X   &=& \sup_{t\in (0,T_1]}  \left\| \int_0^{t} e^{(t-s)\Delta} \left(u_1^p - u_2^p \right)(s) ds  \right\|_{L^\infty}\\
&\leq & T_1C_2 \sup_{t \in  (0,T_1]} \|u_1 - u_2  \|_{L^\infty}.
\end{eqnarray*}
 Then,  if we   impose
  $$ T_1  \leq    \frac{1}{2C_2},$$
  $(ii)$ follows.
  
\noindent  
We now     choose    $T^* =  \min \left(   \frac{1}{2^p \|u_0\|_{L^\infty}^{p-1}},  \frac{\lambda}{2^{p+1} \|u_0\|_{L^\infty}^{p}} , \frac{1}{2 C_2}\right)$. Then,   for all  $T_1 \leq T^*,$  item  $(i)$ and $ (ii)$  hold.   Thanks  to  a Banach fixed-point argument,  there  exists a unique  $u \in B^+_{\lambda}\left(0,  2 \|u_0\|_{L^\infty} \right) $ such that 
  $$\mathcal{Y}(u) (t) = u(t),  \forall t \in (0,T_1],$$ 
  and we easily check that  $u(t)$ satisfies equation \eqref{equ:problem} for all  $ (0,T_1]$ with $u(0) = u_0$. Moreover,  from  the definition  of  $B^+_{\lambda}\left(0,  2 \|u_0\|_{L^\infty} \right) $ we have 
  $$ \text{ Re}(u)(t,x) \geq   \frac{\lambda}{2}.$$
  This  concludes  the proof of Lemma \ref{local-Chauchy-exitece}.
\end{proof}

 \section{Some Taylor expansions}\label{Taylor-expansions-terms}
In this   section appendix,   we  state and prove several  technical  and  straightforward results needed in our paper.

 \noindent

\begin{lemma}[Asymptotics of $\bar B_1, \bar B_2 $]\label{asymptotic-bar-B-1-2}
We consider  $\bar B_1 (\bar w_1, w_2)$ as in \eqref{defini-bar-B-1}, \eqref{defini-bar-B-2}. Then, the following holds:
 \begin{eqnarray}
\bar B_1 (\bar w_1, w_2) &=& \frac{p}{2 \kappa} \bar w^2_1  + O (|\bar w_1|^3 + |w_2|^2)\label{asymptotic-bar-B-1},\\
\bar B_2( \bar w_1, w_2) &=& \frac{p}{ \kappa} \bar w_1 w_2  + O \left( |\bar w_1|^2 |w_2| \right) + O  \left( |w_2|^3 \right).\label{asymptotic-bar-B-2}                                      
\end{eqnarray}
as $(\bar w_1, w_2) \to (0,0)$.
\end{lemma} 
\begin{proof}
The proof  of   \eqref{asymptotic-bar-B-1}  is  quite the same as  the proof of   \eqref{asymptotic-bar-B-2}.  So, we only  prove   \eqref{asymptotic-bar-B-2}, hoping the reader will have  no problem to check   \eqref{asymptotic-bar-B-1}  if necessary.   Since,  $\kappa = (p-1)^{-\frac{1}{p-1}} > 0,$ we derive   $\kappa + \bar w_1 > 0$ when  $\bar w_1$ is near $0$,  so we can   write   $B_2 ( \bar w_1, w_2)$ as follows 
$$ \bar B_2 (\bar w_1, w_2)   =  \left(  (\kappa + \bar w_1)^2  + w_2^2     \right)^{\frac{p}{2}} \sin \left[   p \arcsin \left(   \frac{w_2}{ \sqrt{  (\kappa + \bar w_1)^2   + w_2^2 }} \right)  \right] - \frac{p}{p-1} w_2   ,$$
as $\bar w_1  \to 0 .$  Thus,  
\begin{eqnarray*}
B_2 (\bar w_1, w_2) & = &  \left(  (\kappa + \bar w_1)^2  + w_2^2     \right)^{\frac{p}{2}}   \frac{p w_2}{ \sqrt{  (\kappa + \bar w_1)^2   + w_2^2 }} -  \frac{p}{p-1} w_2\\
& +& \left(  (\kappa + \bar w_1)^2  + w_2^2     \right)^{\frac{p}{2}}  \left\{    \sin \left[   p \arcsin \left(   \frac{w_2}{ \sqrt{  (\kappa + \bar w_1)^2   + w_2^2 }} \right)  \right] -  \frac{p w_2}{ \sqrt{  (\kappa + \bar w_1)^2   + w_2^2 }}  \right\}\\
& = &  \left(  (\kappa + \bar w_1)^2  + w_2^2     \right)^{\frac{p-1}{2}}  p w_2 -  \frac{p}{p-1} w_2 \\
& +& \left(  (\kappa + \bar w_1)^2  + w_2^2     \right)^{\frac{p}{2}}  \left\{    \sin \left[   p \arcsin \left(   \frac{w_2}{ \sqrt{  (\kappa + \bar w_1)^2   + w_2^2 }} \right)  \right] -  \frac{p w_2}{ \sqrt{  (\kappa + \bar w_1)^2   + w_2^2 }}  \right\}\\
&=& (I)  + (II).
\end{eqnarray*}
In addtion to that, we have the fact 
\begin{eqnarray*}
  \sin (p x)         - p x  =  O(|x|^3),\\
\frac{w_2}{ \sqrt{ (\kappa + \bar w_1)^2  + w_2^2}}     =  O(|w_2|),
\end{eqnarray*}
as   $x \to 0$ and  $(\bar w_1, w_2) \to (0,0)$.
Plugging these  estimates in $(II),$ we ontain 
\begin{eqnarray*}
 (II)   = O (|w_2|^3)\label{estima(II)-lema-A-bar-B-2}.
\end{eqnarray*}
as  $(\bar w_1, w_2) \to (0,0)$. For $(I),$ we use a  Taylor expansion   for $ ((\kappa + \bar w_1)^2  + w_2^2)$, around $(\bar w_1, w_2) =  (0,0):$  
$$   ( (\kappa  + \bar w_1)^2 + w_2^2)^{\frac{p}{2}} =  \frac{1}{p-1}  + \frac{(p-1)  }{ \kappa (p-1)} \bar w_1    + O (|\bar w-1|^2)  + O(|w_2|^2).$$
Plugging this in  $(I),$ we derive the following:
$$ (I)  =  \frac{p}{\kappa}  \bar w_1 w_2  + O(|\bar w_1|^2 w_2) + O(|w_2|^3),$$
as $(\bar w_1, w_2) \to (0,0).$ From the estimates of $(I)$ and $(II),$ we conclude the Lemma.
\end{proof}
IN the following lemma,   we aim at  giving  bounds  on the    principal potential $V$ and the   potentials $V_{i,j}$:
\begin{lemma}[The potential functions $V$ and  $ V_{j,k} $ with $ j,k \in \{1,n\}$]\label{lemmas-potentials}  We consider  $ V , V_{1,1},V_{1,2},V_{2,1} $ and $ V_{2,2}$    defined   in \eqref{defini-potentian-V}  and  \eqref{defini-V-1-1} - \eqref{defini-V-2-2}. Then, the following holds:

$(i)$ For all $s \geq 1 $ and  $y \in \mathbb{R}^n$,  we have  $| V(y,s)| \leq C,$ 
   \begin{equation}\label{esti-y-R-n-V-1-1-and-22}
 \left| V (y,s)\right|  \leq \frac{C (1 + |y|^2)}{s},
 \end{equation}
 and 
 \begin{equation}\label{the-potential-V-inside-blowup-region}
	V(y,s) = - \frac{(|y|^2 -2 n  )}{4 s} + \tilde V(y,s),
\end{equation}
where $\tilde V$ satisfies 
\begin{equation}\label{defini-tilde V.}
|\tilde V (y,s)| \leq C \frac{(1 + |y|^4)}{s^2}, \forall s \geq 1, |y| \leq 2 K_0 \sqrt s. 
\end{equation}
 
 $(ii)$  The  potential functions  $V_{j,k}$ with $j,k \in \{ 1,2\}$ satisfy the following
 \begin{eqnarray*}
 \|V_{1,1} \|_{L^\infty}  + \| V_{2,2}\|_{L^\infty}  &\leq &  \frac{C}{s^2},\\
  \|V_{1,2} \|_{L^\infty}  + \| V_{2,1}\|_{L^\infty}  &\leq &  \frac{C}{s},\\
 \left|  V_{1,1} (y,s) \right|  + \left| V_{2,2} (y,s) \right| &\leq & \frac{C (1 + |y|^4)}{s^4},\\ 
 \left|  V_{1,2} (y,s) \right|  + \left| V_{2,1} (y,s) \right| &\leq & \frac{C (1 + |y|^2)}{s^2},
 \end{eqnarray*}
for all $s \geq 1$ and  $y \in \mathbb{R}^n.$
\end{lemma}
\begin{proof}  We  note  that    the proof of    $(i)$   was given  in   Lemma B.1,  page 1270   in \cite{NZens16}. So, it  remains  to  prove  item $(ii)$. Moreover,  the technique for these estimates is the  same, so we only  give the  proof to the following estimates:
\begin{eqnarray}
 \|V_{1,1} \|_{L^\infty}  + \| V_{2,2}\|_{L^\infty}  &\leq &  \frac{C}{s^2}, \label{estima-norm-V-1-1-2_2-s_2}\\
  \left|  V_{1,1} (y,s) \right|  + \left| V_{2,2} (y,s) \right| &\leq & \frac{C (1 + |y|^4)}{s^4}.\label{estimate-absol-V-1-1-V-2-2}
\end{eqnarray} 

+ \textit{The proof of \eqref{estima-norm-V-1-1-2_2-s_2}:} We recall the expression  of  $V_{1,1}$ and $V_{2,2}:$
\begin{eqnarray*}
V_{1,1} &=&  \partial_{u_1} F_1( u_1, u_2) |_{(u_1,u_2) =  (\Phi_1, \Phi_2)}  -  p  \Phi_1^{p-1},\\
V_{2,2} & = & \partial_{u_2} F_2 (u_1,u_2) |_{(u_1,u_2) = (\Phi_1, \Phi_2) }  - p \Phi_1^{p-1}, 
\end{eqnarray*}
where  $\Phi_1, \Phi_2$ are   given  by \eqref{defini-V-1-1} and \eqref{defini-V-2-2}.  Hence, we can  rewrite $V_{1,1} $ and  $V_{2,2}$ as follows 
\begin{eqnarray*}
V_{1,1} & =&    p (u_1^2 + u_2^2)^{\frac{p}{2}}  \left( u_1 \cos \left[  p  \arcsin \left(  \frac{\Phi_2}{\sqrt{ \Phi_1^2  + \Phi_2^2}} \right)\right]    -  u_2 \sin \left[  p  \arcsin \left(  \frac{\Phi_2}{\sqrt{ \Phi_1^2  + \Phi_2^2}} \right)\right]      \right) \\
& -&  p \Phi_1^{p-1},\label{formular-exactl-V-1-1}\\
V_{2,2} & =  & p (u_1^2 + u_2^2)^{\frac{p}{2}}  \left( u_1 \cos \left[  p  \arcsin \left(  \frac{\Phi_2}{\sqrt{ \Phi_1^2  + \Phi_2^2}} \right)\right]      \right) + u_2 \sin \left[  p  \arcsin \left(  \frac{\Phi_2}{\sqrt{ \Phi_1^2  + \Phi_2^2}} \right)\right]    \label{formular-exactl-V-2-2} \\
& -&  p \Phi_1^{p-1},\\
 \end{eqnarray*}
 We first estimate to $V_{1,1}$,    from  the above   equalities,    we decompose $V_{1,1}$ into the following  
 \begin{eqnarray}
 V_{1,1} =    V_{1,1,1} + V_{1,1,2}  + V_{1,1,3},\label{decom-pose-V-1-1}
 \end{eqnarray}
 where 
 \begin{eqnarray*}
 V_{1,1,1}  &=&    p \left(  \Phi_1^2     +     \Phi_2^2 \right)^{\frac{p-2}{2}}  \Phi_1 -  p \Phi_1^{p-1},\\
 V_{1,1,2} & =&   p \left(  \Phi_1^2     +     \Phi_2^2 \right)^{\frac{p-2}{2}} \Phi_1 \left(    \cos\left[  p  \arcsin \left(  \frac{\Phi_2}{\sqrt{ \Phi_1^2  + \Phi_2^2}} \right)\right]  - 1 \right),\\
V_{1,1,3} & = &  - p (\Phi_1^2 + \Phi_2^2)^{\frac{p-2}{2}}  \Phi_2  \sin \left[  p  \arcsin \left(  \frac{\Phi_2}{\sqrt{ \Phi_1^2  + \Phi_2^2}} \right)\right]. 
 \end{eqnarray*}
 As matter of fact,  from the  definitions of  $\Phi_1, \Phi_2 ,$  we have the following
\begin{eqnarray}
\left\|   \frac{\Phi_2(.,s)}{\Phi_1(.,s)} \right\|_{L^\infty} &\leq & \frac{C }{s}, \label{inequa-Phi-2-Phi-1}\\
\| \Phi_1 (.,s)\|_{L^\infty}  & \leq &  C \label{norn-Phi-1-bound},\\
\| \Phi_2 (.,s)\|_{L^\infty}  & \leq  &  \frac{C}{s}, \label{norm-Phi-2-bound}  
\end{eqnarray} 
    for all $s \geq 1$ and 
\begin{eqnarray}
 \left|    \cos ( p  \arcsin x) -  1   \right| \leq  C |x|^2, \label{inequa-cos-p-arcsin}\\
 \left|   \sin  (p \arcsin x)    - x  \right|  \leq  C|x|^3,\label{inequa-sin-p-arcsin}
\end{eqnarray}
 for all  $|x| \leq  1$. By using  \eqref{inequa-Phi-2-Phi-1}, \eqref{norn-Phi-1-bound}, \eqref{norm-Phi-2-bound},  \eqref{inequa-cos-p-arcsin} and  \eqref{inequa-sin-p-arcsin}, we  get the following bound for   $V_{1,1,2}$ and $V_{1,1,3}$
 \begin{equation}\label{norm-V-1-1-2-3-infi-s-2}
 \| V_{1,1,2} (.,s)\|_{L^\infty}   +   \| V_{1,1,3} (.,s)\|_{L^\infty} \leq    \frac{C}{s^2}.
 \end{equation}
 For  $V_{1,1,1},$ using  \eqref{inequa-Phi-2-Phi-1}, we derive
 \begin{eqnarray*}
 | V_{1,1,1} |  =   \left|  p \Phi_1^{p-1}   \left(     \left(   1  +   \frac{\Phi_2^2}{\Phi_1^2}    \right) ^{\frac{p-2}{2}}- 1  \right)  \right|    \leq  \frac{C}{s^2}.
 \end{eqnarray*}
 This  gives the following  
 $$ \|V_{1,1} (.,s)\|_{L^\infty}  \leq \frac{C}{s^2}.$$
 We  can apply the technique to   $V_{2,2}$  to get a similar estimate as follows
 $$  \| V_{2,2}(.,s)\|_{L^\infty}   \leq \frac{C}{s^2}.$$
Then,   \eqref{estima-norm-V-1-1-2_2-s_2} follows.
  
  + \textit{The proof of     \eqref{estimate-absol-V-1-1-V-2-2}:} We can see that  on the domain  $\{ |y|   \geq K_0   \sqrt{s} \}$   we have  
  $$   \frac{1 + |y|^4}{s^4}  \geq \frac{C}{s^2},$$
  and   in particular,   we    have \eqref{estima-norm-V-1-1-2_2-s_2}.   Thus,     for all   $|y|  \geq  K_0 \sqrt s. $
  $$   |V_{1,1} (y,s)|   + |V_{  2,2}  (y,s) |   \leq  \frac{C   (|y|^4  + 1)}{s^4}.$$
  Therefore, it is sufficient to  give the estimate   on the domain  $\{ |y|  \leq 2 K_0 \sqrt s\}$.    On this domain,    we have the following: there existes   $C (K_0) > 0$ such that  
  $$    \frac{1}{C}   \leq     \Phi_1 (y,s)    \leq     C.$$
    In addition to that,    using  the definition  of  $\Phi_2$   given by  \eqref{defi-Phi-1}, we  derive the following  
 \begin{equation}\label{estima-Phi-2-y-s}
 \left|   \Phi_2 (y,s)\right| \leq   C \frac{( |y|^2  + 1 )}{s^2}, \forall  (y,s) \in \mathbb{R}^n \times [1, + \infty).
\end{equation}  
Then,  from  \eqref{decom-pose-V-1-1}    we have 
\begin{eqnarray*}
\left|   V_{1,1,2}  (y,s)    \right|   & \leq  &   |\Phi_2 (y,s)|^2    \leq  C  \frac{ (1 +  |y|^4)}{s^4},\\
\left| V_{1,1,3} (y,s)\right| & \leq &    \left|   \Phi_2 (y,s)\right|^2 \leq     C  \frac{ (1 +  |y|^4)}{s^4}.  
\end{eqnarray*}
We now estimate  $V_{1,1,1}$,   thanks to   a Taylor  expansion  of  $( \Phi_1^2  + \Phi_2^2)^{\frac{p-2}{2}}     ,$ around    $\Phi_2$ 
$$ \left|   (\Phi_1^2 + \Phi_2^2)^{\frac{p-2}{2}}      -    \Phi_1^{p-2}   \right|  \leq   C |  \Phi_2 |^2 .$$
This   directly   yields   
$$  \left|   V_{1,1,1}(y,s) \right|   \leq   C(K_0)  | \Phi_2|^2   \leq   C   \frac{(1 + |y|^4)}{s^4}. $$    
So,      
$$\left|     V_{1,1} (y,s) \right|    \leq  C  \frac{(1 +  |y|^4)}{s^4} , \forall  y  \in \mathbb{R}^n.$$
Moreover,  we can   proceed similarly      for  $V_{2,2}$, and   get 
$$   |  V_{2,2}  (y,s) |    \leq      C  \frac{(1 +  |y|^4)}{s^4} \forall   y \in   \mathbb{R}^n.$$
Thus,   \eqref{estimate-absol-V-1-1-V-2-2}  follows.
\end{proof}

Now, we give some estimates on the nonlinear terms $B_1(q_1, q_2)$ and   $B_2(q_1, q_2)$

\begin{lemma}[The terms $B_1 (q_1,q_2)$ and $B_2 (q_1,q_2)$]\label{lemma-quadratic-term-B-1-2} We consider $B_1 (q_1,q_2),B_2 (q_1,q_2)$  as defined  in \eqref{defini-quadratic-B-1} and  \eqref{defini-term-under-linear-B-2},  respectively. For all $A \geq 1,$ there exists $s_9 (A) \geq 1$ such that for all $s_0 \geq s_{9} (A),$   if  $(q_1,q_2) (s) \in V_A(s)$ and  $q_1 (s) + \Phi_1 (s) \geq \frac{1}{2} e^{-\frac{s}{p-1}}$ for all  $s \in [s_0, s_1]$, then, the  following holds: for all  $s \in [s_0, s_1],$
\begin{eqnarray}
\left| \chi (y,s)  B_1 (q_1,q_2)   \right| &\leq &  C \left( |q_1|^2 + |q_2|^2  \right),\label{estimate-B-1-q-1-2-inside-blowup}\\
| \chi (y,s) B_2 (q_1,q_2)| & \leq  & C   \left(  \frac{|q_1|^2}{s} +  |q_1| |q_2|   +  |q_2|^2  \right),\label{estimate-B-2-q-1-2-inside-blowup}\\
\|B_1(q_1,q_2)\|_{L^\infty} &\leq &   \frac{C A^{4}}{s^{\frac{p}{2}}},\label{estimate-B-1-q-1-2-allspace} \\
\| B_2 (q_1,q_2)\|_{L^\infty} & \leq & \frac{C A^{2}}{s^{  1 +   \min\left(  \frac{p-1}{4}, \frac{1}{2}\right)   }},\label{estimate-B-2-q-1-2-allspace}
\end{eqnarray}
where  $\chi(y,s)$ is defined  as in \eqref{def-chi}.
\end{lemma}           
\begin{proof}
We first  would like to note that  the condition   $ q_1 (s) + \Phi_1 (s) \geq \frac{1}{2} e^{-\frac{s}{p-1}}  $ is to ensure that the real part  $w_1 =  q_1 (s) + \Phi_1 (s) > 0$. Then,   \eqref{condition-Re-u-geq-0} holds   and  $F_1, F_2$ which  iare involved  in  the definition of $B_1, B_2$, are well-defined (see \eqref{defi-mathbb-A-1-2}).      For the proof of Lemma  \ref{lemma-quadratic-term-B-1-2},  we only   prove  for  \eqref{estimate-B-2-q-1-2-inside-blowup} and \eqref{estimate-B-2-q-1-2-allspace}, because the other ones follow  similarly.

+ \textit{The proof for \eqref{estimate-B-2-q-1-2-inside-blowup}:}  Using the  fact that   the support of $\chi (y,s)$ is  $\{  |y| \leq 2 K_0 \sqrt s \},$      it is enough  to prove    \eqref{estimate-B-2-q-1-2-inside-blowup}     for all  $\{  |y| \leq 2 K_0 \sqrt s \}$.  Since  we have  $(q_1,q_2) \in V_A (s),$ we derive from  item  $(ii)$ of Lemma  \ref{lemma-analysis-V-A} and the  definition of  $\Phi_1, \Phi_2$ that
\begin{eqnarray*}
   \frac{1}{C } \leq   q_1   + \Phi_1    \leq C ,  \quad  \left|   q_2  + \Phi_2 \right| \leq \frac{C}{s}.
\end{eqnarray*}
and
 \begin{equation}\label{bound-q-1-2-inside)blowup}
 |q_1| \leq  \frac{C A }{  \sqrt s},   \quad  |q_2| \leq   \frac{CA^2}{s^{\frac{p_1 + 2}{2}}}, \forall |y| \leq  2 K_0 \sqrt s.
 \end{equation}
In addition to that, we write  $B_2 (q_1, q_2)$ as follows: 
\begin{eqnarray*}
B_2(q_1, q_2) & = &   F_2 \left(  \Phi_1 + q_1, \Phi_2 + q_2 \right) - F_2(\Phi_1, \Phi_2) -  \partial_{u_1} F_2(q_1 + \Phi_1, q_2 + \Phi_2)q_1 \\
&  -  &  \partial_{u_2} F_2(q_1 + \Phi_1, q_2 + \Phi_2) q_2.
\end{eqnarray*}
where 
$$  F_2 (u_1,u_2) =    \left(  u_1^2  + u_2^2   \right)^{\frac{p}{2}}   \sin \left[ p \arcsin\left(  \frac{u_2}{ \sqrt{u_1^2 + u_2^2}} \right)\right].$$
Using  a   Taylor expansion  for  the function  $F_2 (q_1 + \Phi_1, q_2  + \Phi_2)$ at $(q_1, q_2) = (0,0),$  we derive the following  
\begin{eqnarray*}
F_2 ( q_1 + \Phi_1,  q_2 + \Phi_2)  &=&  \sum_{j  +  k \leq 4}\frac{1}{j! k!} \partial^{j+k}_{q_1^jq_2^k} ( F_2 (q_1 + \Phi_1, q_2 + \Phi_2 )) \left|_{(q_1,q_2) = (0,0)} \right.    q_1^j q_2^k  +   \\
&+&  \sum_{j + k  = 5}   G_{j,k} (q_1, q_2)  q_1^j q_2^k,
\end{eqnarray*}
where  
$$    G_{j,k}  (q_1, q_2)  = \frac{5}{j! k!}  \int_{0}^1  (1 -t)^4 \partial^5_{q_1^j q_2^k} (F_2  (  \Phi_1 +  tq_1 , \Phi_2 + t q_2   )) dt.$$
In particular, we have 
$$ | G_{j,k} (q_1,q_2) | \leq C , \forall   j + k = 4.$$
As a matter of fact, we have 
 \begin{eqnarray}
 \partial^{j+k}_{q_1^jq_2^k} ( F_2 (q_1 + \Phi_1, q_2 + \Phi_2 )) \left|_{(q_1,q_2) = (0,0)} \right.  =  \partial^{j+k}_{u_1^j u_2^k} F_2 (u_1,u_2) \left|_{(u_1,u_2) =  (0,0)} \right.
 \end{eqnarray}
Therefore, from   \eqref{bound-q-1-2-inside)blowup}, we have
\begin{eqnarray*}
& & \left|   F_2 (q_1 + \Phi_1, q_2 + \Phi_2) -   \sum_{j  +  k \leq 3}\frac{1}{j! k!} \partial^{j+k}_{u_1^j
u_2^k}  F_2 ( u_1, u_2 ) \left|_{(u_1,u_2) = (\Phi_1, \Phi_2)} \right.    q_1^j q_2^k    \right|  \\
& \leq  &C    \sum_{j=0}^5| q_1^j q_2^{5-j}|   \leq  C \left( \frac{|q_1|^2}{s} +  |q_1| |q_2|   +  |q_2|^2  \right).
\end{eqnarray*}
In addition to  that, we have the following fact,
\begin{eqnarray*}
| \partial_{u_1^j u_2^k}^{j+k} F_2  (u_1, u_2)\left|_{(u_1, u_2) = (\Phi_1, \Phi_2)} \right.  | \leq C, \forall  j+ k \leq 3,
\end{eqnarray*} 
and for all $ 1 \leq j \leq 4,$ we have
$$ \left| \partial_{u_1^j }^j F_2  (u_1, u_2)\right|_{(u_1, u_2) = (\Phi_1, \Phi_2)}  \leq \frac{C}{s}.$$
This   concludes    \eqref{estimate-B-2-q-1-2-inside-blowup}.

\textit{  The proof   of  \eqref{estimate-B-2-q-1-2-allspace}:}  We rewrite  $B_2(q_1, q_2)$  explicitly  as follows:
\begin{eqnarray*}
B_2 (q_1, q_2)  &=&   \left( (q_1  + \Phi_1)^2  + (q_2 + \Phi_2)^2 \right)^{\frac{p}{2}} \sin \left[  p \arcsin \left(  \frac{q_2 + \Phi_2}{  \sqrt{ (q_1 + \Phi_1)^2  + (q_2 + \Phi_2)^2}}\right)  \right] \\
& - &  (\Phi_1^2   + \Phi_2^2)^{\frac{p}{2}}  \sin \left[  p  \arcsin \left(  \frac{\Phi_2}{ \sqrt{ \Phi_1^2 + \Phi_2^2}}\right)\right]\\
&- & p\left( \Phi_1^2 + \Phi_2^2\right)^{\frac{p-2}{2}} \left(  \Phi_1  \sin \left[  p  \arcsin \left(  \frac{\Phi_2}{ \sqrt{ \Phi_1^2 + \Phi_2^2}}\right)\right]  - \Phi_2 \cos \left[  p  \arcsin \left(  \frac{\Phi_2}{ \sqrt{ \Phi_1^2 + \Phi_2^2}}\right)\right]\right)  q_1\\
& - & p\left( \Phi_1^2 + \Phi_2^2\right)^{\frac{p-2}{2}} \left(  \Phi_2  \sin \left[  p  \arcsin \left(  \frac{\Phi_2}{ \sqrt{ \Phi_1^2 + \Phi_2^2}}\right)\right]  + \Phi_1 \cos \left[  p  \arcsin \left(  \frac{\Phi_2}{ \sqrt{ \Phi_1^2 + \Phi_2^2}}\right)\right]\right)  q_2.
\end{eqnarray*}
Then, we decompose $B_2 (q_1, q_2)$  as follows:
\begin{eqnarray*}
B_2 (q_1, q_2) =  B_{2,1} (q_1, q_2)  + B_{2,2} (q_1, q_2)   + B_{2,3} (q_1,q_2) + B_{2,4} (q_1,q_2)  + B_{2,5} (q_1,q_2)   + B_{2,6} (q_1,q_2),
\end{eqnarray*}
where 
\begin{eqnarray}
B_{2,1} (q_1,q_2) & =& p ( q_2 + \Phi_2) \left( (q_1  + \Phi_1)^2  + (q_2 + \Phi_2)^2 \right)^{\frac{p-1}{2}}     - p (\Phi_1^2 + \Phi_2^2)^{\frac{p-1}{2}} \Phi_2 \label{defini-B-2-1-app} \\
& - &   p \left( \Phi_1^2 + \Phi_2^2 \right)^{\frac{p-2}{2}} \Phi_1 q_2,      \nonumber\\
B_{2,2} (q_1,q_2)& =&    ((q_1 + \Phi_1)^2 + (q_2 + \Phi_2)^2)^{\frac{p}{2}} \left\{ \sin \left[  p \arcsin \left(  \frac{q_2 + \Phi_2}{  \sqrt{ (q_1 + \Phi_1)^2  + (q_2 + \Phi_2)^2}}\right)  \right]  \right.\nonumber\\
& -  & \left.   \frac{p (q_2 + \Phi_2)}{  \sqrt{ (q_1 + \Phi_1)^2  + (q_2 + \Phi_2)^2}}  \right\},\label{defini-B-2-2-app} \\
B_{2,3} (q_1,q_2)  & = &   \left(  \Phi_1^2 + \Phi_2^2\right)^{\frac{p}{2}} \left(  \frac{p \Phi_2}{ \sqrt{\Phi_1^2 + \Phi_2^2}}   -   \sin \left[  p  \arcsin \left(  \frac{\Phi_2}{ \sqrt{ \Phi_1^2 + \Phi_2^2}}\right)\right]   \right), \label{defi-B-2-3-app}\\
B_{2,4} (q_1,q_2) & =&   p (\Phi_1^2 + \Phi_2^2)^{\frac{p-2}{2}} \Phi_1 \left( 1-  \cos \left[  p  \arcsin \left(  \frac{\Phi_2}{ \sqrt{ \Phi_1^2 + \Phi_2^2}}\right)\right] \right) q_2,\label{defi-B-2-4-app}\\
B_{2,5} (q_1,q_2) & =& p\left( \Phi_1^2 + \Phi_2^2\right)^{\frac{p-2}{2}} \left\{ \Phi_2 \cos \left[  p  \arcsin \left(  \frac{\Phi_2}{ \sqrt{ \Phi_1^2 + \Phi_2^2}}\right)\right]  \right. \nonumber\\
&-& \left.  \Phi_1  \sin \left[  p  \arcsin \left(  \frac{\Phi_2}{ \sqrt{ \Phi_1^2 + \Phi_2^2}}\right)\right] \right\} q_1 \label{defi-B-2-6-app} ,\\
B_{2,6} (q_1,q_2) & =&  - p (\Phi_1^2 + \Phi_2^2)^{\frac{p-2}{2}}   \Phi_2 \sin \left[  p  \arcsin \left(  \frac{\Phi_2}{ \sqrt{ \Phi_1^2 + \Phi_2^2}}\right)\right] q_2. \label{defi-B-2-7}
\end{eqnarray} 
we  prove  that: for all  $ y \in \mathbb{R}^n$:
$$  |B_{2,j} (q_1 ,q_2)| \leq  \frac{C A^{2}}{s^{1+  \min\left( \frac{p-1}{4}, \frac{1}{2}\right)}}, \forall  j = 1,...,6.$$ 
We now aim at  an  estimate    on  $B_{2,1} (q_1,q_2)$:     We first need to prove the following:   
\begin{equation}\label{estima-Phi-1-2-q-1-2-Phi-1-2}
\left|     \left(   (\Phi_1  +  q_1)^2  + ( \Phi_2  + q_2)^2 \right)^{\frac{p-1}{2}} - (\Phi_1^2  + \Phi_2^2)^{\frac{p-1}{2}}  \right|  \leq      C  \left|Z \right|^{\min \left(  \frac{p-1}{2},1\right)}, 
\end{equation}   
where   
$$ | Z|   =  2 q_1 \Phi_1  + 2 q_2 \Phi_2  +  q_1^2  +  q_2^2.$$
Note that    $Z$  is bounded.  On   the other hand,      we have   $\left(\Phi_1  +  q_1)^2  + ( \Phi_2  + q_2)^2 \right)^{\frac{p-1}{2}}   =  (\Phi_1^2 + \Phi_2^2 + Z)^{\frac{p-1}{2}}.$ Then,  if   $\frac{p-1}{2}  \geq 1$, using  a Taylor expansion     of  the   function  $( \Phi_1^2  + \Phi_2^2 + Z)^{\frac{p-1}{2}} $ around   $Z_0 = 0$ (note that   $\Phi_1^2  +  \Phi_2^2$ is     uniformly bounded),  we obtain   the following: 
$$     \left|     \left(   (\Phi_1  +  q_1)^2  + ( \Phi_2  + q_2)^2 \right)^{\frac{p-1}{2}} - (\Phi_1^2  + \Phi_2^2)^{\frac{p-1}{2}}  \right|  \leq      C  \left|Z \right|,  $$
which yields    \eqref{estima-Phi-1-2-q-1-2-Phi-1-2}.  If   $\frac{p-1}{2}   <1,  $  then,  we have 
\begin{eqnarray*}
 \left|     \left(   (\Phi_1  +  q_1)^2  + ( \Phi_2  + q_2)^2 \right)^{\frac{p-1}{2}} - (\Phi_1^2  + \Phi_2^2)^{\frac{p-1}{2}}  \right|  =    \left(   \Phi_1^2  + \Phi_2^2\right)^{\frac{p-1}{2}} \left|   \left( 1 +   \xi  \right)^{\frac{p-1}{2}}     - 1 \right|,
\end{eqnarray*}
where    
$$\xi    =  \frac{Z}{   \Phi_1^2  + \Phi_2^2}. $$
In particular, we have    $\xi   \geq   -1$.
 In addition to  that, we have the following fact:   for all  $\xi  \geq   -1$
\begin{equation}\label{pro-1+xi-p-1-2-leq-xi}
\left|   \left( 1 +   \xi  \right)^{\frac{p-1}{2}}     - 1 \right|  \leq  C \left| \xi  \right|^{\frac{p-1}{2}}
\end{equation}
Therefore,   \eqref{pro-1+xi-p-1-2-leq-xi}   gives   the following
\begin{eqnarray*}
 \left|     \left(   (\Phi_1  +  q_1)^2  + ( \Phi_2  + q_2)^2 \right)^{\frac{p-1}{2}} - (\Phi_1^2  + \Phi_2^2)^{\frac{p-1}{2}}  \right|   \leq C  \left(   \Phi_1^2  + \Phi_2^2\right)^{\frac{p-1}{2}}   \left| \frac{Z}{\Phi_1^2 +  \Phi_2^2} \right|^{\frac{p-1}{2}}  \leq   C\left|Z   \right|^{\frac{p-1}{2}}.
\end{eqnarray*}
Then,       \eqref{estima-Phi-1-2-q-1-2-Phi-1-2} follows.  Using  $(q_1,q_2) (s) \in V_A(s)$  and   $Z =   2 \Phi_1 q_1  + 2 \Phi_2 q_2 +  q_1^2  + q_2^2,$ we write
$$   \|Z\|_{L^\infty}   \leq   \frac{CA^2}{  \sqrt s  }, \forall   s \geq  1.   $$
So,  we deduce from      \eqref{estima-Phi-1-2-q-1-2-Phi-1-2} that
\begin{equation}\label{part-1-B-2-1}
\|  p \Phi_2   \left(  ((\Phi_1   + q_1)^2  + (\Phi_2   + q_2)^2)^{\frac{p-1}{2}}\right)     - p \Phi_2 (\Phi_1^2  + \Phi_2^2)^{\frac{p-1}{2}}  \|_{L^\infty} \leq   \frac{C A^2}{s^{1  +    \min  \left( \frac{p-1}{4}, \frac{1}{2} \right)}}.
\end{equation}
Using   \eqref{estima-Phi-1-2-q-1-2-Phi-1-2},  we have    the following 
\begin{equation}\label{Phi-1-q-1-p-1-2-p-2-1}
\left\|     \left(  (\Phi_1 + q_1)^2  + (\Phi_2  + q_2)^2 \right)^{\frac{p-1}{2}}   -  (\Phi_1^2  + \Phi_2^2)^{\frac{p-2}{2}}  \Phi_1  \right\|_{L^\infty}   \leq    \frac{CA^2}{s^{\min \left(\frac{p-1}{4},   \frac{1}{2} \right)}}.
\end{equation}
  Indeed,  we have
  \begin{eqnarray*}
  \left|    \left(  (\Phi_1 + q_1)^2  + (\Phi_2  + q_2)^2 \right)^{\frac{p-1}{2}}   -  (\Phi_1^2  + \Phi_2^2)^{\frac{p-2}{2}}  \Phi_1       \right|   &\leq &  \left|   \left(  (\Phi_1 + q_1)^2  + (\Phi_2  + q_2)^2 \right)^{\frac{p-1}{2}}   -  (\Phi_1^2  + \Phi_2^2)^{\frac{p-1}{2}} \right|\\
  &  +  & \left|   (\Phi_1^2  + \Phi_2^2)^{\frac{p-1}{2}}   -    (\Phi_1^2  + \Phi_2^2)^{\frac{p-2}{2}}  \Phi_1  \right|\\
  & \leq   &  \frac{C A^2}{s^{\frac{ \min \left(\frac{p-1}{2},   1  \right)}{2}}}   +   \frac{C}{s^2}.
  \end{eqnarray*}
Then,   \eqref{Phi-1-q-1-p-1-2-p-2-1} holds.

 On the  other hand, using     \eqref{Phi-1-q-1-p-1-2-p-2-1}     and the  following  
$$ \|q_2 (.,s)\|_{L^\infty}  \leq   \frac{C A^3}{s^{\frac{p_1 + 2}{ 2}}} ,  p_1  >  0,$$
we conclude that
\begin{equation}\label{result-norm-q-2-Phi-1-2-p-1}
 \left\|    p q_2    \left(  (\Phi_1 + q_1)^2  + (\Phi_2  + q_2)^2 \right)^{\frac{p-1}{2}}   -  (\Phi_1^2  + \Phi_2^2)^{\frac{p-2}{2}}  \Phi_1   \right\|_{L^\infty}     \leq   \frac{C A^2}{s^{1+   \min \left(\frac{p-1}{4},   \frac{1}{2} \right)}},
\end{equation}
provided that   $s \geq  s_{1,9} (A)$.  From    \eqref{part-1-B-2-1}  and  \eqref{result-norm-q-2-Phi-1-2-p-1},     we have    
\begin{equation}\label{norm-B-1}
\| B_{2,1}  (q_1,q_2)\|_{L^\infty} \leq   \frac{CA^2}{ s^{1+   \min \left(\frac{p-1}{4},   \frac{1}{2}  \right)}}.
\end{equation}
We next give a bound   to  $B_{2,2} (q_1,q_2):$  Using      the following fact 
$$ \left|    \sin   (p \arcsin x)       -   x\right| \leq C |x|^3, \forall  |x| \leq  1,$$
we  derive the following
\begin{eqnarray*}
& &\left|  \sin \left[   p  \arcsin \left(  \frac{q_2 + \Phi_2}{  \sqrt{ (q_1 + \Phi_1)^2  + (q_2 + \Phi_2)^2}}\right)  \right]  -     \frac{p (q_2 + \Phi_2)}{  \sqrt{ (q_1 + \Phi_1)^2  + (q_2 + \Phi_2)^2}}   \right|  \\
& \leq   &   C \frac{| (q_2  + \Phi_2)|^3}{ ( (q_1 + \Phi_1)^2  + (q_2 + \Phi_2)^2)^{\frac{3}{2}}}.
\end{eqnarray*}
Plugging the above estimate  into  $B_2 (q_1, q_2),$ we deduce the following
$$ | B_{2,2} (q_1,q_2) |  \leq C  \left(   (q_1  +  \Phi_1)^2   + (\Phi_2 + q_2)^2  \right)^{\frac{p-3}{2}}  \left|      q_2  + \Phi_2 \right|^3,$$
which yields 
\begin{eqnarray*}
\left| B_{2,2}  (q_1,q_2) \right|   \leq  C | q_2 + \Phi_2  |^{\min  \left(  p, 3\right)},
\end{eqnarray*}
Using    $(q_1, q_2)  \in V_A (s),$     it gives the following
$$   |  q_2 +  \Phi_2|   \leq   \frac{C }{s},$$
  provided that    $ s \geq   s_{2,9} (A)$. Then, 
  \begin{equation}\label{norm-B-2-2}
  \|B_{2,2}(q_1, q_2)\|_{L^\infty}   \leq   \frac{C}{s^{\min  \left(  p, 3\right)}}.
\end{equation}    
It is similar  to  estimate to   $B_{2,3}  (q_1,q_2)$   
\begin{equation}\label{norm-B-2-3}
\|  B_{2,3}(q_1,q_2)\|_{L^\infty}   \leq   \frac{C}{ s^3}.
\end{equation}
We estimate to   $B_{2,4} (q_1,q_2)$,  using the following   
$$    \left|     1  - \cos (p  \arcsin x) \right|  \leq   C |x|^2, \forall   |x|  \leq   1,$$
we write 
$$ |  B_{2,4} (q_1, q_2) |  \leq   C    \left\|  \frac{\Phi_2}{\Phi_1} \right\|^2_{L^\infty}  \| q_2 \|_{L^\infty}  \leq   \frac{C  A^3}{ s^{3}}.$$
Then, we derive that 
\begin{equation}\label{norm-B-2-4}
\|B_{2,4}  (q_1, q_2)\|_{L^\infty}    \leq  \frac{C A^3 }{s^3}.
\end{equation}
We also    estimate to   $B_{2,5}, B_{2,6}$  as follows:
\begin{eqnarray}
\|  B_{2,5}  (q_1, q_2) \|_{L^\infty}   \leq   \frac{CA^2}{s^{\frac{3}{2}}} \label{norm-B-2-5},\\
\|   B_{2,6} (q_1, q_2)\|_{L^\infty}  \leq   \frac{CA^3}{s^2}.\label{norm-B-2-6}
\end{eqnarray}
Thus,   from   \eqref{norm-B-1},  \eqref{norm-B-2-2}, \eqref{norm-B-2-3}, \eqref{norm-B-2-4}, \eqref{norm-B-2-5} and    \eqref{norm-B-2-6},     we  conlude    \eqref{estimate-B-2-q-1-2-allspace}, provided that   $s  \geq  s_{3,9} (A)$.
\end{proof}

In the following Lemma, we aim at giving  estimates  to  the rest terms  $R_1, R_2:$ 
 \begin{lemma}[The rest terms $R_1, R_2$]\label{lemma-rest-term-R-1-2} 
For all $s \geq 1,$ we consider  $R_1, R_2$  defined  in \eqref{defini-the-rest-term-R-1} and \eqref{defini-the-rest-term-R-2}. Then,  
\begin{itemize}
\item[$(i)$]  For all $s \geq 1$ and $y \in \mathbb{R}^n$
\begin{eqnarray*}
R_1 (y,s) &= & \frac{c_{1,p}}{ s^2}   +  \tilde R_1 (y,s),\\
 R_2 (y,s)  &=&  \frac{c_{2,p}}{s^3}  + \tilde R_2 (y,s),
\end{eqnarray*}
where $c_{1,p} $and  $ c_{2,p}$ are constants depended on $p$ and   $\tilde  R_1, \tilde  R_2$ satisfy
\begin{eqnarray*}
|\tilde R_1 (y,s) | & \leq &  \frac{C (1 + |y|^4)}{s^3},\\
| \tilde R_2 (y,s) | & \leq &  \frac{C (1 + |y|^6)}{s^4},
\end{eqnarray*}
for all  $|y| \leq  2 K_0 \sqrt s.$
\item[$(ii)$] Moreover, we have for all $s \geq 1$
\begin{eqnarray*}
\| R_1(.,s)\|_{L^\infty(\mathbb{R}^n)}  & \leq  & \frac{C}{s},\\
\| R_2(.,s)\|_{L^\infty(\mathbb{R}^n)}  & \leq & \frac{C}{s^2}.
\end{eqnarray*}
\end{itemize}
\end{lemma}
\begin{proof}
   The proof for $R_1$   is quite the same  as the proof   for    $R_2 $.  For that  reason, we only give the proof of  the estimates   on $R_2$.  This  means that,  we need to  prove the following estimates:
\begin{equation}\label{estimate-R-2-inside}
R_2 (y,s)  = -  \frac{n(n+4) \kappa}{(p-1)s^3}  + \tilde R_2 (y,s),
\end{equation}
with 
$$ | \tilde R_2 (y,s) |  \leq    \frac{C (1 + |y|^6)}{s^4},  \forall   |y|  \leq  2 K_0 \sqrt s.$$
and  
\begin{equation}\label{estima-all-spa-R-2}
\| R_2(.,s)\|_{L^\infty}  \leq \frac{C}{s^2}.
\end{equation}
  We recall the definition of $R_2(y,s)$:
$$  R_2 (y,s) = \Delta \Phi_2 - \frac{1}{2} y \cdot \nabla \Phi_2  - \frac{\Phi_2}{p-1} + F_2 (\Phi_1, \Phi_2) - \partial_s \Phi_2, $$
Then, we can rewrite  $R_2$ as follows
$$R_2 (y,s) =  \Delta \Phi_2 - \frac{1}{2} y \cdot \nabla \Phi_2  - \frac{\Phi_2}{p-1} + p \Phi_1^{p-1} \Phi_2 - \partial_s \Phi_2  + R_2^* (y,s),$$
where 
\begin{eqnarray*}
R_2^*(y,s) &=&  \left( \Phi_1^2   + \Phi_2^2 \right)^{\frac{p}{2}}  \sin \left[ p  \arcsin  \left( \frac{\Phi_2}{\sqrt{ \Phi_1^2   + \Phi_2^2}}  \right)\right]      - p \Phi_1^{p-1} \Phi_2.
\end{eqnarray*}
\noindent
 Using the defintions of  $\Phi_1, \Phi_2$  given in  \eqref{defi-Phi-1} and \eqref{defi-Phi-2},  we obtain the following: 
\begin{eqnarray*}
\left| R_2^* (y,s) \right|  & \leq &  \left|    \left( \Phi_1^2   + \Phi_2^2 \right)^{\frac{p}{2}} \left\{  \sin \left[ p  \arcsin  \left( \frac{\Phi_2}{\sqrt{ \Phi_1^2   + \Phi_2^2}}  \right)\right]    -p \frac{\Phi_2}{\sqrt{ \Phi_1^2   + \Phi_2^2}}    \right\}   \right|  \\
& +&  \left|   p \Phi_2 (   (\Phi_1^2  + \Phi_2^2)^{\frac{p-1}{2}}    -  \Phi_1^{p-1}  )  \right|.
\end{eqnarray*}   
It is similar to  the proofs of  estimations   given in     the proof  of Lemma  \ref{lemma-quadratic-term-B-1-2}, we can prove the following
 
\begin{eqnarray*}
|R_2^* (y,s) | & \leq  &\frac{C ( 1 + | y|^6)}{s^4}, \quad \forall  |y| \leq 2 K_0 \sqrt s,\\
& \text{ and } &\\
\| R_2^*(.,s)\|_{L^\infty } & \leq  &\frac{C}{s^2}.
\end{eqnarray*} 
In addition to that, we introduce $\bar  R_2$ as follows: 
$$ \bar R_2 (y,s) =   \Delta \Phi_2 - \frac{1}{2} y \cdot \nabla \Phi_2  - \frac{\Phi_2}{p-1} + p \Phi_1^{p-1} \Phi_2 - \partial_s \Phi_2.$$
Then,   we  aim at proving the following:
\begin{eqnarray}
 \left| \bar  R_2 (y,s)  + \frac{n(n+4)\kappa }{(p-1)s^3}\right|  &\leq & \frac{C (1 + |y|^6)}{s^4}, \text{ for all } |y| \leq  2 K_0 \sqrt s \label{estimates-bar-R-2-y-2-K-s}\\
 \| \bar R_2 (.,s)\|_{L^{\infty} (\mathbb{R}^n)} & \leq  & \frac{C}{s^2}.\label{estimates-all-sapce-bar-R-2}
\end{eqnarray}
\noindent \textit{+ The proof  of \eqref{estimates-bar-R-2-y-2-K-s}:}
We first aim at  expanding  $\Delta \Phi_2$ in  a polynomial in  $y$  of order less than $4$ via  the  Taylor expansion. Indeed,   $\Delta \Phi_2$ is given by
\begin{eqnarray*}
 \Delta \Phi_2 &=&  \frac{2 n }{s^2} \left( p-1 + \frac{(p-1)^2 |y|^2}{4p s} \right)^{-\frac{p}{p-1}} -  \frac{(p-1)|y|^2}{s^3} \left( p-1 + \frac{(p-1)^2}{4 p} \frac{|y|^2}{s}\right)^{- \frac{2p-1}{p-1}} \\
 & -& \frac{(n+ 2) (p-1)|y|^2}{2 s^3} \left( p-1 + \frac{(p-1)^2}{4 p} \frac{|y|^2}{s}\right)^{- \frac{2p-1}{p-1}}  + \frac{(2p-1)(p-1)^2 |y|^4}{4 p s^4} \left( p-1 + \frac{(p-1)^2}{4 p} \frac{|y|^2}{s}\right)^{- \frac{3p-2}{p-1}}. 
\end{eqnarray*}
Besides that,   we  make a   Taylor expansion  in the  variable  $z  =  \frac{|y|}{\sqrt s}$  for   $\left(  p-1  + \frac{(p-1)^2 }{4 p} \frac{|y|^2}{s} \right)^{- \frac{p}{p-1}}$   when   $|z| \leq 2 K$, and we get
$$  \left| \left( p-1 + \frac{(p-1)^2 |y|^2}{4p s} \right)^{-\frac{p}{p-1}}  - \frac{\kappa}{p-1}+ \frac{\kappa}{4 (p-1) } \frac{|y|^2}{s} \right|    \leq  \frac{C (1 + |y|^4)}{s^2}, \forall |y| \leq 2 K \sqrt s.$$
which yields 
$$  \left|  \frac{2 n }{s^2} \left( p-1 + \frac{(p-1)^2 |y|^2}{4p s} \right)^{-\frac{p}{p-1}} - \frac{2 n \kappa}{ (p-1)s^2}  +  \frac{n \kappa |y|^2}{2 (p-1)s^3} \right| \leq     \frac{C (1 + |y|^4)}{s^4} \leq  \frac{C (1 + |y|^6)}{s^4},  \quad  \forall |y|  \leq 2 K \sqrt{s}.$$
It is similar to   estimate   the other termes     in $\Delta \Phi_2$  as  the above. Finally, we obtain
\begin{equation}\label{Taylor-expan-Delta-Phi_2}
\left|   \Delta \Phi_2  -  \frac{2 n \kappa}{(p-1)s^2} +  \frac{n \kappa |y|^2}{ (p-1)s^3} +  2 \frac{k|y|^2}{(p-1)s^3} \right|   \leq   \frac{C ( 1 + |y|^6) }{s^4} , \forall |y| \leq 2 K \sqrt s.
\end{equation}
As we did  for $\Delta \Phi_2$,  we    estimate  similarly  the   other termes in $\bar R_2$: for all $|y| \leq 2 K \sqrt s$ 
\begin{eqnarray}
\left| - \frac{1}{2} y \cdot \nabla \Phi_2   +  \frac{\kappa |y|^2}{(p-1)s^2}  -  \frac{\kappa |y|^4}{4 (p-1) s^3} -  \frac{\kappa |y|^4}{4 (p-1)s^3} \right|  & \leq &  \frac{C( 1 + |y|^6)}{s^4}, \label{taylor-expan-1-2y-nabla-Phi-2}\\
\left| - \frac{\Phi_2}{p-1} + \frac{\kappa  |y|^2}{(p-1)^2 s^2}  -   \frac{\kappa |y|^4 }{ 4 (p-1)^2 s^3} -  \frac{2 n \kappa }{(p-1)^2 s^2} \right|  & \leq  &\frac{C( 1 + |y|^6)}{s^4}, \label{taylor-expan-1-p-1-Phi_2}\\
\left| p \Phi_1^{p-1}\Phi_2  -  \frac{p \kappa |y|^2}{ (p-1)^2 s^2}  + \frac{(2p -1) \kappa |y|^4}{ 4 (p-1)^2s^3} -  \frac{n \kappa |y|^2}{(p-1)s^3} + \frac{2 p n \kappa}{(p-1)^2 s^2} + \frac{n^2 \kappa}{(p-1)s^3}  \right|& \leq &  \frac{C( 1 + |y|^6)}{s^4},  \label{taylor-expan-1-p-Phi-1-p-1-Phi-2}\\
\left| - \partial_s \Phi_2  -  \frac{2 \kappa |y|^2}{(p-1)s^3} +  \frac{4 n \kappa }{ (p-1)s^3}  \right| &\leq &   \frac{C ( 1 + |y|^6) }{s^4}. \label{taylor-expansion-partial-s-Phi-2}
\end{eqnarray}
Thus,  we use  \eqref{Taylor-expan-Delta-Phi_2}, \eqref{taylor-expan-1-2y-nabla-Phi-2}, \eqref{taylor-expan-1-p-1-Phi_2}, \eqref{taylor-expan-1-p-Phi-1-p-1-Phi-2} and  \eqref{taylor-expansion-partial-s-Phi-2}  to deduce the following
$$ \left| \bar R_2 (y,s)  +  \frac{n(n+4) \kappa}{(p-1)s^3}  \right| \leq  \frac{C(1 + |y|^6)}{s^4},  \quad \forall  |y| \leq 2 K \sqrt s,$$
and    \eqref{estimates-bar-R-2-y-2-K-s} follows
 
\noindent \textit{+ The proof \eqref{estimates-all-sapce-bar-R-2}:}   We rewrite  $\Phi_1, \Phi_2$ as follows
 $$  \Phi_1 (y,s) = R_{1,0} (z)  + \frac{n \kappa}{2 p s}  \text{ and  } \Phi_2 (y,s)  = \frac{1}{s} R_{2,1} (z) - \frac{2n \kappa }{(p-1)s^2} \text{ where } z = \frac{y}{\sqrt s}, $$ 
 where   $R_{1,0}$ and  $R_{2,1}$ are defined  in \eqref{solu-R-0} and \eqref{solu-varphi_1},  respectively. In addition to that, we rewrite $\bar R_2$ in termes of  $R_{1,0}$ and  $R_{2,1}$, and we note that  $R_{1,0}$ and  $R_{2,1}$ satisfy \eqref{equa-R-1-0} and  \eqref{equa-R-2-1}.   Then, it follows that
 $$ |\bar R_2 (y,s)| \leq \frac{C}{s^2}, \forall y \in \mathbb{R}^n.$$ 
 Hence, \eqref{estimates-all-sapce-bar-R-2} follows. This  concludes the proof of  this Lemma.
\end{proof}
\section{ Preparation  of initial data}\label{appendix-preparation-initial-dada}
 Here,   we here    give the    proof of  Lemma  \ref{lemma-control-initial-data}. We can    see  that        part  $(II)$   directly follows from    item  $(i)$ of part  $(II)$.     The techniques  of the proof  are  given in       \cite{MZnon97}  and  \cite{TZpre15}.  Although   those papers  are written in the real-valued case, unlike ours, where we handle the complex-valued case, we reduce in fact to the real case, for  the real and the imaginary parts.   In addition to that,  the set  $\mathcal{D}_{K_0, A, s_0}$ is  the product of  two parts,  the first one  depends  only on   $d_1,$   and the other  one depends only   on  $d_2$.   Moreover,   the  real part  is  almost the  same  as   the initial data in the Vortex  model  in \cite{MZnon97},   except  for  the new term $1$, but this term  is very small  after   changing to similarity variabl: $e^{-\frac{s}{p-1}}.$   In fact, handling the imaginary part is easier than handling the real part. For those  reasons,  we kindly     refer  the reader  to  Lemma  2.4 in  \cite{MZnon97} and  Proposition 4.5 in \cite{TZpre15} for the proof   of item  $(i)$ of $(I)$ and $(II)$.  So, we  only   prove   that      the initial data satisfies item $(ii)$ in  definition of  $S(0)$ (the item  $(iii)$  is  obvious).  
 
\noindent 
 Let us  consider  $T > 0, K_0, \alpha_0, \epsilon_0,  \delta_1 $ which will be   suitably  chosen later, then we will prove that  for all  $|x| \in  \left[  \frac{K_0}{4} \sqrt{T|\ln T|} , \epsilon_0 \right]$ and $|\xi|  \leq  2 \alpha_0 \sqrt{|\ln (T - t(x))|}$ and   $\tau_0(x) = - \frac{t(x)}{ T-t(x)},$  we have
 \begin{equation}\label{want-to-pro-intial-P-2}
 \left|   U(x, \xi, \tau_0(x))   - \hat U(\tau_0(x))\right| \leq \delta_1.
 \end{equation}
 We now introduce some  neccessary notations for our proof, 
 \begin{equation}\label{notation-t-t-0-control-inital-P-2}
 \theta_0  =  T  , \quad  r(0) =  \frac{K_0}{4} \sqrt{\theta_0 |\ln(\theta_0)|}  \text{ and }  R(0) = \theta_0^{\frac{1}{2}} |\ln \theta_0|^{\frac{p}{2}}.
\end{equation}  
 Then, we have the following asymptotics:
 \begin{eqnarray}
 \theta (r(0)) \sim \theta_0 &,&    \theta \left( R(0)\right) \sim \frac{16}{K_0^2} \theta_0 |\ln \theta_0|,  \quad \theta \left( 2R(0)\right) \sim \frac{64}{K_0^2} \theta_0 |\ln \theta_0|^{p-1},\label{asymptotic-R-t-0}\\
   \ln \theta(r(0)) &\sim &  \ln \theta (R(0))   \sim \ln \theta (2 R(T)) \label{asymp-ln-r-equi-ln-R-0}. 
  \end{eqnarray}
In addition to  that, if $\alpha_0  \leq \frac{K_0}{16}$ and $\epsilon_0 \leq \frac{2}{3} C^*,$ where  $C^*$ is introduced in    \eqref{defi-proper-U-*}, then, from the definition      \eqref{def-t(x)} and  $|x| \in   \left[  r(0), \epsilon_0 \right]   ,$ and for all $|\xi| \leq 2 \alpha_0  \sqrt{ |\ln \theta (x)|}, $ with $\theta (x) =  T -t (x)$, we have 
$$ \left|\xi  \sqrt{\theta (x)}  \right| \leq \frac{1}{2} |x|,$$
which yields 
\begin{equation}\label{estima-r-t-0-C-a-1-p}
\frac{r(0)}{2} \leq \frac{|x|}{2} = |x|  - \frac{|x|}{2}  \leq | x + \xi \sqrt{ \theta (x)}| \leq  \frac{3}{2} |x| \leq \frac{3}{2} \epsilon_0 \leq   C^*.
\end{equation}  
 Hence, using  \eqref{def-mathcal-U}, \eqref{initial-data-bar-u}  and  definition of $ \chi_1$ and $|\xi | \leq  2 \alpha_0 \sqrt{\theta (x)}$  wa have 
 $$  U (x, \xi,   \tau_0)    = U_1 (x, \xi,   \tau_0)    + i U_2 (x, \xi, \tau_0) ,$$
where 
\begin{eqnarray*}
U_1 (x, \xi, \tau_0) &= &   (I)  \chi_1 ( x + \xi \sqrt{\theta (x)})  +  (II) (1  - \chi_1 (x + \xi \sqrt{ \theta (x)}))    +  (III), \\[0.2cm]
 (I)   & =&  \left(  \frac{\theta(x)}{ \theta_0}\right)^{ \frac{1}{p-1}}  \Phi_1 \left(  \frac{x + \xi \sqrt{\theta (x)}}{ \sqrt{ T}},   |\ln (T )|\right), \\[0.2cm]
 (II) & = &   \left(  \theta (x)\right)^{ \frac{1}{p-1}}  U^*\left(  x + \xi \sqrt{ \theta (x)} \right),   \\[0.2cm]
 (III) & =& \left( \theta (x)\right)^{\frac{1}{p-1}}  \\[0.2cm] 
U_2 (x, \tau, \tau_0)  &=&   \left(  \frac{\theta (x)}{ \theta_0}\right)^{\frac{1}{p-1}}  \Phi_2 \left(  \frac{x + \xi \sqrt{\theta (x)}}{ \sqrt{ T-t_0}} , |\ln (T -t_0)|\right),
\end{eqnarray*} 
The conclusion of \eqref{want-to-pro-intial-P-2}  follows from  the  4 following estimates:
\begin{eqnarray}
\left|   (I )  -  \hat U (\tau_0) \right| & \leq &    \frac{\delta_1}{4},  \text{ for all } |x| \in \left[ r(0), 2 \frac{ 100}{99} R(0) \right] \text{ and  for all } |\xi| \leq  2 \alpha_0 \sqrt{ \theta (x)}  , \label{esitma-(I)-hat-U-tau-0}\\[0.2cm]
\left|    (II)   -  \hat U (\tau_0)   \right| & \leq &      \frac{ \delta_1}{4}, \text{ for all } |x| \in \left[ \frac{99}{100} r(0) , \epsilon_0 \right]   \text{ and  for all } |\xi| \leq  2 \alpha_0 \sqrt{ \theta (x)},\label{estiam-(II)-hat-U-tau-1} \\
 |(III)| & \leq &   \frac{\delta_1}{4}, \text{  for all  } |x| \in [r(0), \epsilon_0]   \text{ and  for all } |\xi| \leq  2 \alpha_0 \sqrt{ \theta (x)}, \label{estima-(III)-hat-U-tau_0} \\
 |  U_2 (x, \xi, \tau_0) | & \leq  & \frac{\delta_1}{4}, \text{ for all }  |x| \in \left[ r(0),  2 \frac{ 100}{99} R(0) \right]  \text{ and  for all } |\xi| \leq  2 \alpha_0 \sqrt{ \theta (x)}.\label{estima-U-2-delta-1-4}
\end{eqnarray}

It is very easy to estimate for   \eqref{estima-(III)-hat-U-tau_0}    for $\epsilon_0$ small enough.

We  now estimate  \eqref{estima-U-2-delta-1-4}: We rewrite  $U_2 (x, \xi, \tau_0)$   by using  \eqref{parti-imaginary-inital} as follows: 

\begin{eqnarray*}
|U_2(x, \xi, \tau_0)| &=& U_2 \left(x, \xi,  \frac{- t(x)}{T - t(x)} \right)\\
& =&  \left(  \frac{\theta_0}{ \theta (x)}\right)^{- \frac{1}{p-1}}   \frac{| x + \xi \sqrt{ \theta (x)} |^2}{  T |\ln T|}   \left( p-1 +  \frac{| x + \xi \sqrt{ \theta (x)} |^2}{  T |\ln T|}    \right)^{-\frac{p}{p-1}}  \frac{1}{ |\ln T|}\\
& \leq & \frac{C}{ |\ln  T|}  \left(   (p-1) \frac{\theta_0}{ \theta (x)}   +    \frac{(p-1)^2}{4 p }  \frac{| x +  \xi \sqrt{x}|^2}{ \theta (x) |\ln(\theta_0)|}   \right)^{-\frac{1}{p-1}}.
\end{eqnarray*}
In addition to that, for all $|x| \in \left[ r(0), 2 \frac{100}{99}   R(0)\right]$ and  $|\xi| \leq 2 \alpha_0 \sqrt{ \theta (x)},$ we have 
$$ \frac{| x +  \xi \sqrt{x}|^2}{ \theta (x) |\ln(\theta_0)|}  \geq  \frac{1}{C K_0^2}, $$
which yields 
$$ |U_2 (x, \xi, \tau_0)| \leq   \frac{C K_0^{\frac{2}{p-1}}}{ |\ln  T|} \leq  \frac{\delta_1}{4},$$
if $T \leq  T_{1,3} (K_0, \delta_1, \alpha_0)$  and  for all $|x| \in \left[ r(0), 2 \frac{100}{99}   R(0)\right]$.

The estimate of   \eqref{esitma-(I)-hat-U-tau-0}:  We  derive from the definition of  $\Phi_1$ in \eqref{defi-Phi-1} and the  definition  of  $\hat U (\tau)$ in \eqref{ODE-hat-U-K-0}   that
\begin{eqnarray*}
\left|   (I)    -  \hat U \left( \frac{ -  t(x)}{ \theta (x)} \right)\right| &=&  \left|  \left(  (p-1) \left( \frac{\theta_0}{ \theta (x)}\right) +    \frac{(p-1)^2}{4 p}   \frac{\left| x + \xi \sqrt{ \theta (x)}\right|^2}{ \theta (x) |\ln \theta_0|} \right)^{-\frac{1}{p-1}}  \right.\\ 
& - & \left.   \left(  (p-1) \left( \frac{\theta_0}{ \theta (x)}\right) +    \frac{(p-1)^2}{4 p}  \frac{K_0^2}{16}   \right)^{-\frac{1}{p-1}}   \right|
\end{eqnarray*}
In addition to that,   from   \eqref{def-t(x)}, we have      
\begin{equation}\label{estima-|x+xi-|leq-1+2alp-}
(1  - 2 \alpha_0)^2  \frac{K_0^2}{ 16}   \frac{|\ln \theta (x)|}{ |\ln \theta_0|}  \leq     \frac{\left| x + \xi \sqrt{\theta (x)}\right|^2}{ \theta (x) |\ln \theta_0|}  \leq   (1 + 2 \alpha_0)^2  \frac{K_0^2}{ 16}  \frac{|\ln \theta (x)|}{ |\ln \theta_0|},  \forall  |\xi| \leq 2 \alpha_0 \sqrt{\theta (x)}.
\end{equation}  
Using  the  monotonicity of $\theta (x),$ we have for all  $|x| \in \left[ r(0),  2 \frac{100}{99}   R(0)\right]$   
$$   \frac{|\ln r(0)|}{ |\ln \theta_0|}  \leq    \frac{|\ln \theta (x)| }{ |\ln \theta_0|}   \leq  \frac{|\ln R(0)|}{ |\ln \theta_0|}.$$
Thanks to   \eqref{asymptotic-R-t-0}, we have
\begin{equation}\label{asym-theta-x-theta-0}
\frac{|\ln \theta (x)|}{  |\ln \theta_0|} \sim 1 \text{ as }  T \to 0.
\end{equation}
This  yields 

\begin{eqnarray*}
\left|   (I)   -  \hat U \left( \frac{ - t(x)}{ \theta (x)}  \right)\right| \leq C (K_0) \left|   \frac{|x + \xi \sqrt{\theta (x)}|^2}{  \theta (x) |\ln \theta_0|}   - \frac{K_0^2}{16}\right| \to 0  
\end{eqnarray*}
uniformly for all  $ |x| \in \left[  r(0 ),  2 \frac{100}{99}  R(0)\right], |\xi | \leq  2 \alpha_0 \sqrt{\theta (x)}$ as  $\alpha_0  \to 0 $ and $ T \to 0$. Hence, there exists $\alpha_{2,3}(K_0, \delta_1)$ and  $T_{2,3}(K_0, \delta_1)$ such that 
$$ \left|   (I)   -  \hat U \left( \frac{ - t(x)}{ \theta (x)}  \right)\right| \leq  \frac{\delta_1}{4},$$
for all  $ |x| \in \left[  r(0 ),  2 \frac{100}{99}  R(0)\right], |\xi | \leq  2 \alpha_0 \sqrt{\theta (x)}$ provided that  $\alpha_0 \leq \alpha_{2,3}$ and $T \leq T_{2,3}$. This  concludes the proof of     \eqref{esitma-(I)-hat-U-tau-0}.

The estimate of   \eqref{estiam-(II)-hat-U-tau-1}:  Let  $|x| \in \left[  \frac{99}{100} R(0), \epsilon_0   \right] .$ We use    the    definition  of  $U^*$  to rewrite  $(II)$ as follows
\begin{eqnarray*}
(II)   &=&   \left(   \frac{(p-1)^2}{ 8p }   \frac{ \left|  x + \xi \sqrt{\theta (x)}  \right|^2}{ \theta (x) |\ln (x + \xi \sqrt{\theta (x)})|}    \right)^{- \frac{1}{ p-1}}  =  \left( \frac{(p-1)^2}{8 p} \frac{\left|  \frac{K_0 }{4} \sqrt{ |\ln \theta (x)| }  +  \xi \right|^2}{   |\ln  (x + \xi \sqrt{\theta (x)})|}  \right)^{-\frac{1}{p-1}}\\
& = &  \left(        \frac{ (p-1)^2 K_0^2}{ 64 }   +   \frac{(p-1)^2}{ 8p} \left( \frac{\left|  \frac{K_0 }{4} \sqrt{ |\ln \theta (x)| }  +  \xi \right|^2}{   |\ln  (x + \xi \sqrt{\theta (x)})|}  - \frac{K_0^2}{8} \right)  \right)^{-\frac{1}{p-1}}.
\end{eqnarray*}
Then, 
\begin{eqnarray*}
\left|  (II)  - \hat U \left(   \frac{t_0  - t(x)}{ \theta (x)}\right)  \right|  &=&  \left|    \left(        \frac{ (p-1)^2 K_0^2}{ 64 }   +   \frac{(p-1)^2}{ 8p} \left( \frac{\left|  \frac{K_0 }{4} \sqrt{ |\ln \theta (x)| }  +  \xi \right|^2}{   |\ln  (x + \xi \sqrt{\theta (x)})|}  - \frac{K_0^2}{8} \right)  \right)^{-\frac{1}{p-1}} \right. \\
&  - & \left.     \left(  \frac{(p-1)^2 K_0^2}{ 64  p}   + (p-1) \frac{\theta_0}{ \theta (x)} \right)^{- \frac{1}{p-1}}  \right| \\
&\leq & C (K_0)   (  (II_1 ) +  (II_2)  ),
\end{eqnarray*}
where  
\begin{eqnarray*}
(II_1)  &=&  \left|   \frac{\left| \frac{K_0  }{4} \sqrt{ |\ln \theta (x)|}   + \xi  \right|^2}{ |\ln (x + \xi \sqrt{\theta (x)})|}   - \frac{K_0^2}{ 8}  \right| ,\\
(II_2) & = &   (p-1) \frac{\theta_0}{\theta (x)}.
 \end{eqnarray*}
 Let us give a bound to $(II_1)$: Because  $|\xi|  \leq  2 \alpha_0 \sqrt{ \ln \theta (x)},$ we have
 \begin{eqnarray*}
 |(II_1)| &\leq & \left| \frac{\left| \frac{K_0  }{4} \sqrt{ |\ln \theta (x)|}   + 2 \alpha_0 \sqrt{ \ln \theta (x)}  \right|^2}{ |\ln |x +  2 \alpha_0 \sqrt{ \theta (x)| \ln \theta (x)|} ||}   - \frac{K_0^2}{ 8}    \right|\\
 & =&  \left|  \frac{\ln \theta (x)}{ |\ln |x +   \frac{\alpha _0K_0|x|}{2}||   } \left(\frac{K_0}{4}  + 2 \alpha_0 \right)^2 -     \frac{K_0^2}{8} \right|.
\end{eqnarray*}   
Using  the fact that 
$$   \ln \theta (x)  = \ln (T-  t(x) )   \sim  2 \ln |x| ,$$
and  
$$  | \ln (|x +   2 \alpha_0 \sqrt{\theta(x) \ln \theta (x)}| )  |   =  | \ln | x  +  \frac{K_0}{2} |x|  |   |  \sim   
|\ln |x||, $$
as  $|x| \to 0,$ we derive that,  there exists  $\alpha_{3,3} (K_0, \delta_1)$ such that  for all  $\alpha_0 \leq  \alpha_{3,3},$  there exists    $\epsilon_{3,3} (K_0, \alpha_0,  \delta_1 )$ such that  for all $\epsilon_0 \leq \epsilon_{3,3},$ for all  $x \in   \left[   \frac{99}{100} R(0), \epsilon_0      \right]$  and  for all  $|\xi| \leq  2 \alpha_0 \sqrt{ |\ln \theta (x)|},$  we obtain
$$   |( II_1 )|  \leq  \frac{\delta_1}{2}. $$
It remains  to bound $(II_2).$  From   \eqref{asymptotic-R-t-0},   the fact that   $|x| \geq  \frac{99}{100} R(0)$ and the   monotonicity    of   $\theta (x)$,  we have 
$$ | (II_2) |  \leq \left|  \frac{\theta(0)}{ \theta (R(0))}  \right|  \leq   C |\ln \theta(0) |^{-(p-1)} \leq \frac{\delta_1}{2},$$
provided  that  $T \leq  T_{4,3} (K_0, \delta_1).$ This gives    \eqref{want-to-pro-intial-P-2}, and  concludes the proof  of Lemma \ref{lemma-control-initial-data}.
 \bibliographystyle{alpha}
\bibliography{mybib} 

\vspace{1cm}
\quad \quad \textbf{Address}
Paris 13  University, Institute Galil\'ee, Laboratory of Analysis, Geometry and Applications, CNRS UMR 7539, 95302, 99 avenue J.B Cl\'ement, 93430 Villetaneuse, France  

\quad \quad \texttt{e-mail: duong@math.univ-paris13.fr}

\end{document}